\documentclass[11pt,reqno]{amsart}
\usepackage[utf8]{inputenc}
\usepackage{amsmath}
\usepackage{amsthm}
\usepackage{amssymb}
\usepackage{color}
\usepackage{mathtools}
\usepackage{hyperref}
\usepackage{tikz}
\newtheorem{lem}{Lemma}
\newtheorem{thm}{Theorem}
\newtheorem*{thm*}{Theorem}

\newtheorem{prop}{Proposition}
\newtheorem{rem}{Remark}
\theoremstyle{definition}
\newtheorem{defn}{Definition}
\title{Branching polytopes for classical Lie algebras over $A_{n-1}$}
\author{Daniel Kalmbach}
\address{Chair of Algebra and Representation Theory, RWTH Aachen University, Germany}
\email{kalmbach@art.rwth-aachen.de}

\newcommand{\R}{\mathbb{R}}
\newcommand{\C}{\mathbb{C}}
\newcommand{\Z}{\mathbb{Z}}
\newcommand{\N}{\mathbb{N}}
\newcommand{\comment}[1]{}

\usepackage[left=3cm,right=2.9cm,top=3.4cm,bottom=4cm]{geometry}

\begin{document}

\maketitle

\begin{abstract}
We describe the branching of Lie algebras of classical type over $A_{n-1}$ using an inductive approach, which was motivated by the work of Gornitskii. This allows us to label the highest weight vectors of the modules occurring in the decomposition of the restriction of a finite-dimensional simple module to $A_{n-1}$ by lattice points of a string or a Lusztig polytope.
\end{abstract}

\section*{Introduction}
Let $\mathfrak{g}$ be a finite-dimensional, semisimple, complex Lie algebra of rank $n$ with Lie group $G$ and $\mathfrak{g}=\mathfrak{n}^-\oplus\mathfrak{h}\oplus\mathfrak{n}^+$ a Cartan decomposition with Cartan subalgebra $\mathfrak{h}$ and Borel subalgebra $\mathfrak{b}=\mathfrak{h}\oplus\mathfrak{n}^+$. 
Let $U^-$ and $U^+$ be the corresponding maximal unipotent subgroups of $G$ having $\mathfrak{n}^-$ and $\mathfrak{n}^+$ as Lie algebras, $\Lambda$ be the weight lattice of $\mathfrak{g}$ and $\Lambda^+\subset\Lambda$ be the set of dominant weights. Let $\mathfrak{g}'$ be a semisimple Lie subalgebra of $\mathfrak{g}$. 
For a dominant weight $\lambda\in\Lambda^+$ and a highest weight module $V(\lambda)$ we can consider the restriction $V(\lambda)|_{\mathfrak{g}'}$, which decomposes into a direct sum of irreducible $\mathfrak{g}'$ modules. 
Locating the highest weight vectors for these modules in $V(\lambda)$ is called the \textit{branching problem}. For characterizing this problem, we have to find a basis of $V(\lambda)$ containing these vectors and a nice parametrization of this basis.\par
We consider the branching problem for $\mathfrak{g}=\mathfrak{so}_{2n}$, $\mathfrak{g}=\mathfrak{so}_{2n+1}$ or $\mathfrak{g}=\mathfrak{sp}_{2n}$ and $\mathfrak{g}'=\mathfrak{sl}_n$. For describing the problem, we need a parametrization of bases of the finite dimensional irreducible representations of $\mathfrak{g}$ (i.e. highest weight modules) which is compatible with our embedding of $\mathfrak{g}'$ in $\mathfrak{g}$.

Fang, Fourier and Littelmann \cite{FaFL17} gave a unified approach for constructing monomial bases of highest weight modules of $\mathfrak{g}$: Let $N$ be the number of positive roots of $\mathfrak{g}$ and $S=(\beta_1,\dots,\beta_N)$ a sequence of positive roots (not necessarily pairwise distinct). They call this sequence a \textit{birational sequence} if the product map of the associated unipotent root subgroups 
$$\eta:U_{-\beta_1}\times\dots\times U_{-\beta_N}\rightarrow U^-$$
is birational.
For any chosen total order on $\N^N$ one obtains a monoid $\Gamma\subseteq\Lambda^+\times\Z^N$, whose projection to $\Lambda^ +$ yields as a fiber a monomial basis of $V(\lambda)$ for each $\lambda\in \Lambda^+$. If the monoid $\Gamma$ is finitely generated and saturated, it is the set of lattice points of some polyhedral cone in $\Lambda_{\R}\times \R^N$, the fiber over each $\lambda \in \Lambda^+$ is a convex polytope.

\par
An example for a birational sequence is the \textit{string case}: Let $\underline{\omega_0}=s_{i_1}\dots s_{i_N}$ be a reduced decomposition of the longest element in the Weyl group $W$ of $\mathfrak{g}$. 
Then we call $\mathbf{i}=(i_1,\dots,i_N)$ a \textit{reduced word} for $\omega_0$ and define the sequence $S$ just as $S=(\alpha_{i_1},\dots,\alpha_{i_N})$, where $\alpha_1,\dots,\alpha_n$ are the simple roots of $\mathfrak{g}$.
This parametrization was defined for a fixed monomial order by Berenstein and Zelevinsky \cite{BZ93}, \cite{BZ96}, \cite{BZ01} and further described by Littelmann \cite{Li98}. Berenstein and Zelevinsky gave an explicit description of the string cone and polytope for type $A_n$ and a special reduced word for $\omega_0$ and as well an algorithm for an arbitrary reduced word.
Littelmann gave a description for a whole class of reduced words in every type, calling these words \textit{nice}.
Moreover, he gave a formula to calculate the additional inequalities for the polytopes for an arbitrary reduced word. 
\par
Another example is the \textit{Poincaré-Birkhoff-Witt (PBW)-type case}, where $S$ is just an enumeration on the set of all positive roots.
A special case of PBW bases is the Lusztig case, where a convex order on the set of all positive roots is used: Fix again a reduced word $\mathbf{i}$ for $\omega_0$. We define the sequence $S$ by $\beta_{\mathbf{i},k}=s_{i_1}\dots s_{i_{k-1}}(\alpha_{i_k})$. 
This parametrization is up to a filtration compatible with Lusztig's \textit{canonical basis} defined in \cite{Lu90}, \cite{Lu90II}, \cite{Lu93} . The canonical basis was later proved to coincide with Kashiwara's \cite{Ka91}, \cite{Ka93} \textit{global basis}.\par
For describing branching problems, the order on the positive roots should be compatible with the chosen embedding of $\mathfrak{g}'$ in $\mathfrak{g}$. 
Molev and Yakimova described the branching of $C_n$ over $C_{n-1}$ in \cite{MY18} using the \textit{Feigin-Fourier-Littelmann-Vinberg (FFLV) method} \cite{FFL11}, \cite{FFL112}, \cite{FFL17}, \cite{Vi05} for constructing monomial bases for the highest weight modules.
Gornitskii \cite{Go19} described a branching of $B_n$ and $D_{n+1}$ over $D_{n}$, using a filtration defined by a non-homogeneous order on $\Z^N$ and an order on the set of positive roots which is not of Lusztig type, but allows an inductive construction of bases embedding $D_n$ in $B_n$ and $D_{n+1}$.
Due to the observation, that in the $A_n$ case this approach gives bases of Lusztig type, we are aiming for an order of Lusztig type on the set of positive roots of $B_n$ and $D_n$ which also allows us to use a restriction to some smaller dimensional Lie algebra.
It turns out, that this is possible with a projection of $B_n$, $D_n$ and also $C_n$ to $A_{n-1}$. As the results in all three cases are very similar, we will focus on the $D_n$ case in this introduction. The results for $B_n$ can be found in Theorem \ref{thmBZfulfillsineqB}, Theorem \ref{StringConeB} and Theorem \ref{thmbranchingD} and for $C_n$ in Theorem \ref{thmBZfulfillsineqC}, Theorem \ref{StringConeC} and Theorem \ref{thmbranchingD}.\\\par

The reduced words we are using are not nice in the sense of Littelmann, but our inductive approach allows us to calculate all the cone inequalities explicitly, using a recursive description of the string cone by Littelmann and a polyhedral cone described by Berenstein and Zelevinsky, which contains the string cone (\cite{BZ01}, Theorem 3.14).
In the $A_n$ case, the two cones coincide for any $\omega_0$. This yields the suggestion, that the cones are also closely related in other cases and we will prove that they in fact are, at least under a special condition on the choice of a reduced word for $\omega_0$.

For $B_2,\;C_2$ and $D_3$ the cones again coincide with the string cones and starting from this we can use induction on $n$ to find the inequalities defining the string cones for $B_n,\;C_n$ and $D_n$.\\\par
To be precise, corresponding to our embedding of $\mathfrak{g}'$ in $\mathfrak{g}$ we choose a reduced decomposition of $\omega_0$ such that its first part is a reduced decomposition for the longest element for the Weyl group of type $A_{n-1}$, so we can use a projection to $A_{n-1}$ using another result from \cite{BZ01}, which tells us that the string cone is the direct product of two cones for our reduced decomposition. 
For the reduced decomposition given below, the first of these cones is just the string cone for $A_{n-1}$, where we know all the inequalities. The second one looks like a copy of the first with some additional facets.
In fact, the results for the branching cone and polytope mentioned below are independent of the choice of a reduced decomposition of $\omega_0$ as long as it respects our embedding as mentioned above.\par
Our favorite reduced word is  $\mathbf{i}^{D_n}=\mathbf{i}^{A_{n-1}}\mathbf{i}^{\tilde{D_n}}$ with 
$$
\begin{array}{lll}
&\mathbf{i}^{A_{n-1}}=(n-1,n-2,n-1,n-3,n-2,n-1,\dots,1,2,\dots,n-1),\\
&\mathbf{i}^{\tilde{D_n}}=(n,n-2,n-1,n-3,n-2,n,\dots,1 ,2,\dots,n-2,n/n-1).
\end{array}$$
The last entry is $n$ for $n$ even and $n-1$ for $n$ odd. We illustrate the pattern with two examples:
$$\mathbf{i}^{\tilde{D_4}}=(4,2,3,1,2,4),\mathbf{i}^{\tilde{D_5}}=(5,3,4,2,3,5,1,2,3,4).$$\par

The first main theorem of this paper describes the string cone of this reduced word:
\begin{thm*}
For $\mathfrak{g}=\mathfrak{so}_{2n}$ the string cone $\mathcal{S}_{\mathbf{i}^{D_n}}$ is given by the set of all points $$(t_{1,1}^-,t_{1,2}^-,t_{2,2}^-,\dots,t_{1,n-1}^-,\dots,t_{n-1,n-1}^-,t_{1,1}^+,t_{1,2}^+,t_{2,2}^+,\dots,t_{1,n-1}^+,\dots,t_{n-1,n-1}^+)\in\Z^N_{\geq 0}$$ satisfying the inequalities 
$$
t_{i,j}^-\geq t_{i+1,j}^-,\quad t_{i,j}^+\geq t_{i+1,j}^+, \;\forall\; 1\leq i<j \leq n-1
$$
and
$$
t_{i,j}^+\geq t_{i,j+1}^+, \;\forall\; 1\leq i\leq j\leq n-2.
$$
\end{thm*}
Projecting the string cone to the $t^+$ variables, we obtain the \textit{string branching cone}. It is of surprisingly easy structure, for example, for the branching of $D_6$ over $A_5$ it is the order polyhedron of the following poset:\\
\scalebox{0.7}
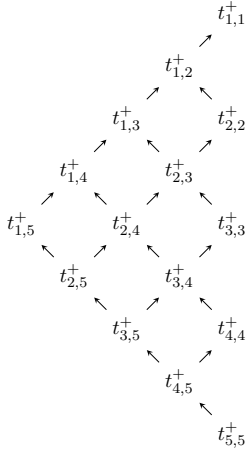
{
\begin{tikzpicture}
  \tikzstyle{arrow} = [->,>=stealth]

\draw (11,0)node(11+){$t_{1,1}^+$};
\draw (10.0,-1.0)node(12+){$t_{1,2}^+$};
\draw (9.0,-2.0)node(13+){$t_{1,3}^+$};\draw (11.0,-2.0)node(22+){$t_{2,2}^+$};
\draw (8.0,-3.0)node(14+){$t_{1,4}^+$};\draw (10.0,-3.0)node(23+){$t_{2,3}^+$};
\draw (7.0,-4.0)node(15+){$t_{1,5}^+$};\draw (9.0,-4.0)node(24+){$t_{2,4}^+$};\draw (11.0,-4.0)node(33+){$t_{3,3}^+$};
\draw (8.0,-5.0)node(25+){$t^+_{2,5}$};\draw (10.0,-5.0)node(34+){$t^+_{3,4}$};\draw (9.0,-6.0)node(35+){$t^+_{3,5}$};\draw (11.0,-6.0)node(44+){$t^+_{4,4}$};\draw (10.0,-7.0)node(45+){$t^+_{4,5}$};\draw (11.0,-8.0)node(55+){$t^+_{5,5}$};
\draw[arrow](12+)--(11+);
\draw[arrow](13+)--(12+);\draw[arrow](22+)--(12+);
\draw[arrow](14+)--(13+);\draw[arrow](23+)--(13+);\draw[arrow](23+)--(22+);
\draw[arrow](15+)--(14+);\draw[arrow](24+)--(14+);\draw[arrow](24+)--(23+);\draw[arrow](33+)--(23+);
\draw[arrow](25+)--(15+);\draw[arrow](25+)--(24+);\draw[arrow](34+)--(24+);\draw[arrow](34+)--(33+);
\draw[arrow](35+)--(25+);\draw[arrow](35+)--(34+);\draw[arrow](44+)--(34+);\draw[arrow](55+)--(45+);
\draw[arrow](45+)--(35+);\draw[arrow](45+)--(44+);

\end{tikzpicture}
}

In order to locate the highest weight vectors in the branching, we next focus on the polytope. The additional inequalities for the string polytope can be directly read of from Littelmann's work \cite{Li98}. We do not get a direct product of two polytopes here because the additional inequalities for the $t^-$ variables depend on the $t^+$ variables.
However, considering the projection on the $t^+$ variables, we obtain a polytope which we call the \textit{string branching polytope} $\mathcal{S}_{\mathbf{i}^{\tilde{D}_n}}(\lambda)$. It consists of $\mathfrak{sl}_n$ highest weight vectors, which means we get a decomposition
\begin{equation}
\displaystyle{\mathcal{S}^L_{\mathbf{i}^{D_n}}(\lambda)=\Dot{\bigcup}_{t\in \mathcal{S}^L_{\mathbf{i}^{\tilde{D}_n}}(\lambda)}\mathcal{S}^L_{\mathbf{i}^{A_{n-1}}}(\lambda-t\cdot\alpha^T_{\tilde{D}})\times \{t\}.}\label{branchingD}    
\end{equation}

Here, $\mathcal{S}^L$ is the set of lattice points in $\mathcal{S}$ and $\alpha^T_{\tilde{D}}=(\alpha_{i_{\frac{N}{2}+1}},\dots,\alpha_{i_N})$. So, for a $D_n$ highest weight module $V(\lambda)$ we can calculate the multiplicity of an $A_{n-1}$ module of highest weight $\mu$ in $V(\lambda)|_{\mathfrak{sl}_n}$ by counting the vectors $t$ in $\mathcal{S}_{\mathbf{i}^{\tilde{D}_n}}(\lambda)$ with $t\cdot\alpha^T_{\tilde{D}}=\lambda-\mu$. \par
Let $\lambda=\lambda_1\omega_1+\dots+\lambda_n\omega_n$ where $\omega_i$ are the $D_n$ fundamental weights.
Morier-Genoud gave in \cite{MG08} linear bijective maps between the string and the Lusztig polytope, which allow us to transfer our results to Lusztig's parametrization and obtain the \textit{Lusztig branching polytope}:
\begin{thm*}
The Lusztig branching polytope $\mathcal{L}_{\mathbf{i}^{\tilde{D}_n}}(\lambda)$ is given by the set of all points $$(u_{1,1}^+,u_{1,2}^+,u_{2,2}^+,\dots,u_{1,n-1}^+,\dots,u_{n-1,n-1}^+)\in\Z^{\frac{N}{2}}_{\geq 0}$$ satisfying the inequalities
\[\sum\limits_{k=j}^{n-1}u_{i,k}^+\leq \sum\limits_{k=j+1}^{n-1}u_{i+1,k}^+ +\lambda_i, \quad u_{n-1,n-1}^+\leq\lambda_n,\;\forall \; 1\leq i<j\leq n-1,
\]
\[
      \sum\limits_{k=i}^{j}u_{k,j}^+ +\sum\limits_{k=j+2}^{n-1}u_{j+1,k}^+\leq \sum\limits_{k=i+1}^{j}u_{k,j+1}^+ +\sum\limits_{k=j+2}^{n-1}u_{j+2,k}^+ +\lambda_{j+1},\;\forall \; 1\leq i\leq j< n-1.
\]    

\end{thm*}
In both the string and the Lusztig parametrization, the similarity to the $A_{n-1}$ case is remarkable. There, for $\mathbf{i}^{A_{n-1}}$, the Lusztig polytope is the image of the projection of the polytope defined in Theorem \ref{thmLP} (1) onto the $u^-$ variables.\\\par

It is worth mentioning that the embedding of $A_{n-1}$ in $C_n$ and $D_n$ we use is also used to construct symplectic ball embeddings for studying the coadjoint orbit and the Gromov width by Fang, Littelmann and Pabiniak \cite{FLP17}. See Theorem 7.5 and Example 7.7 there for more details.\par
In upcoming papers we will study the branching algebras encoding the branching rules describing the multiplicities in the decomposition (\ref{branchingD}). Moreover, we will study the generators of the monoid $\Gamma$.
In the $A_n$ case, using the ordering associated to $\mathbf{i}^{A_n}$, this monoid is generated by all tuples $(\omega_i,t)$, $t\in \mathcal{S}_{\mathbf{i}^{A_n}}(\omega_i)$. This is not true for the bases we construct for $B_n,\;C_n$ and $D_n$ using the orderings associated to $\mathbf{i}^{B_n}$, $\mathbf{i}^{C_n}$ and $\mathbf{i}^{D_n}$, so we want to find a finite set of weights, such that the corresponding tuples generate the monoid.\\\par

The paper is organized as follows: 
We start by recalling the definition of the string cone and some useful results about it. Next, we recall the definition of the Lusztig cone from \cite{Lu93} and the connection to the string cone via the maps described in \cite{MG08}.\par
In section~\ref{section3} we will bring together the results from Berenstein-Zelevinsky and Littelmann to explicitly give the inequalities for the string cone for the case $\mathfrak{g}=\mathfrak{so}_{2n}$ and a special choice of the reduced decomposition of $\omega_0$. In section \ref{section4} and \ref{section5} we calculate the inequalities for $\mathfrak{g}=\mathfrak{so}_{2n+1}$ and $\mathfrak{g}=\mathfrak{sp}_{2n}$. The only difference in the latter case is that we get a factor $2$ in some of the inequalities here.\par
In section \ref{section6} we take a closer look at the branching, describing the \textit{purely} $D_n$ respectively $B_n$ or $C_n$ part of the string polytope and then transfer our result to the Lusztig polytope using the linear maps from \cite{MG08}.
\subsection*{Acknowledgments} The author would like to thank Xin Fang and Ghislain Fourier for many inspiring discussions and a lot of helpful comments during the process of writing this paper. Furthermore, he would like  to thank Oksana Yakimova for explaining some ideas of her branching approach and Christian Steinert and George Balla for pointing out some misprints and incomprehensibility.
%%%%%%%%%%%%%%%%%%%%%%%%%%%%%%%%%

\section{String and Lusztig cones and polytopes}\label{section2}
We now want to collect the most important results about string and Lusztig cones, which we will later need to describe one special string cone for each $B_n$, $C_n$ and $D_n$ explicitly. Moreover, we will introduce what we call the \textit{branching polytope} to motivate the upcoming chapters.
\subsection{Preliminaries}
We first fix our general setup. Let $\mathfrak{g}$ be a finite-dimensional, semisimple, complex Lie algebra with a semisimple Lie subalgebra $\mathfrak{g}'$ and $\mathfrak{g}=\mathfrak{n}^-\oplus\mathfrak{h}\oplus\mathfrak{n}^+$ a Cartan decomposition with Cartan subalgebra $\mathfrak{h}$ and Borel subalgebra $\mathfrak{b}=\mathfrak{h}\oplus\mathfrak{n}^+$. 
Let $\alpha_1,\dots,\alpha_n$ be the simple roots of $\mathfrak{g}$ and $W$ be the Weyl group generated by $s_1,\dots,s_n$. Let $\omega_0$ be the element of maximal length in $W$ and denote by $R(\omega_0)$ the set of reduced decompositions of $\omega_0$. Let $(a_{i,j})$ be the Cartan matrix  and $w_1,\dots,w_n$ be the fundamental weights of $\mathfrak{g}$. \par
Let $\mathcal{U}^+$ be the universal enveloping algebra of $\mathfrak{n}^+$ generated by $E_1,\dots,E_n$. We define a grading on $\mathcal{U}^+$ using the root lattice of $\mathfrak{g}$ by
$$\deg(E_i)=\alpha_i.$$

\comment{
We now want to define the quantum enveloping algebra of $\mathfrak{g}$. Therefore, we fix positive integers $d_1,\dots ,d_n$ such that $d_ia_{i,j}=d_ja_{j,i}$.
\begin{defn}
For any Lie algebra $\mathfrak{g}$ we define the quantized enveloping algebra $U=U_q(\mathfrak{g})$ as the $\C(q)$-algebra with unit generated by the elements $E_i$, $K_i^{\pm 1},$ and $F_i$ for $i=1,\dots,n$ subject to the relations
\begin{equation}\notag
\begin{aligned}
K_iK_j&=K_jK_i,\\
K_iE_jK_i^{-1}&=q^{d_ia_{i,j}}E_j,\\
K_iF_jK_i^{-1}&=q^{-d_ia_{i,j}}E_j,\\
E_iF_j-F_jE_i&=\delta_{i,j}\frac{K_i-K_i^{-1}}{q^{d_i}-q^{-d_i}}
\end{aligned}    
\end{equation}
for all $i$ and $j$ and the quantum Serre relations
\begin{equation*}
    \sum_{k+l=1-a_{i,j}}(-1)^k E_i^{(k)}E_jE_i^{(l)}=\sum_{k+l=1-a_{i,j}}(-1)^k F_i^{(k)}F_j F_i^{(l)}=0
\end{equation*}
for $i\neq j$. Here $E_i^{(k)}$ and $F_i^{(k)}$ stand for the divided powers defined by
$$E_i^{(k)}=\frac{1}{[1]_i[2]_i\dots[k]_i}E_i^k,\;F_i^{(k)}=\frac{1}{[1]_i[2]_i\dots[k]_i}F_i^k,$$
where 
$$[Li]_i=\frac{q^{d_il}-q^{-d_il}}{q^{d_i}-q^{-d_i}}.$$
We define a grading on $u$ using the root lattice of $\mathfrak{g}$ by
$$\deg(K_i)=0,\;\deg(E_i)=-\deg(F_i)=\alpha_i.$$

\end{defn}
Let $\mathcal{U}^+$ denote the subalgebra of $U$ generated by $E_1,\dots,E_n$.
}
Lusztig defined a basis of $\mathcal{U}^+$ which is denoted by $\mathfrak{B}$ and is called the \textit{canonical basis}. See \cite{Lu93} for the definition of this basis.

\subsection{Definition of string cones}

We now recall the string parametrization of the basis dual to $\mathfrak{B}$, which was first introduced in \cite{BZ96},\cite{GP00}.  Therefore, let $V$ be a $\mathcal{U}^+$-module such that each $E_i$ acts on $V$ as a locally nilpotent operator. This means, for any $v\in V$ there exists a positive integer $m$ such that $E_i^m(v)=0$. We denote by $c_i(v)$ the maximal integer $m$ such that $E_i^m(v)\neq 0$.
\begin{defn}\label{defString}
For any $v\in V$ and $\mathbf{i}\in R(\omega_0)$ we define the string of $v$ in direction $\mathbf{i}$ as $c_{\mathbf{i}}(v)=(t_1,\dots,t_N)$, where $N$ is the length of $\mathbf{i}$ and
$$t_1=c_{i_1}(v),t_2=c_{i_2}(E_{i_1}^{t_1}(v)),\dots,t_N=c_{i_N}(E_{i_{N-1}}^{t_{N-1}}\dots E_{i_1}^{t_1}(v)).$$
\end{defn}
Now we define an $\mathcal{U}^+$-module $\mathcal{A}$ as follows. As a $\C$-vector space, $\mathcal{A}$ is the restricted dual vector space of $\mathcal{U}^+$, i.e. the direct sum of dual spaces of all homogeneous (with respect to the degree defined above) components of $\mathcal{U}^+$. 
The action of $\mathcal{U}^+$ on $\mathcal{A}$ is given by $E(f)(u)=f(E^{\iota}u)$, where $\iota$ is the involutive $\C$-algebra antiautomorphism $E\rightarrow E^{\iota}$  of $\mathcal{U}^+$ such that $E_i^{\iota}=E_i$ for all $i$. As each $E_i$ acts in $\mathcal{A}$ as a locally nilpotent operator, we can apply our definition of strings to elements of $\mathcal{A}.$ \par
Now let $\mathfrak{B}^{\mathrm{dual}}$ be the basis of $\mathcal{A}$ dual to the canonical basis $\mathfrak{B}$, called the dual canonical basis. From \cite{BZ01} we get the following result.
\begin{prop}
For any $\mathbf{i}\in R(\omega_0)$, the string parametrization $c_\mathbf{i}$ is a bijection of $\mathfrak{B}^{\mathrm{dual}}$ onto the set of all lattice points $\mathcal{S}_{\mathbf{i}}^L$ of some rational polyhedral convex
cone $\mathcal{S}_{\mathbf{i}}$ in $\R^N$.
\end{prop}
This cone is refered to as the \textit{string cone}.

\subsection{The Berenstein-Zelevinsky-Cone}
In \cite{BZ01} Bereinstein and Zelevinsky describe for arbitrary $\mathfrak{g}$ a cone which contains the string cone and in the case $\mathfrak{g}=\mathfrak{sl}_n$ coincides with it. To recall their results, we need some more notation: 
We define for each $i$ in $\{1,\dots,n\}$ $W_{\hat{i}}$ to be the minimal parabolic subgroup in $W$ generated by all $s_j$ with $j\neq i$ and denote by $z^{(i)}$ the minimal representative of the coset $W_{\hat{i}}s_i\omega_0$ in $W$.

\begin{thm}\label{BZcone}
Let $\mathbf{i}= (i_1,...,i_N)\in R(\omega_0)$. For any $i\in\{1,\dots,n\}$ and any subword $(i_{j_1},\dots,i_{j_r})$ of $\mathbf{i}$ which is a reduced word for $z^{(i)}$, all the points $(t_1,\dots,t_N)$ in the string cone $\mathcal{S}_{\mathbf{i}}$ satisfy the inequality
\begin{equation}
    \sum_{p=0}^r\sum_{j_p<k<j_{p+1}}(s_{i_{j_1}}\dots s_{i_{j_p}}\alpha_{i_k})(w_i^{\vee})\cdot t_k\geq 0
\end{equation}
(with the convention that $j_0 = 0$ and $j_{r+1} =N+1$). 
\end{thm}
We now recall another theorem from \cite{BZ01}. Therefore, for any subset $I\subseteq \{1,\dots,n\}$ we define $\omega_0(I)$ as the longest element of the parabolic subgroup in $W$ generated by all the permutations $s_i$ with $i\in I$.
\begin{thm}\label{coneproduct}
Let $\emptyset\subset I_0\subset I_1\subset\dots\subset I_p=\{1,\dots,n\}$ be any flag of subsets in $\{1,\dots,n\}$. Suppose $\mathbf{i}\in R(\omega_0)$ is the concatenation $(\mathbf{i}^{(1)},\dots,\mathbf{i}^{(p)})$ where $\mathbf{i}^{(j)}\in R(\omega_0(I_{j-1})^{-1}\omega_0(I_j))$ for $j=1,\dots,p$. Then the string cone $\mathcal{S}_{\mathbf{i}}$ is the direct product of cones:
\begin{equation*}
    \mathcal{S}_{\mathbf{i}}=\mathcal{S}_{\mathbf{i}^{(1)}}(e,\omega_0(I_1))\times \mathcal{S}_{\mathbf{i}^{(2)}}(\omega_0(I_1),\omega_0(I_2))\times\dots\times \mathcal{S}_{\mathbf{i}^{(p)}}(\omega_0(I_{p-1}),\omega_0(I_p)).
\end{equation*}
\end{thm}
We will not need to know the exact definition of each of these cones for our purposes. We just note that the dimension of $\mathcal{S}_{\mathbf{i}^{(k)}}(\omega_0(I_{k-1}),\omega_0(I_k))$ equals the length of $\mathbf{i}^{(k)}$.

\subsection{Littelmann's description of the string cone}
In \cite{Li98}, Littelmann gives a characterization of all lattice points of the string cone by recursive formulas. To give this characterization, we need the following notation.\par
For $t\in\Z^N$ let us define a sequence $m^N,\dots,m^1$ with $m^i\in\Z^i$: We set $m^N=t$ and define recursively $m^{j-1}=(m^{j-1}_1,\dots,m^{j-1}_{j-1})$ for $1\leq j<N$ by
$$m^{j-1}_k:=\min \{m^j_k,\Delta^j(k)\},$$
where
$$\Delta^j(k)=\begin{dcases}\max\{\theta(k,l,j)\mid k<l\leq j,\alpha_{i_l}=\alpha_{i_j}\}&\mathrm{if }\;\alpha_{i_k}=\alpha_{i_j},\\
m^j_k&\mathrm{otherwise},
\end{dcases}$$
with
$$\theta(k,l,j)=m_l^j-\sum_{k<s\leq l} m^j_s\alpha_{i_s}(\alpha^{\vee}_{i_j} ).$$
With this notation, we are now able to recall the following theorem from \cite{Li98}:
\begin{thm}\label{thmrecl}
Let $\mathbf{i}\in R(\omega_0)$, then $t\in \mathcal{S}^L_\mathbf{i}$ if and only if $$\Delta^j(k)\geq 0\;\forall\; 1\leq k< j\leq N.$$
\end{thm}
\begin{rem}
Due to Theorem 1.5 from \cite{Li98}, the string cone is a rational cone, which implies that Theorem \ref{thmrecl} can be generalized from the set of lattice points in the string cone to the set of all points in the cone.
\end{rem}

From Littelmann's restriction rule \cite{Li95} we obtain the following Lemma:
\begin{lem} \label{lemres} Let $\mathfrak{l}$ be a Levi subalgebra of $\mathfrak{g}$, $\mathbf{i}_{\mathfrak{g}}=(i_1,\dots,i_N)\in R(\omega_0^{\mathfrak{g}})$ and $\mathbf{i}_{\mathfrak{l}}=(i_{j_1},\dots,i_{j_r})\in R(\omega_0^{\mathfrak{l}})$ be a subword of $\mathbf{i}_{\mathfrak{g}}$. 
Let $\lambda$ be a dominant weight for $\mathfrak{g}$, $v\in V_{\mathfrak{g}}(\lambda)$ be a basis element and $t_v\in\mathcal{S}_{\mathbf{i}_{\mathfrak{g}}}(\lambda)$ the string of $v$ in direction $\mathbf{i}_{\mathfrak{g}}$.
If we consider $V_{\mathfrak{g}}(\lambda)$ as a $\mathfrak{l}$-module, we have $v\in V_{\mathfrak{l}}(\mu)$ for some dominant weight $\mu$ of $\mathfrak{l}$.
Moreover, if for each $1\leq k\leq j_r$ we have $k\in\{j_1,\dots,j_r\}$ or $t_k=0$, then $(t_{i_{j_1}},\dots,t_{i_{j_r}})$ is the string of $v$ in direction $\mathbf{i}_{\mathfrak{l}}$.
\end{lem}
\begin{proof}
Due to Littelmann's restriction rule we know that -- considered as $\mathfrak{l}$-module -- $V_{\mathfrak{g}}(\lambda)$ decomposes into the direct sum of some highest weight $\mathfrak{l}$-modules. As $v$ is a basis element, it already has to be contained in one of these modules.\\
Due to Kashiwara \cite{Ka96} and Joseph \cite{Jo94}, the graph for Littelmann's path model is isomorphic to Kashiwara's crystal graph, so we may use the string language here. 
If for each $1\leq k\leq j_r$ we have $k\in\{j_1,\dots,j_r\}$ or $t_k=0$ , we get $c_{\mathbf{i}_{\mathfrak{g}}}(v)=c_{\mathbf{i}_{\mathfrak{l}}}(v)$ by the definition of the string parametrization, which proves the second statement of the Lemma.
\end{proof}
\subsection{Lusztig cones}
Whereas the string parametrization of $\mathcal{B}$ just works with simple roots, the Lusztig parametrization works with the set of all positive roots of $\mathfrak{g}$. We therefore fix an ordering on the set of positive roots which depends on $\mathbf{i}$. We set $\beta_{\mathbf{i},k}=s_{i_1}\dots s_{i_{k-1}}(\alpha_{i_k})$, which defines the ordering $\beta_{\mathbf{i},1}<\dots<\beta_{\mathbf{i},N}.$ Let $F_1,\dots,F_n$ be the generators of $\mathcal{U}^-$, the enveloping algebra of the subalgebra $\mathfrak{n}^-.$ We denote by $F_i^{(r)}$ the \textit{divided power} of $F_i$, which just differs from $F_i^r$ by a multiple that is not of interest for our purposes.
\begin{defn}
Let $\mathbf{i}\in R(\omega_0)$ and $v$ be a cyclic generator of $\mathfrak{n}^-$. For any $u\in\R^N_{\geq 0}$ we set $F_{\mathbf{i}}(u)=F_{\beta_{\mathbf{i},1}}^{(u_1)}\dots F_{\beta_{\mathbf{i},N}}^{(u_N)}v$.
\end{defn}
The set $\{F_{\mathbf{i}}(u)\mid u\in\Z_{\geq 0}^N\}$ is a PBW-type basis of $\mathcal{U}^-$. Now, Lusztig \cite{Lu96} associates to each $F_{\mathbf{i}}(u)$ an element $b_{\mathbf{i}}(u)$ in the canonical basis $\mathfrak{B}$.
We call the map $u\rightarrow b_{\mathbf{i}}(u)$ the Lusztig parametrization of $\mathcal{B}.$ The cone $\R^N_{\geq 0}$ is called the Lusztig cone. We denote it by $\mathcal{L}_{\mathbf{i}}$.
%%%%%%%%%%%%%%%%%%%%%%%%%%%%%%%%%%%%%%%%%%%%%%%%%%%%%%%%%%%%%%%%%%%%%%%%%%%%%%%%%%%%%%%%%%%%%%%%%%
%%%%%%%%%%%%%%%%%%%%%%%%%%%%%%%%%%%%%%%%%%%%%%%%%%%%%%%%%%%%%%%%%%%%%%%%%%%%%%%%%%%%%%%%%%%%%%%%%%
\subsection{String and Lusztig polytopes and their connection}
Given a highest weight module $V(\lambda)$, we are interested in polytopes $\mathcal{S}_{\mathbf{i}}(\lambda)\subset \mathcal{S}_{\mathbf{i}}$ and $\mathcal{L}_{\mathbf{i}}(\lambda)\subset \mathcal{L}_{\mathbf{i}}$ whose lattice points label those basis elements in $\mathcal{B}$ which are contained in $V(\lambda)$. We denote them as \textit{string} and \textit{Lusztig polytopes}.
From \cite{MG08} we know the transformation maps between the two parametrizations:
\begin{thm}\label{StringLusztig}
The transformation map $\varphi:\mathcal{S}_{\mathbf{i}}(\lambda)\rightarrow \mathcal{L}_{\mathbf{i}}(\lambda)$ is given by
$$\varphi(t)_k=\lambda_k-t_k-\sum_{j>k}a_{i_ki_j}t_j.$$
\end{thm}
\begin{thm}\label{LusztigString}
The transformation map $\psi:\mathcal{L}_{\mathbf{i}}(\lambda)\rightarrow \mathcal{S}_{\mathbf{i}}(\lambda)$ is given by
$$\psi(u)_k=l_k-u_k-\sum_{j>k}d_{i_ki_j}u_j,$$
where $d_{i_ki_j}=\langle\beta_{\mathbf{i},j},\beta^{\vee}_{\mathbf{i},k}\rangle$ and $l_k=\langle\lambda,\beta^{\vee}_{\mathbf{i},k}\rangle$.
\end{thm}
From \cite{Li98} we know how to compute the weight inequalities for the string polytope: 
\begin{prop}\label{Proppoly}
The string polytope $\mathcal{S}_{\mathbf{i}}(\lambda)\subset \mathcal{S}_{\mathbf{i}}$ is the polytope defined by 
\begin{equation}t_N\leq\langle\lambda,\alpha_{i_N}^{\vee}\rangle,t_{N-1}\leq\langle\lambda-t_N\alpha_{i_N},\alpha_{i_{N-1}}^{\vee}\rangle,\dots, t_{1}\leq\langle\lambda-t_N\alpha_{i_N}-\dots-t_2\alpha_{i_2},\alpha_{i_{1}}^{\vee}\rangle.\label{ineqpoly}
\end{equation}
\end{prop}
\subsection{Branching polytopes}\label{subsecbp}
\comment{For the Lusztig parametrization, using the ordering $\beta_{\mathbf{i},1}<\dots<\beta_{\mathbf{i},N}$ means, we use the left opposite lexicographic order on $\Z^N_{\geq 0}$. 
For the string parametrization we use the same order due to Definition \ref{defString}. }
Due to Theorem \ref{coneproduct}, if we find $I_0\subset I_1=\{1,\dots,n\}$ such that $\omega_0(I_1)$ is the longest element of the Weyl group of $\mathfrak{g}'$ and $\mathbf{i}^{\mathfrak{g}}=\mathbf{i}^{\mathfrak{g}'}\mathbf{i}^{\tilde{\mathfrak{g}}}\in R(\omega_0)$ such that $\mathbf{i}^{\mathfrak{g}'}\in R(\omega_0(I_0))$, we can write the string cone $\mathcal{S}_{\mathbf{i}^{\mathfrak{g}}}$ as the direct product of two cones $\mathcal{S}_{\mathbf{i}^{\mathfrak{g}'}}\times \mathcal{S}_{\mathbf{i}^{\tilde{\mathfrak{g}}}}$, where the first one is the string cone for $\mathfrak{g}'$.
Note that the notation for the second cone should not suggest the existence of a Lie algebra $\tilde{\mathfrak{g}}$ but just refers to the remaining part of the string cone for $\mathfrak{g}.$
From Proposition \ref{Proppoly} we see that the image of the restriction of the projection $\pi_{\tilde{\mathfrak{g}}}^{\mathfrak{g}}:\mathcal{S}_{\mathbf{i}^{\mathfrak{g}}}\rightarrow \mathcal{S}_{\mathbf{i}^{\tilde{\mathfrak{g}}}}$ to a polytope $\mathcal{S}_{\mathbf{i}^{\mathfrak{g}}}(\lambda)$ is a polytope $\mathcal{S}_{\mathbf{i}^{\tilde{\mathfrak{g}}}}(\lambda)$ as none of the inequalities for the $\tilde{\mathfrak{g}}$ variables depend on the $\mathfrak{g}'$ variables.\par
We call the polytope $\mathcal{S}_{\mathbf{i}^{\tilde{\mathfrak{g}}}}(\lambda)$ the \textit{string branching polytope}. 
Now, for any point $t$ in the string branching polytope we can define a polytope in the cone $\mathcal{S}_{\mathbf{i}^{\mathfrak{g}'}}$ as $\pi^{\mathfrak{g}}_{\mathfrak{g}'}((\pi^{\mathfrak{g}}_{\tilde{\mathfrak{g}}}|_{\mathcal{S}_{\mathbf{i}^{\mathfrak{g}}(\lambda)}})^{-1}(t))$ where $\pi_{\mathfrak{g}'}^{\mathfrak{g}}$ is the projection $\mathcal{S}_{\mathbf{i}^{\mathfrak{g}}}\rightarrow \mathcal{S}_{\mathbf{i}^{\mathfrak{g}'}}$.  
Due to weight considerations, this polytope is just the polytope 
$\mathcal{S}_{\mathbf{i}^{\mathfrak{g}'}}(\lambda-t\cdot \alpha^T_{\tilde{\mathfrak{g}}})$, where $\alpha_{\tilde{\mathfrak{g}}}=(\alpha_{i_{\frac{N}{2}+1}},\dots,\alpha_{i_N})$.\par
So we get a decomposition of the lattice points of the polytope $\mathcal{S}_{\mathbf{i}^{\mathfrak{g}}}(\lambda)$ into the disjoint union of direct products of the lattice points in $\mathfrak{g}'$-polytopes and integral points in the branching polytope:
$$\displaystyle{\mathcal{S}^L_{\mathbf{i}^{\mathfrak{g}}}(\lambda)=\Dot{\bigcup}_{t\in \mathcal{S}^L_{\mathbf{i}^{\tilde{\mathfrak{g}}}}(\lambda)}\mathcal{S}^L_{\mathbf{i}^{\mathfrak{g}'}}(\lambda-t\cdot\alpha^T_{\tilde{\mathfrak{g}}})\times \{t\}.}$$
This means that the highest weight module $V_{\mathfrak{g}}(\lambda)$ can be written as
$$\displaystyle{V_{\mathfrak{g}}(\lambda)\cong\bigoplus_{t\in \mathcal{S}^L_{\mathbf{i}^{\tilde{\mathfrak{g}}}}(\lambda)}V_{\mathfrak{g}'}(\lambda-t\cdot\alpha^T_{\tilde{\mathfrak{g}}})}.$$

The transformation maps from Theorems \ref{StringLusztig} and \ref{LusztigString} show that a similar construction is possible in the Lusztig parametrization. In the following chapters we will give a decomposition of $\mathbf{i}$ as above for $\mathfrak{g}=\mathfrak{so}_{2n}$, $\mathfrak{g}=\mathfrak{so}_{2n+1}$ or $\mathfrak{g}=\mathfrak{sp}_{2n}$ and $\mathfrak{g}'=\mathfrak{sl}_{n}$ and explicitly calculate the inequalities defining the string cones. Together with Proposition \ref{Proppoly} this allows us to describe the string branching polytopes. Afterwards, we will translate our results to Lusztig's parametrization.

\section{String cones in type \texorpdfstring{$D_n$}{Dn} }\label{section3}

We now use the Berenstein-Zelevinsky cone and Littelmann's description of the string cone to give explicit inequalities to describe the string cone for a special choice of $\mathbf{i}\in R(\omega_0)$ in the case $\mathfrak{g}=\mathfrak{so}_{2n}$.
We start by fixing our notation: Let $\mathfrak{g}=\mathfrak{so}_{2n}$ be the Lie algebra of type $D_n$. 
We realize it as $\mathfrak{so}_{2n}=\{a\in\mathfrak{gl}_{2n}(\C)\mid a+Ba^TB^{-1}=0\},$ where $B$ is the symplectic non-degenerate bilinear form on $\C^{2n}$ with the matrix
$$B=\begin{pmatrix}
0&I_n\\
I_n&0
\end{pmatrix}.
$$
The Cartan subalgebra is $\mathfrak{h}=\{\operatorname{diag}(x_1,\dots,x_n,-x_1,\dots,-x_n)\}$.
A basis of $\mathfrak{h}^*$ is given by $\{\epsilon_1,\dots,\epsilon_n\}$, where 
$$\epsilon_i:\mathrm{diag}(x_1,\dots x_n,-x_1,\dots,-x_n)\mapsto x_i.$$
We enumerate the simple roots by $$\alpha_1=\epsilon_{1}-\epsilon_2,\dots,\alpha_{n-1}=\epsilon_{n-1}-\epsilon_n,\alpha_n=\epsilon_{n-1}+\epsilon_n.$$ 
Note that we follow the standard enumeration like in \cite{Bo68} here, which differs from the enumeration in \cite{Li98}.
Let $W$ be the Weyl group corresponding to $\mathfrak{g}$ with generators $s_1,\dots ,s_n$ acting on $\mathfrak{h}^*$, where $s_i=(i\;i+1)$ for $i<n$ and $$s_n:(\zeta_1,\dots,\zeta_{n-2},\zeta_{n-1},\zeta_n)\mapsto(\zeta_1,\dots,\zeta_{n-2},-\zeta_{n},-\zeta_{n-1}).$$

We fix our favorite reduced decomposition $\underline{\omega}_0^{D_n}=\underline{\omega}_0^{A_{n-1}}\underline{\omega}_0^{\tilde{D_n}}$ with 
$$\begin{array}{lll}
&\underline{\omega}_0^{A_{n-1}}=(s_{n-1})(s_{n-2}s_{n-1})(s_{n-3}s_{n-2}s_{n-1})\dots (s_{1}s_{2}\dots s_{n-1}),\\
&\underline{\omega}_0^{\tilde{D_n}}=(s_n)(s_{n-2}s_{n-1})(s_{n-3}s_{n-2}s_n)\dots (s_1 s_2\dots s_{n-2}s_{n/n-1}),
\end{array}$$ 
where the last element is $s_n$ for $n$ even and $s_{n-1}$ for $n$ odd. 
We will encounter several cases where we have to make a distinction between $n$ even and $n$ odd. We will stick to the convention that using $/$, the first option is our choice for $n$ even and the second for $n$ odd. Using $\pm$ or $\mp$, the upper sign is our choice for $n$ even and the lower sign for $n$ odd.
The corresponding word is  $\mathbf{i}^{D_n}=\mathbf{i}^{A_{n-1}}\mathbf{i}^{\tilde{D_n}}$ with 
$$
\begin{array}{lll}
&\mathbf{i}^{A_{n-1}}=(n-1,n-2,n-1,n-3,n-2,n-1,\dots,1,2,\dots,n-1),\\
&\mathbf{i}^{\tilde{D_n}}=(n,n-2,n-1,n-3,n-2,n,\dots,1 ,2,\dots,n-2,n/n-1).
\end{array}$$
We see that we can divide both $\mathbf{i}^{A_{n-1}}$ and $\mathbf{i}^{\tilde{D_n}}$ in $n-1$ blocks with increasing entries.\par
Regarding this structure, we now introduce a double indication for $t=(t_1,\dots,t_N)$ ($N=n\cdot(n-1)$) by $t=(t_{1,1}^-,t_{1,2}^-,t_{2,2}^-,\dots,t_{1,n-1}^-,\dots,t_{n-1,n-1}^-,t_{1,1}^+,t_{1,2}^+,t_{2,2}^+,\dots,t_{1,n-1}^+,\dots,t_{n-1,n-1}^+).$
As the notation already yields, $\underline{\omega}_0^{A_{n-1}}$ is a reduced decomposition for the longest element of the Weyl group of type $A_{n-1}$. Therefore we directly obtain all the inequalities just containing the variables with upper index minus as we can project down every $t\in \R^N$ onto its first $N/2$ coordinates.
So, the image of the projection of an element of $\mathcal{S}_{\mathbf{i}^{D_{n}}}$ is in the string cone $\mathcal{S}_{\mathbf{i}^{A_{n-1}}}$. The inequalities there are well known, see for example \cite{BZ96}, \cite{Li98}.
\begin{rem}
Our reduced decomposition is not \textit{nice} in the sense of Littelmann \cite{Li98} as it does not respect the enumeration on the set of simple roots. The nice decomposition for our enumeration would be $(1,2,1,3,2,1,\dots,n-1,\dots,1)\mathbf{i}^{\tilde{D}_n}$ (it is unique up to transpositions of orthogonal reflections). Our results for the branching cone and polytope are valid for any reduced decomposition for the longest word of the Weyl group of $A_{n-1}$. Moreover, the condition that the reduced decomposition for $\omega_0^{D_n}$ should start with one for $\omega_0^{A_{n-1}}$ determines the second part uniquely up to the exchange of orthogonal reflections. Therefore, the branching cones and polytopes which we will study in this and also the upcoming sections are (up to a relabeling of the variables) uniquely determined by the branching itself. 
\end{rem}

From Theorem \ref{BZcone}, we deduce the following Lemma:
\begin{lem}\label{LemIneqD}
For $\mathfrak{g}=\mathfrak{so}_{2n}$, all the points in the string cone $\mathcal{S}_{\mathbf{i}^{D_n}}$ satisfy the inequalities 
\begin{equation*}
    \begin{aligned}
    t_{i,j}^-&\geq t_{i+1,j}^-,\quad t_{i,j}^+\geq t_{i+1,j}^+ \;&\forall\; 1\leq i<j \leq n-1,\\
    t_{i,n-2}^+&\geq t_{i,n-1}^+ \;&\forall\; 1\leq i\leq n-2.    
    \end{aligned}
\end{equation*}

In other words, writing $x\leftarrow y$ for all cover relations $x\geq y$, the string cone $\mathcal{S}_{\mathbf{i}^{D_n}}$ is contained in the order polyhedron of the poset in figure \ref{figD1}.\\

\begin{figure}
\scalebox{0.7}{
\begin{tikzpicture}
  \tikzstyle{arrow} = [->,>=stealth]
\draw (0,0)node(11-){$t_{1,1}^-$};
\draw (0.8,0)node(12-){$t_{1,2}^-$};\draw (0.8,-1.0)node(22-){$t_{2,2}^-$};
\draw (1.6,0)node(d1-){$\dots$};\draw (1.6,-1.0)node(d2-){$\dots$};
\draw (2.4,0)node(1n-2-){$t_{1,n-2}^-$};;\draw (2.4,-1.0)node(2n-2-){$t_{2,n-2}^-$};\draw (2.4,-2.0)node(d3-){$\dots$};\draw (2.4,-3.0)node(n-2n-2-){$t_{n-2,n-2}^-$};
\draw (3.9,0)node(1n-1-){$t_{1,n-1}^-$};\draw (3.9,-1.0)node(2n-1-){$t_{2,n-1}^-$};\draw (3.9,-2.0)node(d4-){$\dots$};;\draw (3.9,-3.0)node(n-2n-1-){$t_{n-2,n-1}^-$};\draw (3.9,-4.0)node(n-1n-1-){$t_{n-1,n-1}^-$};
\draw[arrow](22-)--(12-);\draw[arrow](2n-2-)--(1n-2-);\draw[arrow](2n-1-)--(1n-1-);
\draw[arrow](d3-)--(2n-2-);\draw[arrow](d4-)--(2n-1-);
\draw[arrow](n-2n-2-)--(d3-);\draw[arrow](n-2n-1-)--(d4-);
\draw[arrow](n-1n-1-)--(n-2n-1-);

\draw (11,0)node(11+){$t_{1,1}^+$};
\draw (10.0,-1.0)node(12+){$t_{1,2}^+$};
\draw (9.0,-2.0)node(13+){$t_{1,3}^+$};\draw (11.0,-2.0)node(22+){$t_{2,2}^+$};
\draw (8.0,-3.0)node(14+){$t_{1,4}^+$};\draw (10.0,-3.0)node(23+){$t_{2,3}^+$};
\draw (7.0,-4.0)node(15+){$t_{1,5}^+$};\draw (9.0,-4.0)node(24+){$t_{2,4}^+$};\draw (11.0,-4.0)node(33+){$t_{3,3}^+$};
\draw (6.0,-5.0)node(d1+){};\draw (8.0,-5.0)node(d2+){};\draw (10.0,-5.0)node(d3+){};\draw (9.0,-5.0)node(d4+){};\draw (11.0,-7.0)node(d5+){};\draw (5.0,-7.0)node(d6+){};\draw (9.0,-7.0)node(d7+){};\draw (10.0,-6.5)node(d8+){};\draw (7.0,-7.0)node(d9+){};
\draw (8.5,-6.0)node(d1+){$\dots\dots\dots\dots\dots\dots\dots\dots\dots\dots\dots\dots\dots\dots$};
\draw (6.0,-8.0)node(n-6n-1+){$t_{n-6,n-1}^+$};\draw (8.0,-8.0)node(n-5n-2+){$t_{n-5,n-2}^+$};\draw (10.0,-8.0)node(n-4n-3+){$t_{n-4,n-3}^+$};
\draw (7.0,-9.0)node(n-5n-1+){$t_{n-5,n-1}^+$};\draw (9.0,-9.0)node(n-4n-2+){$t_{n-4,n-2}^+$};\draw (11.0,-9.0)node(n-3n-3+){$t_{n-3,n-3}^+$};
\draw (8.0,-10.0)node(n-4n-1+){$t_{n-4,n-1}^+$};\draw (10.0,-10.0)node(n-3n-2+){$t_{n-3,n-2}^+$};
\draw (9.0,-11.0)node(n-3n-1+){$t_{n-3,n-1}^+$};\draw (11.0,-11.0)node(n-2n-2+){$t_{n-2,n-2}^+$};

\draw (10.0,-12.0)node(n-2n-1+){$t_{n-2,n-1}^+$};
\draw (11.0,-13.0)node(n-1n-1+){$t_{n-1,n-1}^+$};

\draw[arrow](22+)--(12+);
\draw[arrow](23+)--(13+);
\draw[arrow](24+)--(14+);\draw[arrow](33+)--(23+);
\draw[arrow](d2+)--(15+);\draw[arrow](d3+)--(24+);
\draw[arrow](n-6n-1+)--(d6+);\draw[arrow](n-6n-1+)--(d9+);\draw[arrow](n-5n-2+)--(d9+);\draw[arrow](n-4n-3+)--(d7+);

\draw[arrow](n-5n-1+)--(n-6n-1+);\draw[arrow](n-5n-1+)--(n-5n-2+);\draw[arrow](n-4n-2+)--(n-5n-2+);\draw[arrow](n-3n-3+)--(n-4n-3+);
\draw[arrow](n-4n-1+)--(n-5n-1+);\draw[arrow](n-4n-1+)--(n-4n-2+);\draw[arrow](n-3n-2+)--(n-4n-2+);
\draw[arrow](n-3n-1+)--(n-4n-1+);\draw[arrow](n-3n-1+)--(n-3n-2+);\draw[arrow](n-2n-2+)--(n-3n-2+);
\draw[arrow](n-2n-1+)--(n-3n-1+);\draw[arrow](n-2n-1+)--(n-2n-2+);
\draw[arrow](n-1n-1+)--(n-2n-1+);

\end{tikzpicture}
}
\caption{Order polyhedron for Lemma \ref{LemIneqD}}
\label{figD1}
\end{figure}
\end{lem}

\begin{proof}
We use Theorem \ref{BZcone}:
We know that $$\omega_0=\begin{pmatrix}1&2&\dots&n-1&n\\
-1&-2&\dots&-(n-1)&\mp n
\end{pmatrix}.$$
So for $1\leq i<n-1$ we get $$s_i\omega_0=\begin{pmatrix}1&2&\dots&i-1&i&i+1&i+2&\dots&n-1&n\\
-1&-2&\dots&-(i-1)&-(i+1)&-i&-(i+2)&\dots&-(n-1)&\mp n
\end{pmatrix}.$$
Now, the minimal representative of the coset $W_{\hat{i}}s_i\omega_0$ is:
$$z^{(i)}=\begin{pmatrix}1&2&\dots&i-1&i&i+1&i+2&\dots&n-1&n\\
-i&-(i-1)&\dots&-2&i+1&-1&i+2&\dots&n-1&\pm n
\end{pmatrix}.$$
Here, the choice of the sign depends on the parity of $i$ (in the same manner as described above for $n$).\par
We now introduce the notation $$\overrightarrow{i,j}=i,i+1,\dots,j-1,j\textrm{\;for\; }i<j.$$%\textrm{\;and\;}\overleftarrow{i,j}=i,i-1,\dots,j+1,j\textrm{\;for\;} i>j.$$
A word for $z^{(i)}$ is given by $$\mathbf{i}^{(i)}=\mathbf{i}^{A_{n-1},(i)}\mathbf{i}_{1}^{\tilde{D_n},(i)}\mathbf{i}_{2}^{\tilde{D_n},(i)},$$ where
$$
\begin{array}{lll}
\mathbf{i}^{A_{n-1},(i)}&=(\overrightarrow{i,n-1},\overrightarrow{i-1,n-2},\dots,\overrightarrow{1,n-i}),\\[5pt]
\mathbf{i}_{1}^{\tilde{D_n},(i)}&=(n,n-2,n-1,n-3,n-2,n,\dots \overrightarrow{n-i,n-2},n-1/n),\\[5pt]
\mathbf{i}_{2}^{\tilde{D_n},(i)}&=(\overrightarrow{n-i-1,n-2},\dots,\overrightarrow{2,i+1},\overrightarrow{1,i-1}).    
\end{array}
$$
Again, the choice of $n-1/n$ depends on the parity of $i$.
The corresponding permutations are
$$
\begin{array}{lll}
z^{A_{n-1},(i
)}&=\left(\begin{array}{lllllllll}
1&2&\dots&n-i&n-i+1&n-i+2&\dots&n-1&n\\
i+1&i+2&\dots&n&1&2&\dots&i-1&i
\end{array}\right),\\\\
z^{\tilde{D_n},(i
)}_1&=\left(\begin{array}{lllllllll}
1&2&\dots&n-i-1&n-i&n-i+1&\dots&n-1&n\\
1&2&\dots&n-i-1&-n&-(n-1)&\dots&-(n-i+1)&\pm(n-i)
\end{array}\right),\\\\
z^{\tilde{D_n},(i
)}_2&=\left(\begin{array}{llllllllll}
1&2&\dots&i-1&i&i+1&i+2&\dots&n-1&n\\
n-i&n-i+1&\dots&n-2&1&n-1&2&\dots&n-i-1&n
\end{array}\right),\\\\
z^{\tilde{D_n},(i
)}&=z^{\tilde{D_n},(i
)}_1z^{\tilde{D_n},(i
)}_2\\&=\left(\begin{array}{llllllllll}
1&\dots&i-1&i&i+1&i+2&\dots&n-1&n\\
-n&\dots&-(n-i+1)&1&-(n-i+1)&2&\dots&n-i-1&\pm(n-i)
\end{array}\right).
\end{array}
$$

We can write $\mathbf{i}^{A_{n-1},(i)}$ as a subword of $\mathbf{i}^{A_{n-1}}$ in the following way (the underlined parts form the subword):
$$(n-1,n-2,n-1,\overrightarrow{n-3,n-1},\dots,\overrightarrow{\underline{i,n-1}},\overrightarrow{\underline{i-1,n-2}},n-1,\dots,\overrightarrow{\underline{1,n-i}},\overrightarrow{n-i+1,n-1}).$$
This means, starting from the $(n-i)$-th block, we take the first $n-i$ elements of each block.\par
For $\mathbf{i}_2^{\tilde{D}_n,(i)}$ we make the observation, that due to commutation relations we can move the last element from the second last block to the last block. Afterwards we can move the last element from the third block from the right to the second last block and so on. So we get $n-1-i$ different words for $z_2^{\tilde{D_n},(i)}$ and thus also for $z^{\tilde{D_n},(i)}.$

We denote by $\mathbf{i}^{\tilde{D_n},(i),m}$ the word where we stopped this shifting of elements in the $m$-th block and write it as a subword of $\mathbf{i}^{\tilde{D_n}}$ in the following way:
We take the first $i$ elements of the first $m-1$ blocks (or the whole block, if its length is less than $i$), the first $i-1$ elements from the remaining blocks and also the $(i+1)$-th element of all blocks right from the $m$-th block.\par
We now start calculating our inequalities. We just need to consider variables corresponding to entries of $\mathbf{i}$, which are not in the subword.\par
In fact, only two coefficients are not zero:
For $l=m$ and $k=i$ we get 
$$
    \begin{array}{lll}
    &(s_{z^{A_{n-1},(i)}}s_{z^{\tilde{D_n},(i)}_1}\overrightarrow{s_{n-i-1},s_{n-2}}\dots\overrightarrow{s_{n-m+1},s_{n-m+i}}\overrightarrow{s_{n-m},s_{n-m+i-2}}\alpha_{n-m+i-1})(w_i^{\vee})\\[5pt]
    =&(s_{z^{A_{n-1},(i)}}s_{z^{\tilde{D_n},(i)}_1}\overrightarrow{s_{n-i-1},s_{n-2}}\dots\overrightarrow{s_{n-m+1},s_{n-m+i}}\overrightarrow{s_{n-m},s_{n-m+i-2}}(\epsilon_{n-m+i-1}-\epsilon_{n-m+i}))(w_i^{\vee})\\[5pt]
    =&(s_{z^{A_{n-1},(i)}}s_{z^{\tilde{D_n},(i)}_1}(\epsilon_{n-m}-\epsilon_{n-1}))(w_i^{\vee})
    =(s_{z^{A_{n-1},(i)}}(\epsilon_{n-m}+\epsilon_{n-i+1}))(w_i^{\vee})\\[5pt]
    =&(\epsilon_{n-m+i}+\epsilon_{1})(w_i^{\vee})=0+1=1,
    \end{array}
$$
as $n-m+i>i$.
As $i<m$, the case $k=l$ cannot occur here.\par
For $l=m$ and $k=i+1$ we get \par
$$
    \begin{array}{lll}
    &(s_{z^{A_{n-1},(i)}}s_{z^{\tilde{D_n},(i)}_1}\overrightarrow{s_{n-i-1},s_{n-2}}\dots\overrightarrow{s_{n-m+1},s_{n-m+i}}\overrightarrow{s_{n-m},s_{n-m+i-2}}\alpha_{n-m+i})(w_i^{\vee})\\[5pt]
    =&(s_{z^{A_{n-1},(i)}}s_{z^{\tilde{D_n},(i)}_1}\overrightarrow{s_{n-i-1},s_{n-2}}\dots\overrightarrow{s_{n-m+1},s_{n-m+i}}\overrightarrow{s_{n-m},s_{n-m+i-2}}(\epsilon_{n-m+i}-\epsilon_{n-m+i+1}))(w_i^{\vee})\\[5pt]
    =&(s_{z^{A_{n-1},(i)}}s_{z^{\tilde{D_n},(i)}_1}(\epsilon_{n-1}-\epsilon_{n-m+1}))(w_i^{\vee})
    =(s_{z^{A_{n-1},(i)}}(-\epsilon_{n-i+1}-\epsilon_{n-m+1}))(w_i^{\vee})\\[5pt]
    =&(-\epsilon_{1}-\epsilon_{n-m+1+i})(w_i^{\vee}) =-1+0 =-1,
    \end{array}
$$
as $n-m+1+i>i$.
For $k=l=m$ if $(\mathbf{i}^{\tilde{D_n}})_{k,l}=n$ we need to change the sign between the two $\epsilon$ but as the second summand is $0$, this does not change the result.\par
The calculations for all the other coefficients can be found in the appendix.\par
This gives us the inequality $t^+_{i,m}\geq t^ +_{i+1,m}$ for $1\leq i< m\leq n-1$.
As the computations for the inequalities obtained from $s_{n-1}$ are quite similar to those above, we refer again to the appendix for the details. 
As result, again only two coefficients are not zero, which gives us the inequality $t^+_{m,n-2}\geq t^+_{m,n-1}$ for $1\leq m\leq n-2$. 
\end{proof}
We use this lemma together with Theorem \ref{coneproduct} to get more inequalities for the points in the string cone $\mathcal{S}_{\mathbf{i}^{D_n}}$. Here, the decomposition of $\mathbf{i}^{D_n}$ in $\mathbf{i}^{A_{n-1}}\mathbf{i}^{\tilde{D}_n}$ is crucial, because it allows us to apply Theorem \ref{coneproduct}, which means in this case that we can do a projection from $\mathcal{S}_{\mathbf{i}^{D_n}}$ to $\mathcal{S}_{\mathbf{i}^{A_{n-1}}}$.
\begin{thm}\label{thmBZfulfillsineq}
For $\mathfrak{g}=\mathfrak{so}_{2n}$ all the points in the string cone $\mathcal{S}_{\mathbf{i}^{D_n}}$ satisfy the inequalities 
\begin{equation}
\begin{aligned}
t_{i,j}^-&\geq t_{i+1,j}^-,\quad t_{i,j}^+\geq t_{i+1,j}^+ \;&\forall\; 1\leq i<j \leq n-1,\\
t_{i,j}^+&\geq t_{i,j+1}^+ \;&\forall\; 1\leq i\leq j\leq n-2.    \label{ineqCimin}
\end{aligned}
\end{equation}
In other words, the string cone $\mathcal{S}_{\mathbf{i}^{D_n}}$ is contained in the order polyhedron of the poset in figure \ref{figD2}.\\
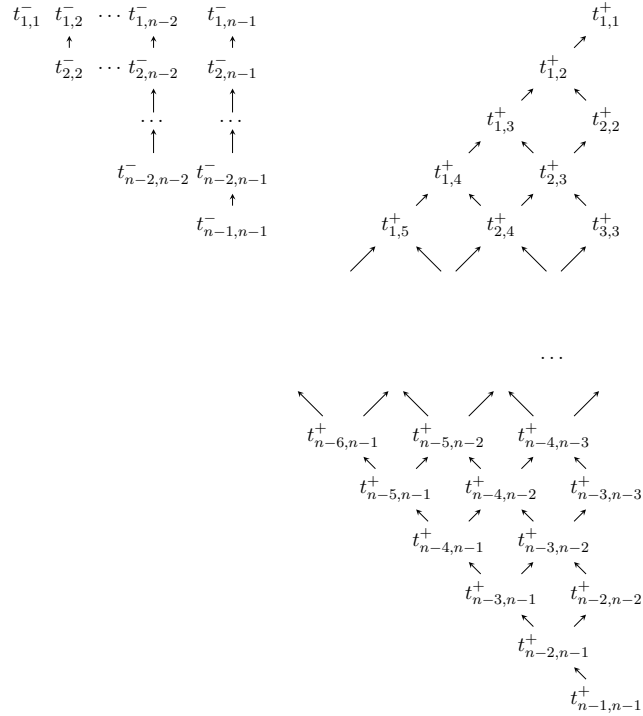
\begin{figure}
\scalebox{0.7}{
\begin{tikzpicture}\label{ineqCiminD}
  \tikzstyle{arrow} = [->,>=stealth]
\draw (0,0)node(11-){$t_{1,1}^-$};
\draw (0.8,0)node(12-){$t_{1,2}^-$};\draw (0.8,-1.0)node(22-){$t_{2,2}^-$};
\draw (1.6,0)node(d1-){$\dots$};\draw (1.6,-1.0)node(d2-){$\dots$};
\draw (2.4,0)node(1n-2-){$t_{1,n-2}^-$};;\draw (2.4,-1.0)node(2n-2-){$t_{2,n-2}^-$};\draw (2.4,-2.0)node(d3-){$\dots$};\draw (2.4,-3.0)node(n-2n-2-){$t_{n-2,n-2}^-$};
\draw (3.9,0)node(1n-1-){$t_{1,n-1}^-$};\draw (3.9,-1.0)node(2n-1-){$t_{2,n-1}^-$};\draw (3.9,-2.0)node(d4-){$\dots$};;\draw (3.9,-3.0)node(n-2n-1-){$t_{n-2,n-1}^-$};\draw (3.9,-4.0)node(n-1n-1-){$t_{n-1,n-1}^-$};
\draw[arrow](22-)--(12-);\draw[arrow](2n-2-)--(1n-2-);\draw[arrow](2n-1-)--(1n-1-);
\draw[arrow](d3-)--(2n-2-);\draw[arrow](d4-)--(2n-1-);
\draw[arrow](n-2n-2-)--(d3-);\draw[arrow](n-2n-1-)--(d4-);
\draw[arrow](n-1n-1-)--(n-2n-1-);

\draw (11,0)node(11+){$t_{1,1}^+$};
\draw (10.0,-1.0)node(12+){$t_{1,2}^+$};
\draw (9.0,-2.0)node(13+){$t_{1,3}^+$};\draw (11.0,-2.0)node(22+){$t_{2,2}^+$};
\draw (8.0,-3.0)node(14+){$t_{1,4}^+$};\draw (10.0,-3.0)node(23+){$t_{2,3}^+$};
\draw (7.0,-4.0)node(15+){$t_{1,5}^+$};\draw (9.0,-4.0)node(24+){$t_{2,4}^+$};\draw (11.0,-4.0)node(33+){$t_{3,3}^+$};
\draw (6.0,-5.0)node(d1+){};\draw (8.0,-5.0)node(d2+){};\draw (10.0,-5.0)node(d3+){};\draw (9.0,-5.0)node(d4+){};\draw (11.0,-7.0)node(d5+){};\draw (5.0,-7.0)node(d6+){};\draw (9.0,-7.0)node(d7+){};\draw (10.0,-6.5)node(d8+){$\dots$};\draw (7.0,-7.0)node(d9+){};\draw (11.0,-7.0)node(d10+){};
\draw (6.0,-8.0)node(n-6n-1+){$t_{n-6,n-1}^+$};\draw (8.0,-8.0)node(n-5n-2+){$t_{n-5,n-2}^+$};\draw (10.0,-8.0)node(n-4n-3+){$t_{n-4,n-3}^+$};
\draw (7.0,-9.0)node(n-5n-1+){$t_{n-5,n-1}^+$};\draw (9.0,-9.0)node(n-4n-2+){$t_{n-4,n-2}^+$};\draw (11.0,-9.0)node(n-3n-3+){$t_{n-3,n-3}^+$};
\draw (8.0,-10.0)node(n-4n-1+){$t_{n-4,n-1}^+$};\draw (10.0,-10.0)node(n-3n-2+){$t_{n-3,n-2}^+$};
\draw (9.0,-11.0)node(n-3n-1+){$t_{n-3,n-1}^+$};\draw (11.0,-11.0)node(n-2n-2+){$t_{n-2,n-2}^+$};

\draw (10.0,-12.0)node(n-2n-1+){$t_{n-2,n-1}^+$};
\draw (11.0,-13.0)node(n-1n-1+){$t_{n-1,n-1}^+$};
\draw[arrow](12+)--(11+);
\draw[arrow](13+)--(12+);\draw[arrow](22+)--(12+);
\draw[arrow](14+)--(13+);\draw[arrow](23+)--(13+);\draw[arrow](23+)--(22+);
\draw[arrow](15+)--(14+);\draw[arrow](24+)--(14+);\draw[arrow](24+)--(23+);\draw[arrow](33+)--(23+);
\draw[arrow](d1+)--(15+);\draw[arrow](d2+)--(15+);\draw[arrow](d2+)--(24+);\draw[arrow](d3+)--(24+);\draw[arrow](d3+)--(33+);
\draw[arrow](n-6n-1+)--(d6+);\draw[arrow](n-6n-1+)--(d9+);\draw[arrow](n-5n-2+)--(d9+);\draw[arrow](n-4n-3+)--(d7+);\draw[arrow](n-5n-2+)--(d7+);\draw[arrow](n-4n-3+)--(d10+);

\draw[arrow](n-5n-1+)--(n-6n-1+);\draw[arrow](n-5n-1+)--(n-5n-2+);\draw[arrow](n-4n-2+)--(n-5n-2+);\draw[arrow](n-4n-2+)--(n-4n-3+);\draw[arrow](n-3n-3+)--(n-4n-3+);\draw[arrow](n-3n-2+)--(n-3n-3+);
\draw[arrow](n-4n-1+)--(n-5n-1+);\draw[arrow](n-4n-1+)--(n-4n-2+);\draw[arrow](n-3n-2+)--(n-4n-2+);
\draw[arrow](n-3n-1+)--(n-4n-1+);\draw[arrow](n-3n-1+)--(n-3n-2+);\draw[arrow](n-2n-2+)--(n-3n-2+);
\draw[arrow](n-2n-1+)--(n-3n-1+);\draw[arrow](n-2n-1+)--(n-2n-2+);
\draw[arrow](n-1n-1+)--(n-2n-1+);

\end{tikzpicture}
}
\caption{Order polyhedron for Theorem \ref{thmBZfulfillsineq}}
\label{figD2}
\end{figure}
\end{thm}
\begin{proof}
We use induction on $n$. For $n=3$, the statement is equivalent to Lemma \ref{LemIneqD}. \par
Now we assume our claim holds true for $D_n$ and consider $D_{n+1}$. Due to Lemma \ref{LemIneqD}, we obtain all the inequalities except
\begin{equation}
t^+_{i,j}\geq t^+_{i,j+1}\;\forall\; 1\leq i\leq j\leq n-2.\label{ineqinduction}    
\end{equation}
Due to Theorem \ref{coneproduct} with $I_0=\{1,\dots,n-1\}$ and $I_1=\{1,\dots,n\}$ we know that the string cone $\mathcal{S}_{\mathbf{i}^{D_n}}$ is the direct product of two cones. 
The first one is the string cone for $A_{n-1}$. The second one we denote by $\mathcal{S}_{\mathbf{i}^{\tilde{D}_n}}$.
So, for a dominant weight $\lambda$ for $D_{n+1}$ and a basis element $v\in V_{D_{n+1}}(\lambda)$ we can write the string of $v$ in direction $\mathbf{i}^{D_{n+1}}$ as $t_v=(t_1,t_2)$ with $t_2\in \mathcal{S}_{\mathbf{i}^{\tilde{D}_{n+1}}}(\lambda)$ and $t_1\in \mathcal{S}_{\mathbf{i}^{A_n}}(\lambda-t_2\cdot\alpha^T_{\tilde{D}_{n+1}})$ (see subsection \ref{subsecbp}).\par
As $\underline{0}\in \mathcal{S}_{\mathbf{i}^{A_n}}(\lambda-t_2\cdot\alpha^T_{\tilde{D}_{n+1}})$, we also have $t':=(\underline{0},t_2)\in \mathcal{S}_{\mathbf{i}^{D_{n+1}}}(\lambda)$.
Let $v'$ be the basis element in $V_{D_{n+1}}(\lambda)$ with string $t'$ in direction $\mathbf{i}^{D_{n+1}}$.
Now, let $\mathfrak{l}$ be the Levi subalgebra of $\mathfrak{g}$ associated to $\alpha_2,\dots,\alpha_{n+1} $. It is isomorphic to $\mathfrak{so}_{2n}$.
We consider the projection $$\pi_D:\R^{n\cdot(n+1)}\rightarrow\R^{(n-1)\cdot n},\;(t^-_{1,1},\dots,t^-_{n,n},t^+_{1,1},\dots,t^+_{n,n})\mapsto (t^-_{1,1},\dots,t^-_{n-1,n-1},t^+_{1,1},\dots,t^+_{n-1,n-1})$$ which just forgets the coordinates $t^-_{i,n}$ and $t^+_{i,n}$ for $1\leq i\leq n$.
If we consider $\mathbf{i}^{D_{n+1}}$ as a tuple living in $\R^{n\cdot(n+1)}$, we can also apply $\pi_D$ on it.
As $\pi_D(\mathbf{i}^{D_{n+1}})$ is a reduced expression for the longest word in the Weyl group of $\mathfrak{l}$ and $(t'^-_{1,1},\dots,t'^-_{n,n})=\underline{0}$, we know due to Lemma \ref{lemres} that, when considering $V_{D_{n+1}}(\lambda)$ as a $\mathfrak{l}-$module, $\pi_D(t')$ is the string of  $v'$ in direction $\pi_D(\mathbf{i}^{D_{n+1}})$.
As $\mathfrak{l}$ is isomorphic to $\mathfrak{so}_{2n}$ by the index shift $i\mapsto i-1$ on the set of simple roots and the same index shift applied on $\pi_D(\mathbf{i}^{D_{n+1}})$ gives $\mathbf{i}^{D_n}$, due to our induction hypothesis $\pi_D(t')$ and therefore also $t'$ and $t$ fulfill the inequalities (\ref{ineqinduction}).

\end{proof}
\begin{rem}
The proof above does not depend on the reduced decomposition for $\omega_0^{A_{n-1}}$. We might take any reduced decomposition for which we know the string cone $\mathcal{S}_{\mathbf{i}^{A_{n-1}}}$ or even one for which we do not know this cone if we are just interested in the branching cone. 
\end{rem}
We now want to show that the string cone $\mathcal{S}_{\mathbf{i}^{D_n}}$ is exactly given by the inequalities (\ref{ineqCimin}).

Therefore, we use Theorem \ref{thmBZfulfillsineq} and Littelmann's description of the string cone to give an explicit description of the string cone $\mathcal{S}_{\mathbf{i}^{D_n}}$:
\begin{thm}\label{thmStringCone}
The string cone $\mathcal{S}_{\mathbf{i}^{D_n}}$ is exactly given by the set of all points in $t\in \R^N_{\geq 0}$ satisfying the inequalities (\ref{ineqCimin}).
\end{thm}
\begin{proof}
We already know from Theorem \ref{thmBZfulfillsineq} that all the points in the string cone fulfill the inequalities (\ref{ineqCimin}). So it remains to show that each $t\in \R^{N}$ fulfilling the inequalities is in the string cone. As the string cone is rational (\cite{Li98}, Proposition 1.5), we might restrict to the case $t\in\Z^N_{\geq 0}$ and use Theorem \ref{thmrecl}:\par
Let $t\in\Z^N$ such that $t$ fulfills all the inequalities (\ref{ineqCimin}). 
We now claim the following: For $j=1,\dots,N$, $m^j$ also fulfills the inequalities (\ref{ineqCimin}). We use the convention $m^j_k=0$ for $k>j$ here.
We will prove this claim by induction on $j$, starting with $j=N$ and going down from $j$ to $j-1$.\par
As $m^N=t$, our induction hypothesis holds true for $j=N$. So we now can assume that $m^j$ fulfills the inequalities and prove them for $m^{j-1}.$ We will use our double indication for this proof again, writing $k=(k_1,k_2)^-$ for $k\leq\frac{N}{2}$ and $k=(k_1,k_2)^+$ for $k>\frac{N}{2}$. \par
We always assume $k<j$ in the following and use the convention $m^j_{(k_1,k_2)^{\pm} }=0$ for $k_1>k_2$.\par
For $i_k\neq i_j$, which is equivalent to $\alpha_{i_k}\neq\alpha_{i_j}$, we have $\Delta^j(k)=m^j_k$ and so $$m_k^{j-1}=\min\{m_k^j,\Delta^j(k)\}=m_k^j.$$ 
Therefore, we get by our induction hypothesis:
\begin{equation}\notag
\begin{aligned}
m^{j-1}_{(k_1,k_2)^\pm}&=m^j_{(k_1,k_2)^\pm}\overset{I.H.}{\geq}m^j_{(k_1+1,k_2)^\pm}
&\geq\min\{m^j_{(k_1+1,k_2)^\pm},\Delta^j((k_1+1,k_2)^\pm)\}=m^{j-1}_{(k_1+1,k_2)^\pm}.
\end{aligned}    
\end{equation}
For $k>\frac{N}{2}$ and $k_2<n-1$, we also get
\begin{equation}\notag
    \begin{aligned}
m^{j-1}_{(k_1,k_2)^+}&=m^j_{(k_1,k_2)^+}\overset{I. H.}{\geq}m^j_{(k_1,k_2+1)^+}
&\geq\min\{m^j_{(k_1,k_2 +1)^+},\Delta^j((k_1,k_2+1)^+)\}=m^{j-1}_{(k_1,k_2+1)^+}.    
    \end{aligned}
\end{equation}
Now we consider the case $i_k=i_j$. \par
For $k_2<n-1$ and $k<\frac{N}{2}$ or $k_1<k_2$, we get
$$
    \begin{array}{lll}
\Delta^j((k_1,k_2)^\pm)&=\max\{\theta((k_1,k_2)^\pm,l,j)\mid k<l\leq j,\;\alpha_{i_l}=\alpha_{i_j}\}\\[5pt]
&\geq \theta((k_1,k_2)^\pm,(k_1+1,k_2+1)^\pm,j)\\[5pt]
&=m^j_{(k_1+1,k_2+1)^\pm}-\sum_{k<s\leq (k_1+1,k_2+1)^\pm}m^j_s\alpha_{i_s}(\alpha^{\vee}_{i_j})\\[5pt]
&=m^j_{(k_1+1,k_2+1)^\pm}+m^j_{(k_1+1,k_2)^\pm}+m^j_{(k_1,k_2+1)^\pm}-2m^j_{(k_1+1,k_2+1)^\pm}\\[5pt]
&=m^j_{(k_1+1,k_2)^\pm}+\underbrace{m^j_{(k_1,k_2+1)^\pm}-m^j_{(k_1+1,k_2+1)^\pm}}_{\geq 0 \;(I.H.)}\\[5pt]
&\geq m^j_{(k_1+1,k_2)^\pm}.
    \end{array}
$$
For $k>\frac{N}{2}$, we can rewrite our calculation from above to obtain 
$$
    \begin{array}{lll}
\Delta^j((k_1,k_2)^+)&=m^j_{(k_1+1,k_2)^+}+m^j_{(k_1,k_2+1)^+}-m^j_{(k_1+1,k_2+1)^+}\\[5pt]
&=m^j_{(k_1,k_2+1)^+}+\underbrace{m^j_{(k_1+1,k_2)^+}-m^j_{(k_1+1,k_2+1)^+}}_{\geq 0 \;(I.H.)}\\[5pt]
&\geq m^j_{(k_1,k_2+1)^+}.
    \end{array}
$$
The remaining calculations are quite similar and can be found in the appendix.
This finishes our claim.\par
Now, we show $$\Delta^j(k)\geq 0\;\forall\; 1\leq k< j\leq N.$$ We essentially already proved that in the proof of our claim above and have just to recollect the important statements here.\par
For $1\leq k<j\leq N$ we find:\par
If $i_k\neq i_j$ 
$$\Delta^j(k)=m^j_k\geq 0,$$
as $m^j$ fulfills (\ref{ineqCimin}).\par
If $i_k=i_j$
$$\Delta^j(k)\geq m^j_{(k_1+1,k_2)^\pm}\geq 0.$$
This finishes the proof of the Theorem.

\end{proof}

%%%%%%%%%%%%%%%%%%%%%%%%%%%%%%%%%%%%%%%%%%%%%%%%%%%%%%%%%%%%%%%%%%%%%%%%%%%%%%%%%%%%%%%%%%%%%%%%%%
%%%%%%%%%%%%%%%%%%%%%%%%%%%%%%%%%%%%%%%%%%%%%%%%%%%%%%%%%%%%%%%%%%%%%%%%%%%%%%%%%%%%%%%%%%%%%%%%%%

\section{String cones in type \texorpdfstring{$B_n$}{Bn} }\label{section4}
We can prove a very similar result for $B_n$. The proofs in this section will be nearly completely analogous to those in the previous one. The main difference will appear in the proof of Theorem \ref{thmBZfulfillsineqB}, where we will make use of our knowledge of the $D_n$-cone from Theorem \ref{thmBZfulfillsineq}. \par
We start again by fixing our notation: Let $\mathfrak{g}=\mathfrak{so}_{2n+1}$ be the Lie algebra of type $B_n$.
We realize it as $\mathfrak{so}_{2n+1}=\{a\in\mathfrak{gl}_{2n+1}(\C)\mid a+Ba^TB^{-1}=0\},$ where $B$ is the symplectic non-degenerate bilinear form on $\C^{2n+1}$ with the matrix
$$B=\begin{pmatrix}
0&I_n&0\\
I_n&0&0\\
0&0&1
\end{pmatrix}.
$$
The Cartan subalgebra is $\mathfrak{h}=\{\operatorname{diag}(x_1,\dots,x_n,-x_1,\dots,-x_n,0)\}$.
A basis of $\mathfrak{h}^*$ is given by $\{\epsilon_1,\dots,\epsilon_n\}$, where 
$$\epsilon_i:\mathrm{diag}(x_1,\dots x_n,-x_1,\dots,-x_n,0)\mapsto x_i.$$
We enumerate the simple roots by $$\alpha_1=\epsilon_{1}-\epsilon_2,\dots,\alpha_{n-1}=\epsilon_{n-1}-\epsilon_n,\alpha_n=\epsilon_n.$$ 
Let $W$ be the Weyl group corresponding to $\mathfrak{g}$ with generators $s_1,\dots ,s_n$ acting on $\mathfrak{h}^*$, where $s_i=(i\;i+1)$ for $i<n$ and $$s_n:(\zeta_1,\dots,\zeta_{n-2},\zeta_{n-1},\zeta_n)\mapsto(\zeta_1,\dots,\zeta_{n-1},-\zeta_{n}).$$

We fix our favorite reduced decomposition $\underline{\omega}_0^{B_n}=\underline{\omega}_0^{A_{n-1}}\underline{\omega}_0^{\tilde{B_n}}$, where the corresponding reduced word for th second part is
$$\begin{array}{lll}

&\mathbf{i}^{\tilde{B_n}}=(n,n-1,n,n-2,n-1,n,\dots,1 ,2,\dots,n-1,n).
\end{array}$$
We do not need to distinguish between $n$ even and $n$ odd for $B_n$.

We see that we can divide $\mathbf{i}^{\tilde{B_n}}$ in $n$ increasing blocks.\par
The double indication for $t=(t_1,\dots,t_N)$ ($N=n^2$) here is $$t=(t_{1,1}^-,t_{1,2}^-,t_{2,2}^-,\dots,t_{1,n-1}^-,\dots,t_{n-1,n-1}^+,t_{1,1}^+,t_{1,2}^+,t_{2,2}^+,\dots,t_{1,n}^+,\dots,t_{n,n}^+).$$
Note that in this case we have more variables with upper index $+$ than with upper index $-$, which corresponds to the fact that to extend the set of positive roots of $A_{n-1}$ to the set of positive roots of $B_n$ we do not only have to add all the pairs $\epsilon_i+\epsilon_j$ but also all the $\epsilon_i$.\par

From Theorem \ref{BZcone} we deduce the following Lemma. Note the difference to the $D_n$ case concerning the third set of inequalities.
\begin{lem}\label{LemineqB}
For $\mathfrak{g}=\mathfrak{so}_{2n+1}$ all the points in the string cone $\mathcal{S}_{\mathbf{i}^{B_n}}$ satisfy the inequalities 
\begin{equation*}
\begin{aligned}
t_{i,j}^-&\geq t_{i+1,j}^-, \;&\forall\; 1< i\leq j \leq n-1,\\
t_{i,j}^+&\geq t_{i+1,j}^+ \;&\forall\; 1< i\leq j \leq n,\\
t_{i,i}^+&\geq t_{i,i+1}^+ \;&\forall\; 1\leq i\leq n-1.    
\end{aligned}
\end{equation*}
In other words, the cone is a subcone of the order polyhedron of the poset in figure \ref{figB1}.\\
\begin{figure}
\scalebox{0.7}{
\begin{tikzpicture}
  \tikzstyle{arrow} = [->,>=stealth]
\draw (0,0)node(11-){$t_{1,1}^-$};
\draw (0.8,0)node(12-){$t_{1,2}^-$};\draw (0.8,-1.0)node(22-){$t_{2,2}^-$};
\draw (1.6,0)node(d1-){$\dots$};\draw (1.6,-1.0)node(d2-){$\dots$};
\draw (2.4,0)node(1n-2-){$t_{1,n-2}^-$};;\draw (2.4,-1.0)node(2n-2-){$t_{2,n-2}^-$};\draw (2.4,-2.0)node(d3-){$\dots$};\draw (2.4,-3.0)node(n-2n-2-){$t_{n-2,n-2}^-$};
\draw (3.9,0)node(1n-1-){$t_{1,n-1}^-$};\draw (3.9,-1.0)node(2n-1-){$t_{2,n-1}^-$};\draw (3.9,-2.0)node(d4-){$\dots$};;\draw (3.9,-3.0)node(n-2n-1-){$t_{n-2,n-1}^-$};\draw (3.9,-4.0)node(n-1n-1-){$t_{n-1,n-1}^-$};
\draw[arrow](22-)--(12-);\draw[arrow](2n-2-)--(1n-2-);\draw[arrow](2n-1-)--(1n-1-);
\draw[arrow](d3-)--(2n-2-);\draw[arrow](d4-)--(2n-1-);
\draw[arrow](n-2n-2-)--(d3-);\draw[arrow](n-2n-1-)--(d4-);
\draw[arrow](n-1n-1-)--(n-2n-1-);

\draw (11,0)node(11+){$t_{1,1}^+$};
\draw (10.0,-1.0)node(12+){$t_{1,2}^+$};
\draw (9.0,-2.0)node(13+){$t_{1,3}^+$};\draw (11.0,-2.0)node(22+){$t_{2,2}^+$};
\draw (8.0,-3.0)node(14+){$t_{1,4}^+$};\draw (10.0,-3.0)node(23+){$t_{2,3}^+$};
\draw (7.0,-4.0)node(15+){$t_{1,5}^+$};\draw (9.0,-4.0)node(24+){$t_{2,4}^+$};\draw (11.0,-4.0)node(33+){$t_{3,3}^+$};
\draw (6.0,-5.0)node(d1+){};\draw (8.0,-5.0)node(d2+){};\draw (10.0,-5.0)node(d3+){};\draw (9.0,-5.0)node(d4+){};\draw (11.0,-7.0)node(d5+){};\draw (5.0,-7.0)node(d6+){};\draw (9.0,-7.0)node(d7+){};\draw (10.0,-6.5)node(d8+){$\dots$};\draw (7.0,-7.0)node(d9+){};\draw (11.0,-7.0)node(d10+){};
\draw (6.0,-8.0)node(n-5n+){$t_{n-5,n}^+$};\draw (8.0,-8.0)node(n-4n-1+){$t_{n-4,n-1}^+$};\draw (10.0,-8.0)node(n-3n-2+){$t_{n-3,n-2}^+$};
\draw (7.0,-9.0)node(n-4n+){$t_{n-4,n}^+$};\draw (9.0,-9.0)node(n-3n-1+){$t_{n-3,n-1}^+$};\draw (11.0,-9.0)node(n-2n-2+){$t_{n-2,n-2}^+$};
\draw (8.0,-10.0)node(n-3n+){$t_{n-3,n}^+$};\draw (10.0,-10.0)node(n-2n-1+){$t_{n-2,n-1}^+$};
\draw (9.0,-11.0)node(n-2n+){$t_{n-2,n}^+$};\draw (11.0,-11.0)node(n-1n-1+){$t_{n-1,n-1}^+$};

\draw (10.0,-12.0)node(n-1n+){$t_{n-1,n}^+$};
\draw (11.0,-13.0)node(nn+){$t_{n,n}^+$};
\draw[arrow](12+)--(11+);
\draw[arrow](22+)--(12+);
\draw[arrow](23+)--(13+);\draw[arrow](23+)--(22+);
\draw[arrow](24+)--(14+);\draw[arrow](33+)--(23+);
\draw[arrow](d2+)--(15+);\draw[arrow](d3+)--(24+);\draw[arrow](d3+)--(33+);
\draw[arrow](n-5n+)--(d6+);\draw[arrow](n-4n-1+)--(d9+);\draw[arrow](n-3n-2+)--(d7+);\draw[arrow](n-3n-2+)--(d10+);

\draw[arrow](n-4n+)--(n-5n+);\draw[arrow](n-3n-1+)--(n-4n-1+);\draw[arrow](n-2n-2+)--(n-3n-2+);\draw[arrow](n-2n-1+)--(n-2n-2+);
\draw[arrow](n-3n+)--(n-4n+);\draw[arrow](n-2n-1+)--(n-3n-1+);
\draw[arrow](n-2n+)--(n-3n+);\draw[arrow](n-1n-1+)--(n-2n-1+);
\draw[arrow](n-1n+)--(n-2n+);\draw[arrow](n-1n+)--(n-1n-1+);
\draw[arrow](nn+)--(n-1n+);
\end{tikzpicture}
}
\caption{Order polyhedron for Lemma \ref{LemineqB}}
\label{figB1}
\end{figure}
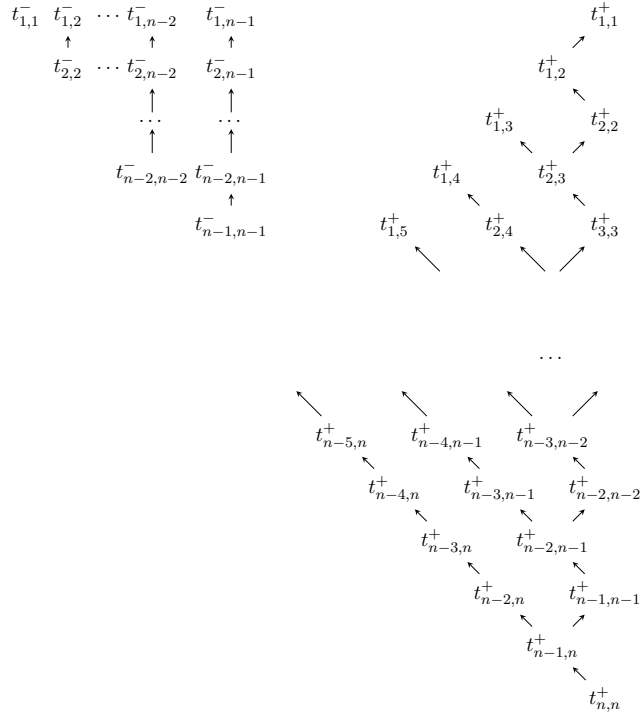

\end{lem}
\begin{proof}
We use Theorem \ref{BZcone}:
We know that $$\omega_0=\begin{pmatrix}1&2&\dots&n-1&n\\
-1&-2&\dots&-(n-1)&- n
\end{pmatrix}.$$
So for $1\leq i\leq n-1$, we get $$s_i\omega_0=\begin{pmatrix}1&2&\dots&i-1&i&i+1&i+2&\dots&n-1&n\\
-1&-2&\dots&-(i-1)&-(i+1)&-i&-(i+2)&\dots&-(n-1)&- n
\end{pmatrix}.$$
Now acting by the parabolic subgroup $W_{\hat{i}}$ gives us the minimal representative of $W_{\hat{i}}s_i\omega_0$ :$$z^{(i)}=\begin{pmatrix}1&2&\dots&i-1&i&i+1&i+2&\dots&n-1&n\\
-i&-(i-1)&\dots&-2&i+1&-1&i+2&\dots&n-1& n
\end{pmatrix}.$$
A word for $z^{(i)}$ is given by $$\mathbf{i}^{(i)}=\mathbf{i}^{A_{n-1},(i)}\mathbf{i}_{1}^{\tilde{B_n},(i)}\mathbf{i}_{2}^{\tilde{B_n},(i)}$$ where
$$
\begin{array}{lll}
\mathbf{i}^{A_{n-1},(i)}&=(\overrightarrow{i,n-1},\overrightarrow{i-1,n-2},\dots,\overrightarrow{1,n-i}),\\[5pt]
\mathbf{i}_{1}^{\tilde{B_n},(i)}&=(n,n-1,n,\overrightarrow{n-2,n},\dots \overrightarrow{n-i+1,n}),\\[5pt]
\mathbf{i}_{2}^{\tilde{B_n},(i)}&=(\overrightarrow{n-i,,n-1},\dots,\overrightarrow{2,i+1},\overrightarrow{1,i-1}).    
\end{array}
$$
The corresponding permutations are
$$
\begin{array}{lll}
z^{A_{n-1},(i
)}&=\left(\begin{array}{lllllllll}
1&2&\dots&n-i&n-i+1&n-i+2&\dots&n-1&n\\
i+1&i+2&\dots&n&1&2&\dots&i-1&i
\end{array}\right),\\\\
z^{\tilde{B_n},(i
)}_1&=\left(\begin{array}{lllllllll}
1&2&\dots&n-i&n-i+1&n-i+2&\dots&n-1&n\\
1&2&\dots&n-i&-n&-(n-1)&\dots&-(n-i)&-(n-i+1)
\end{array}\right),\\\\
z^{\tilde{B_n},(i
)}_2&=\left(\begin{array}{llllllllll}
1&2&\dots&i-1&i&i+1&i+2&\dots&n-1&n\\
n-i+1&n-i+2&\dots&n-1&1&n&2&\dots&n-i-1&n-i
\end{array}\right),\\\\
z^{\tilde{B_n},(i
)}&=z^{\tilde{B_n},(i
)}_1z^{\tilde{B_n},(i
)}_2\\&=\left(\begin{array}{llllllllll}
1&\dots&i-1&i&i+1&i+2&\dots&n-1&n\\
-n&\dots&-(n-i+2)&1&-(n-i+1)&2&\dots&n-i-1&n-i
\end{array}\right).
\end{array}
$$

Again, we can write $\mathbf{i}^{A_{n-1},(i)}$ as a subword of $\mathbf{i}^{A_{n-1}}$ in the following way:
$$(n-1,n-2,n-1,\overrightarrow{n-3,n-1},\dots,\overrightarrow{\underline{i,n-1}},\overrightarrow{\underline{i-1,n-2}},n-1,\dots,\overrightarrow{\underline{1,n-i}},\overrightarrow{n-i+1,n-1}).$$

Due to commutation relations we can again find $m$ different words for $z^{\tilde{B_n},(i)}$.\par
We find two coefficients which are not zero:
For $l=m$ and $k=i$ we get 
$$
    \begin{array}{lll}
    &(s_{z^{A_{n-1},(i)}}s_{z^{\tilde{B_n},(i)}_1}\overrightarrow{s_{n-i},s_{n-1}}\dots\overrightarrow{s_{n-m+2},s_{n-m+i+1}}\overrightarrow{s_{n-m+1},s_{n-m+i-1}}\alpha_{n-m+i})(w_i^{\vee})\\[5pt]
    =&(s_{z^{A_{n-1},(i)}}s_{z^{\tilde{B_n},(i)}_1}\overrightarrow{s_{n-i},s_{n-1}}\dots\overrightarrow{s_{n-m+2},s_{n-m+i+1}}\overrightarrow{s_{n-m+1},s_{n-m+i-1}}(\epsilon_{n-m+i}-\epsilon_{n-m+i+1}))(w_i^{\vee})\\[5pt]
    =&(s_{z^{A_{n-1},(i)}}s_{z^{\tilde{B_n},(i)}_1}(\epsilon_{n-m+1}-\epsilon_{n}))(w_i^{\vee})=(s_{z^{A_{n-1},(i)}}(\epsilon_{n-m+1}+\epsilon_{n-i+1}))(w_i^{\vee})\\[5pt]
    =&(\epsilon_{n-m+1+i}+\epsilon_{1})(w_i^{\vee})=0+1 =1,
    \end{array}
$$
as $n-m+1+i>i$.\par
For $l=m>i+1$ and $k=i+1$ we get 
$$
\begin{array}{lll}
    &(s_{z^{A_{n-1},(i)}}s_{z^{\tilde{B_n},(i)}_1}\overrightarrow{s_{n-i},s_{n-1}}\dots\overrightarrow{s_{n-m+2},s_{n-m+i+1}}\overrightarrow{s_{n-m+1},s_{n-m+i-1}}\alpha_{n-m+i+1})(w_i^{\vee})\\[5pt]
    =&(s_{z^{A_{n-1},(i)}}s_{z^{\tilde{B_n},(i)}_1}\overrightarrow{s_{n-i},s_{n-1}}\dots\overrightarrow{s_{n-m+2},s_{n-m+i+1}}\overrightarrow{s_{n-m+1},s_{n-m+i-1}}(\epsilon_{n-m+i+1}-\epsilon_{n-m+i+2}))(w_i^{\vee})\\[5pt]
    =&(s_{z^{A_{n-1},(i)}}s_{z^{\tilde{B_n},(i)}_1}(\epsilon_{n}-\epsilon_{n-m+2}))(w_i^{\vee})=(s_{z^{A_{n-1},(i)}}(-\epsilon_{n-i+1}-\epsilon_{n-m+2}))(w_i^{\vee})\\[5pt]
    =&(-\epsilon_{1}-\epsilon_{n-m+2+i})(w_i^{\vee})=-1+0=-1,
    \end{array}
$$
as $n-m+2+i>i$.\par
For $l=m=i+1=k$ we get 
$$
    \begin{array}{lll}
     &(s_{z^{A_{n-1},(i)}}s_{z^{\tilde{B_n},(i)}_1}\overrightarrow{s_{n-i},s_{n-2}}\alpha_{n})(w_i^{\vee})=(s_{z^{A_{n-1},(i)}}s_{z^{\tilde{B_n},(i)}_1}\overrightarrow{s_{n-i},s_{n-2}}(\epsilon_{n}))(w_i^{\vee})\\[5pt]
    =&(s_{z^{A_{n-1},(i)}}(-\epsilon_{n-i+1}))(w_i^{\vee})=(-\epsilon_{1})(w_i^{\vee})=-1.
    \end{array}
$$

So only two coefficients are not zero, which gives us the inequality $t^+_{i,m}\geq t^+_{i+1,m}$ for $1\leq i< m\leq n$.
For the computation of the zero coefficients and for the  inequality $t^+_{i,i}\geq t^+_{i,i+1}$ for $1\leq i<n$ we refer to the appendix.

\end{proof}
We use this Lemma together with Theorem \ref{coneproduct} to get more inequalities for the points in the string cone $\mathcal{S}_{\mathbf{i}^{B_n}}$ via a projection from $B_n$ to $A_{n-1}$ . Again, the decomposition of $\mathbf{i}^{B_n}$ in $\mathbf{i}^{A_{n-1}}\mathbf{i}^{\tilde{B}_n}$ is crucial. Here we obtain a completely analogous result as in Theorem \ref{thmBZfulfillsineq}.
\begin{thm}\label{thmBZfulfillsineqB}
For $\mathfrak{g}=\mathfrak{so}_{2n+1}$ all the points in the string cone $\mathcal{S}_{\mathbf{i}^{B_n}}$ satisfy the inequalities 
\begin{equation}\notag
\begin{aligned}
t_{i,j}^-&\geq t_{i+1,j}^- \;&\forall\; 1\leq i\leq j \leq n-1,\\
t_{i,j}^+&\geq t_{i+1,j}^+ \;&\forall\; 1\leq i\leq j \leq n    ,\\
t_{i,j}^+&\geq t_{i,j+1}^+ \;&\forall\; 1\leq i\leq  j\leq n-1.    \label{ineqCiminB}
\end{aligned}
\end{equation}

The corresponding poset looks the same as figure \ref{figD2} with the only difference that the indices of the $t^+$ variables go up to $n$ instead of $n-1$.
\end{thm}
\begin{proof}
Due to Lemma \ref{LemineqB} we obtain all the inequalities except
\begin{equation}
t^+_{i,j}\geq t^+_{i,j+1}\;\forall\; 1\leq i< j\leq n-1.\label{ineqinductionB}    
\end{equation}

\comment{Due to Theorem \ref{coneproduct} with $I_0=\{1,\dots,n-1\}$ and $I_1=\{1,\dots,n\}$ we know that the string cone $\mathcal{S}_{\mathbf{i}^{B_n}}$ is the direct product of two cones. 
The first one is the string cone for $A_{n-1}$. The second one we denote by $\mathcal{S}_{\mathbf{i}^{\tilde{B}_n}}$.
So, for a dominant weight $\lambda$ for $B_{n}$ and a basis element $v\in V_{B_{n}}(\lambda)$ we can write the string of $v$ in direction $\mathbf{i}^{B_{n}}$ as $t_v=(t_1,t_2)$ with $t_2\in \mathcal{S}_{\mathbf{i}^{\tilde{B}_{n}}}(\lambda)$ and $t_1\in \mathcal{S}_{\mathbf{i}^{A_{n-1}}}(\lambda-t_2\cdot\alpha^T_{\tilde{B}_{n}})$ (see subsection \ref{subsecbp}).\par
As $\underline{0}\in \mathcal{S}_{\mathbf{i}^{A_{n-1}}}(\lambda-t_2\cdot\alpha^T_{\tilde{B}_{n}})$, we also have $t':=(\underline{0},t_2)\in \mathcal{S}_{\mathbf{i}^{B_{n}}}(\lambda)$.
Let $v'$ be the basis element in $V_{B_{n}}(\lambda)$ with string $t'$ in direction $\mathbf{i}^{B_{n}}$.}
We define 
$$\begin{array}{lll}
&\mathbf{i}'^{\tilde{B_n}}=(n,n-1,n,n-2,n-1,n-3,n-2,n,n-1,n,n-4,n-3,n-2,n-1,\dots,\\
&2,\dots,n-2,n-1,1 ,2,\dots,n-2,n,n-1,n)
\end{array}$$
for $n$ even
and 
$$\begin{array}{lll}
&\mathbf{i}'^{\tilde{B_n}}=(n,n-1,n,n-2,n-1,n-3,n-2,n,n-1,n,n-4,n-3,n-2,n-1,\dots,\\
&2,\dots,n-2,n,n-1,n,1 ,2,\dots,n-2,n-1,n)
\end{array}$$
for $n$ odd. This only differs from $\mathbf{i}^{\tilde{B}_n}$ by the exchange of orthogonal reflections. Namely, we move every second $n$ as much to the right as possible. Let $\mathbf{i}'^{B_n}=\mathbf{i}^{A_{n-1}}\mathbf{i}'^{\tilde{B}_n}$.\par
Let $v$ be a basis element in $V_{B_n}(\lambda)$ and $t$ its string in direction $\mathbf{i}'^{B_n}$.
Now, let $\mathfrak{l}$ be the Levi subalgebra of $\mathfrak{g}$ associated to $\alpha_1,\dots,\alpha_{n-2},\alpha_{n-1}+2\alpha_n $. It is isomorphic to $\mathfrak{so}_{2n}$.
Considering $V_{B_n}(\lambda)$ as an $\mathfrak{l}$-module, we thus might compute the string of $v$ in direction $\mathbf{i}^{D_n}$.
As $E_{\alpha_{n-1}+2\alpha_n}$ acts on $v$ as $E_{\alpha_{n-1}}E_{\alpha_n}^2$, $E_{\alpha_n}E_{\alpha_{n-1}}E_{\alpha_n}$ or as $E_{\alpha_{n}}^2E_{\alpha_{n-1}}$, this string is given by $\pi_B(t)$, where $\pi_B:\R^{n^2}\rightarrow \R^{n(n-1)}$ is the projection which forgets all the coordinates $t_k$ such that $\mathbf{i}'^{B_n}_k=n$.\par
As $v$ is a basis element in the $\mathfrak{so}_{2n+1}$-module $V_{B_n}(\lambda)$, it is also a basis element if we consider this module as a $\mathfrak{l}$-module. So, due to Theorem \ref{thmStringCone} and as $\mathbf{i}'^{B_n}$ and $\mathbf{i}^{B_n}$ only differ by the exchange of orthogonal reflections, we get exactly the inequalities (\ref{ineqinductionB}) for $\pi(t)$ and thus also for $t$.  

\end{proof}
We now want to show, that the string cone $\mathcal{S}_{\mathbf{i}^{B_n}}$ is exactly given by the inequalities (\ref{ineqCiminB}).

Therefore, we use this Theorem and again Littelmann's description of the string cone to give an explicit description of the string cone $\mathcal{S}_{\mathbf{i}^{B_n}}$:
\begin{thm}\label{StringConeB}
The string cone $\mathcal{S}_{\mathbf{i}^{B_n}}$ is exactly given by the set of all points in $t\in \Z^N_{\geq 0}$ satisfying the inequalities (\ref{ineqCiminB}).
\end{thm}
The proof does not contain any new stategies compared to the $D_n$ case.
\section{String cones in type \texorpdfstring{$C_n$}{Cn} }\label{section5}
In the $C_n$ case, we get a factor $2$ in some of our inequalities, but as $B_n$ and $C_n$ share the same Weyl group, the proofs for the results in this subsection are nearly exactly the same as in the previous one, so we will leave them out and just give a short overview of the results for the sake of completeness.\par
We start again by fixing our notation: Let $\mathfrak{g}=\mathfrak{sp}_{2n}$ be the Lie algebra of type $C_n$. 
We realize it as $\mathfrak{sp}_{2n}=\{a\in\mathfrak{gl}_{2n}(\C)\mid a+Ja^TJ^{-1}=0\},$ where $B$ is the skew-symmetric non-degenerate bilinear form on $\C^{2n}$ with the matrix
$$J=\begin{pmatrix}
0&I_n\\
-I_n&0
\end{pmatrix}.
$$
The Cartan subalgebra is $\mathfrak{h}=\{\operatorname{diag}(x_1,\dots,x_n,-x_1,\dots,-x_n)\}$.
A basis of $\mathfrak{h}^*$ is given by $\{\epsilon_1,\dots,\epsilon_n\}$, where 
$$\epsilon_i:\mathrm{diag}(x_1,\dots x_n,-x_1,\dots,-x_n)\mapsto x_i.$$
We enumerate the simple roots by $$\alpha_1=\epsilon_{1}-\epsilon_2,\dots,\alpha_{n-1}=\epsilon_{n-1}-\epsilon_n,\alpha_n=2\epsilon_n.$$ 
The Weyl group is the same as for $B_n$ so we set $\mathbf{i}^{C_n}=\mathbf{i}^{B_n}.$

\begin{lem}\label{LemineqC}
For $\mathfrak{g}=\mathfrak{sp}_{2n}$ all the points in the string cone $\mathcal{S}_{\mathbf{i}^{C_n}}$ satisfy the inequalities 
\begin{equation*}
\begin{aligned}
t_{i,j}^-&\geq t_{i+1,j}^-, \;&\forall\; 1\leq i< j \leq n-1,\\
t_{i,j}^+&\geq t_{i+1,j}^+ \;&\forall\; 1\leq i< j-1 < n,\\
t_{i,i+1}^+&\geq 2t_{i+1,i+1}^+ \;&\forall\; 1\leq i< n,\\
2t_{i,i}^+&\geq t_{i,i+1}^+ \;&\forall\; 1\leq i\leq n-1.    
\end{aligned}
\end{equation*}
The corresponding poset looks the same as in figure \ref{figB1} with the only difference that we have the factor two before each variable in the last column.

\end{lem}
\begin{thm}\label{thmBZfulfillsineqC}
For $\mathfrak{g}=\mathfrak{sp}_{2n}$ all the points in the string cone $\mathcal{S}_{\mathbf{i}^{C_n}}$ satisfy the inequalities 
\begin{equation}\notag
\begin{aligned}
t_{i,j}^-&\geq t_{i+1,j}^- \;&\forall 1\leq i< j \leq n-1,\\
t_{i,j}^+&\geq t_{i+1,j}^+ \;&\forall 1\leq i< j-1< n,\\
t_{i,i+1}^+&\geq 2t_{i+1,i+1}^+ \;&\forall 1\leq i< n,    \\
t_{i,j}^+&\geq t_{i,j+1}^+ \;&\forall 1\leq i<  j\leq n-1,\\
2t_{i,i}^+&\geq t_{i,i+1}^+ \;&\forall 1\leq i< n.    \label{ineqCiminC}
\end{aligned}
\end{equation}
The corresponding poset looks the same as in figure \ref{figD2} with the only difference that the indices of the $t^+$ variables go up to $n$ instead of $n-1$ and we have a factor two before each variable in the last column.
\end{thm}

\begin{thm}\label{StringConeC}
The string cone $\mathcal{S}_{\mathbf{i}^{C_n}}$ is exactly given by the set of all points in $t\in \Z^N_{\geq 0}$ satisfying the inequalities (\ref{ineqCiminC}).
\end{thm}

\section{Branching}\label{section6}
Our choices of $\mathbf{i}$ in the three previous sections allow us to consider the string branching polytopes as mentioned in subsection \ref{subsecbp}. For example, for $D_n$ we obtain a decomposition
$$\displaystyle{ \mathcal{S}^L_{\mathbf{i}^{D_n}}(\lambda)=\Dot{\bigcup}_{t\in \mathcal{S}^L_{\mathbf{i}^{\tilde{D}_n}}(\lambda)}\mathcal{S}^L_{\mathbf{i}^{A_{n-1}}}(\lambda-t\cdot\alpha^T_{\tilde{D}})\times \{t\}.}$$
This means, that the highest weight module $V_{D_n}(\lambda)$ can be written as
$$\displaystyle{ V_{D_n}(\lambda)\cong \bigoplus_{v\in \mathcal{S}^L_{\mathbf{i}^{\tilde{D}_n}}(\lambda)}V_{A_{n-1}}(\lambda-v\cdot\alpha^T_{\tilde{D}}).}$$
A similar decomposition is possible for the Lusztig polytope:
$$\displaystyle{ \mathcal{L}^L_{\mathbf{i}^{D_n}}(\lambda)=\Dot{\bigcup}_{u\in \mathcal{L}^L_{\mathbf{i}^{\tilde{D}_n}}(\lambda)}\mathcal{L}^L_{\mathbf{i}^{A_{n-1}}}(\lambda-v\cdot \beta^T_{\tilde{D}})\times \{u\}}$$
where $\beta_{\tilde{D}}=(\beta_{\mathbf{i}^{D_n},\frac{N}{2}+1},\dots,\beta_{\mathbf{i}^{D_n},N}).$
The string branching polytopes are defined by the inequalities for the $t^+$ variables from Theorems \ref{thmBZfulfillsineq}, \ref{thmBZfulfillsineqB} and \ref{thmBZfulfillsineqC} and the additional weight inequalities which can be read off from Theorem \ref{Proppoly}.

From Theorems \ref{StringLusztig} and \ref{LusztigString} we now deduce the following Theorem describing the Lusztig branching polytopes:

\begin{thm}\label{thmbranchingD}
\begin{enumerate}
    \item The Lusztig branching polytope $\mathcal{L}_{\mathbf{i}^{\tilde{D}_n}}(\lambda)$ is defined by the inequalities 
$$\begin{array}{ccc}
      \sum\limits_{k=j}^{n-1}u_{i,k}^+\leq \sum\limits_{k=j+1}^{n-1}u_{i+1,k}^+ +\lambda_i\;\forall\; 1\leq i<j\leq n-1,\\
      \sum\limits_{k=i}^{j}u_{k,j}^+ +\sum\limits_{k=j+2}^{n-1}u_{j+1,k}^+\leq \sum\limits_{k=i+1}^{j}u_{k,j+1}^+ +\sum\limits_{k=j+2}^{n-1}u_{j+2,k}^+ +\lambda_{j+1}\;\forall\; 1\leq i\leq j< n-1,\\
      u_{n-1,n-1}^+\leq\lambda_n,\\
      u^+_{i,j}\geq 0\;\forall\; 1\leq i\leq j\leq n-1.
\end{array}$$
\item The Lusztig branching polytope $\mathcal{L}_{\mathbf{i}^{\tilde{B}_n}}(\lambda)$ is defined by the inequalities 
$$\begin{array}{ccc}
      \sum\limits_{k=j}^{n}u_{i,k}^+\leq \sum\limits_{k=j+1}^{n}u_{i+1,k}^+ +\lambda_i\;\forall\; 1\leq i<j\leq n,\\
      \sum\limits_{k=i}^{j}u_{k,j}^+ +\sum\limits_{k=j+1}^{n}u_{j,k}^+\leq \sum\limits_{k=i+1}^{j+1}u_{k,j+1}^+ +\sum\limits_{k=j+2}^{n}u_{j+1,k}^+ +\lambda_{j}\;\forall\; 1\leq i\leq j< n,\\
      u_{n,n}^+\leq\lambda_n,\\
      u^+_{i,j}\geq 0\;\forall\; 1\leq i\leq j\leq n.
\end{array}$$
\item The Lusztig branching polytope $\mathcal{L}_{\mathbf{i}^{\tilde{C}_n}}(\lambda)$ is defined by the inequalities 
$$\begin{array}{ccc}
      \sum\limits_{k=j}^{n}u_{i,k}^+\leq \sum\limits_{k=j+1}^{n}u_{i+1,k}^+ +\lambda_i\;\forall\; 1\leq i< j\leq n,\\
      \sum\limits_{k=i}^{j}u_{k,j}^+ +\sum\limits_{k=j}^{n}u_{j,k}^+\leq \sum\limits_{k=i+1}^{j+1}u_{k,j+1}^+ +\sum\limits_{k=j+1}^{n}u_{j+1,k}^+ +\lambda_{j}\;\forall\; 1\leq i\leq j< n,\\
      u_{n,n}^+\leq\lambda_n,\\
      u^+_{i,j}\geq 0\;\forall\; 1\leq i\leq j\leq n.
\end{array}$$

\end{enumerate}
\end{thm}
\begin{proof}
We only prove $(1)$ here, the proofs for the other cases follow the same idea. We use the bijections $\varphi$ and $\psi$ to translate the inequalities from the string polytope to the Lusztig polytope. For $\mathbf{i}=\mathbf{i}^{D_n}$ the ordering on the set of positive roots is the following: 
$$\varepsilon_{n-1}-\varepsilon_n,\varepsilon_{n-2}-\varepsilon_n,\varepsilon_{n-2}-\varepsilon_{n-1},\dots,\varepsilon_1-\varepsilon_n,\dots,\varepsilon_1-\varepsilon_2,\varepsilon_1+\varepsilon_2,\varepsilon_1+\varepsilon_3,\varepsilon_2+\varepsilon_3,\dots,\varepsilon_1+\varepsilon_n,\dots,\varepsilon_{n-1}+\varepsilon_n.$$
So for $1\leq i\leq j<n-1$ we get from Theorem \ref{LusztigString}:
$$\psi(u)^+_{i,j}=\sum_{k=i}^{n}\lambda_k+\sum_{k=j+1}^{n-2}\lambda_k-u^+_{i,j}-\sum_{k=i+1}^j u_{k,j}^+-\sum_{k=j+1}^{n-1}(u_{i,k}^++u_{j+1,k}^+).$$
For $1\leq i\leq j=n-1$ we get:
$$\psi(u)^+_{i,j}=\sum_{k=i}^{n-2}\lambda_k+\lambda_n+\sum_{k=j+1}^{n-2}\lambda_k-u^+_{i,j}-\sum_{k=i+1}^j u_{k,j}^+-\sum_{k=j+1}^{n-1}(u_{i,k}^++u_{j+1,k}^+).$$
So we have for $1\leq i<j\leq n-1,$
$$\begin{array}{lll}
\psi(u)_{i,j}^+\geq\psi(u)_{i+1,j}^+&\Leftrightarrow \lambda_i-u_{i,j}^+-\sum\limits_{k=j+1}^{n-1}u_{i,k}^+\geq -\sum\limits_{k=j+1}^{n-1}u_{i+1,k}^+\\
&\Leftrightarrow \lambda_i+\sum\limits_{k=j+1}^{n-1}u_{i+1,k}^+\geq \sum\limits_{k=j}^{n-1}u_{i,k}^+.
\end{array}$$
For $1\leq i\leq j< n-1$ we have
$$\begin{array}{lll}
\psi(u)_{i,j}^+\geq\psi(u)_{i,j+1}^+&\Leftrightarrow \lambda_{j+1}-u_{i,j}^+-\sum\limits_{k=i+1}^{j}u_{k,j}^+-\sum\limits_{k=j+1}^{n-1}u^+_{j+1,k}\geq -\sum\limits_{k=i+1}^{j+1}u_{k,j+1}^+-\sum\limits_{k=j+2}^{n-1}u^+_{j+2,k}\\
&\Leftrightarrow \lambda_{j+1}+\sum\limits_{k=i+1}^{j}u_{k,j+1}^++\sum\limits_{k=j+2}^{n-1}u^+_{j+2,k}\geq \sum\limits_{k=i}^{j}u_{k,j}^++\sum\limits_{k=j+2}^{n-1}u^+_{j+1,k}.
\end{array}$$
Moreover, we have 
$$\begin{array}{lll}
\psi(u)^+_{n-1,n-1}\geq 0&\Leftrightarrow \lambda_n-u_{n-1,n-1}^+\geq 0\Leftrightarrow\lambda_n\geq u_{n-1,n-1}^+.
\end{array}$$
From Theorem \ref{StringLusztig} we see that the inequalities (\ref{ineqpoly}) are equivalent to the condition $\varphi(t)_k\geq 0$ for $1\leq k\leq N.$
\end{proof}

We can of course also calculate the inequalities for the complete Lusztig polytopes:
\begin{thm}\label{thmLP}
\begin{enumerate}
    \item The Lusztig polytope $\mathcal{L}_{\mathbf{i}^{D_n}}(\lambda)$ is defined by the inequalities 
$$\begin{array}{ccc}
 \comment{\sum\limits_{k=j}^{n-1}u_{i,k}^-+\sum\limits_{k=i}^{n-1}u_{i,k}^++\sum\limits_{k=1}^{i-1}u_{k,i-1}^+\leq \lambda_i+\sum\limits_{k=j+1}^{n-1}u_{i+1,k}^-+\sum\limits_{k=i+1}^{n-1}u_{i+1,k}^++\sum\limits_{k=1}^{i}u_{k,i}^+\;\forall\; 1\leq i\leq j\leq n-1,\\}
  \sum\limits_{k=j}^{n-1}u_{i,k}^-+\sum\limits_{k=n-i}^{n-1}u_{n-i,k}^++\sum\limits_{k=1}^{n-i-1}u_{k,n-i-1}^+\leq \lambda_{n-i}+\sum\limits_{k=j+1}^{n-1}u_{i+1,k}^-+\sum\limits_{k=n-i+1}^{n-1}u_{n-i+1,k}^++\sum\limits_{k=1}^{n-i}u_{n-k,n-i}^+\\\forall\; 1\leq i\leq j\leq n-1,\\
      \sum\limits_{k=j}^{n-1}u_{i,k}^+\leq \sum\limits_{k=j+1}^{n-1}u_{i+1,k}^+ +\lambda_i\;\forall\; 1\leq i<j\leq n-1,\\
      \sum\limits_{k=i}^{j}u_{k,j}^+ +\sum\limits_{k=j+1}^{n-1}u_{j+1,k}^+\leq \sum\limits_{k=i+1}^{j+1}u_{k,j+1}^+ +\sum\limits_{k=j+2}^{n-1}u_{j+2,k}^+ +\lambda_{j+1}\;\forall\; 1\leq i\leq j< n-1,\\
      u_{n-1,n-1}^+\leq\lambda_n,\\
      u^{\pm}_{i,j}\geq 0\;\forall\; 1\leq i\leq j\leq n-1.
\end{array}$$
\item The Lusztig polytope $\mathcal{L}_{\mathbf{i}^{B_n}}(\lambda)$ is defined by the inequalities 
$$\begin{array}{ccc}
\comment{\sum\limits_{k=j}^{n-1}u_{i,k}^-+\sum\limits_{k=i}^{n}u_{i,k}^++\sum\limits_{k=1}^{i}u_{k,i}^+\leq \lambda_i+\sum\limits_{k=j+1}^{n-1}u_{i+1,k}^-+\sum\limits_{k=i+1}^{n}u_{i+1,k}^++\sum\limits_{k=1}^{i+1}u_{k,i+1}^+\;\forall\; 1\leq i\leq j\leq n-1,\\}
\sum\limits_{k=j}^{n-1}u_{i,k}^-+\sum\limits_{k=n-i+1}^{n}u_{n-i,k}^++\sum\limits_{k=1}^{n-i}u_{k,n-i}^+\leq \lambda_{n-i}+\sum\limits_{k=j+1}^{n-1}u_{i+1,k}^-+\sum\limits_{k=n-i+2}^{n}u_{n-i+1,k}^++\sum\limits_{k=1}^{n-i+1}u_{k,n-i+1}^+\\\forall\; 1\leq i\leq j\leq n-1,\\

 \sum\limits_{k=j}^{n}u_{i,k}^+\leq \sum\limits_{k=j+1}^{n}u_{i+1,k}^+ +\lambda_i\;\forall\; 1\leq i<j\leq n,\\
      \sum\limits_{k=i}^{j}u_{k,j}^+ +\sum\limits_{k=j+1}^{n}u_{j,k}^+\leq \sum\limits_{k=i+1}^{j+1}u_{k,j+1}^+ +\sum\limits_{k=j+2}^{n}u_{j+1,k}^+ +\lambda_{j}\;\forall 1\leq i\leq j< n,\\
      u_{n,n}^+\leq\lambda_n,\\
      u^{\pm}_{i,j}\geq 0\;\forall \;1\leq i\leq j\leq n.
\end{array}$$
\item The Lusztig polytope $\mathcal{L}_{\mathbf{i}^{C_n}}(\lambda)$ is defined by the inequalities 
$$\begin{array}{ccc}
\comment{\sum\limits_{k=j}^{n-1}u_{i,k}^-+\sum\limits_{k=i}^{n}u_{i,k}^++\sum\limits_{k=1}^{i}u_{k,i}^+\leq \lambda_i+\sum\limits_{k=j+1}^{n-1}u_{i+1,k}^-+\sum\limits_{k=i+1}^{n}u_{i+1,k}^++\sum\limits_{k=1}^{i+1}u_{k,i+1}^+\;\forall 1\leq i\leq j\leq n-1,\\}
\sum\limits_{k=j}^{n-1}u_{i,k}^-+\sum\limits_{k=n-i}^{n}u_{n-i,k}^++\sum\limits_{k=1}^{n-i}u_{k,n-i}^+\leq \lambda_{n-i}+\sum\limits_{k=j+1}^{n-1}u_{i+1,k}^-+\sum\limits_{k=n-i+1}^{n}u_{n-i+1,k}^++\sum\limits_{k=1}^{n-i+1}u_{k,n-i+1}^+\\\forall 1\leq i\leq j\leq n-1,\\
 \sum\limits_{k=j}^{n}u_{i,k}^+\leq \sum\limits_{k=j+1}^{n}u_{i+1,k}^+ +\lambda_i\;\forall\; 1\leq i< j\leq n,\\
      \sum\limits_{k=i}^{j}u_{k,j}^+ +\sum\limits_{k=j}^{n}u_{j,k}^+\leq \sum\limits_{k=i+1}^{j+1}u_{k,j+1}^+ +\sum\limits_{k=j+1}^{n}u_{j+1,k}^+ +\lambda_{j}\;\forall\; 1\leq i\leq j< n,\\
      u_{n,n}^+\leq\lambda_n,\\
      u^{\pm}_{i,j}\geq 0\;\forall\; 1\leq i\leq j\leq n.
\end{array}$$

\end{enumerate}
\end{thm}
\newpage
\section{Appendix}
We will here repeat the proofs of Lemmas \ref{LemIneqD} and \ref{LemineqB} and Theorems \ref{thmStringCone} and \ref{StringConeB} with all details which we omitted in the main part of this paper.
\subsection{\texorpdfstring{$D_n$}{Dn}}
\begin{proof} of Lemma \ref{LemIneqD}:
We use Theorem \ref{BZcone}:
We know that $$\omega_0=\begin{pmatrix}1&2&\dots&n-1&n\\
-1&-2&\dots&-(n-1)&\mp n
\end{pmatrix}.$$
So for $1\leq i<n-1$ we get $$s_i\omega_0=\begin{pmatrix}1&2&\dots&i-1&i&i+1&i+2&\dots&n-1&n\\
-1&-2&\dots&-(i-1)&-(i+1)&-i&-(i+2)&\dots&-(n-1)&\mp n
\end{pmatrix}.$$
Now acting by the parabolic subgroup $W_{\hat{i}}$ gives us the minimal representative of $W_{\hat{i}}s_i\omega_0$ :$$z^{(i)}=\begin{pmatrix}1&2&\dots&i-1&i&i+1&i+2&\dots&n-1&n\\
-i&-(i-1)&\dots&-2&i+1&-1&i+2&\dots&n-1&\pm n
\end{pmatrix}.$$
Here, the choice of the sign depends on the parity of $i$ (in the same manner as described above for $n$).\par
We now introduce the notation $$\overrightarrow{i,j}=i,i+1,\dots,j-1,j$$ for $i<j$.\par% and $$\overleftarrow{i,j}=i,i-1,\dots,j+1,j$$ for $i>j$.\par
A word for $z^{(i)}$ is given by $$\mathbf{i}^{(i)}=\mathbf{i}^{A_{n-1},(i)}\mathbf{i}_{1}^{\tilde{D_n},(i)}\mathbf{i}_{2}^{\tilde{D_n},(i)}$$ where
$$
\begin{array}{lll}
\mathbf{i}^{A_{n-1},(i)}&=(\overrightarrow{i,n-1},\overrightarrow{i-1,n-2},\dots,\overrightarrow{1,n-i}),\\[5pt]
\mathbf{i}_{1}^{\tilde{D_n},(i)}&=(n,n-2,n-1,n-3,n-2,n,\dots \overrightarrow{n-i,n-2},n-1/n),\\[5pt]
\mathbf{i}_{2}^{\tilde{D_n},(i)}&=(\overrightarrow{n-i-1,n-2},\dots,\overrightarrow{2,i+1},\overrightarrow{1,i-1}).    
\end{array}
$$
Here, the choice of $n-1/n$ depends on the parity of $i$.
The corresponding permutations are
$$
\begin{array}{lll}
z^{A_{n-1},(i)}&=\left(\begin{array}{lllllllll}
1&2&\dots&n-i&n-i+1&n-i+2&\dots&n-1&n\\
i+1&i+2&\dots&n&1&2&\dots&i-1&i
\end{array}\right),\\\\
z^{\tilde{D_n},(i
)}_1&=\left(\begin{array}{lllllllll}
1&2&\dots&n-i-1&n-i&n-i+1&\dots&n-1&n\\
1&2&\dots&n-i-1&-n&-(n-1)&\dots&-(n-i+1)&\pm(n-i)
\end{array}\right),\\\\
z^{\tilde{D_n},(i
)}_2&=\left(\begin{array}{llllllllll}
1&2&\dots&i-1&i&i+1&i+2&\dots&n-1&n\\
n-i&n-i+1&\dots&n-2&1&n-1&2&\dots&n-i-1&n
\end{array}\right),\\\\
z^{\tilde{D_n},(i
)}&=z^{\tilde{D_n},(i
)}_1z^{\tilde{D_n},(i
)}_2\\&=\left(\begin{array}{llllllllll}
1&2&\dots&i-1&i&i+1&i+2&\dots&n-1&n\\
-n&-(n-1)&\dots&-(n-i+1)&1&-(n-i+1)&2&\dots&n-i-1&\pm(n-i)
\end{array}\right).
\end{array}
$$

We can write $\mathbf{i}^{A_{n-1},(i)}$ as a subword of $\mathbf{i}^{A_{n-1}}$ in the following way (the underlined parts form the subword):
$$(n-1,n-2,n-1,\overrightarrow{n-3,n-1},\dots,\overrightarrow{\underline{i,n-1}},\overrightarrow{\underline{i-1,n-2}},n-1,\dots,\overrightarrow{\underline{1,n-i}},\overrightarrow{n-i+1,n-1}).$$
This means, starting from the $(n-i)$-th block, we take the first $n-i$ elements of each block.
We make the observation, that due to commutation relations for $i< m\leq n-1$ also the following words are words for $z^{\tilde{D_n},(i)}_2$:

$$
\begin{array}{lll}
\mathbf{i}_{2}^{\tilde{D_n},(i),m}=&(\overrightarrow{n-i-1,n-2},\overrightarrow{n-i-2,n-3},\dots,\overrightarrow{n-m+1,n-m+i},\overrightarrow{n-m,n-m+i-2},\\[5pt]
&\overrightarrow{n-m-1,n-m+i-3},n-m+i-1,\overrightarrow{n-m-2,n-m+i-4},n-m+i-2,\\[5pt]
&\dots,\overrightarrow{2,i},i+2,\overrightarrow{1,i-1},i+1).    
\end{array}
$$
We can write $\mathbf{i}^{\tilde{D_n},(i),m}$ as a subword of $\mathbf{i}^{\tilde{D_n}}$ in the following way:
$$
\begin{array}{lll}
&(\underline{n},\underline{n-2,n-1},\underline{n-3,n-2,n},\dots \underline{\overrightarrow{n-i,n-2},n-1/n}\underline{\overrightarrow{n-i-1,n-2}},n/n-1,\\[5pt]
&\underline{\overrightarrow{n-i-2,n-3}},n-2,n-1/n,\dots,\underline{\overrightarrow{n-m+1,n-m+i}},\overrightarrow{n-m+i+1,n-2},n/n-1,\\[5pt]
&\underline{\overrightarrow{n-m,n-m+i-2}},\overrightarrow{n-m+i-1,n-2},n/n-1,\underline{\overrightarrow{n-m-1,n-m+i-3}},\\[5pt]
&n-m+i-2,\underline{n-m+i-1},\overrightarrow{n-m+i,n-2},n/n-1,\underline{\overrightarrow{n-m-2,n-m+i-4}},\\[5pt]
&n-m+i-3,\underline{n-m+i-2},\overrightarrow{n-m+i-1,n-2},n/n-1,\\[5pt]
&\dots,\underline{\overrightarrow{2,i}},i+1,\underline{i+2},\overrightarrow{i+3,n-2},n/n-1,\underline{\overrightarrow{1,i-1}},i,\underline{i+1},\overrightarrow{i+2,n-2},n/n-1).
\end{array}
$$
This means, we take the first $i$ elements of the first $m-1$ blocks (or the whole block, if its length is less than $i$), the first $i-1$ elements from the remaining blocks and also the $(i+1)$-th element of all blocks right from the $m$-th block.\par
We now start calculating our inequalities. We just need to consider variables corresponding to entries of $\mathbf{i}^{D_n}$, which are not in the subword.\par
For the variables until the first entry of the subword, we get as coefficient
$\alpha_{i_k}w_i^{\vee}=(\epsilon_{i_k}-\epsilon_{i_{k}+1})w_i^{\vee}=0-0=0,$ 
as $i_k>i$.
We use our double indication now. In this notation we have $(\mathbf{i}^{A_{n-1}})_{k,l}=n-l-1+k$.
\par
For $i<k\leq l<n$ we get
$$
\begin{array}{lll}
&(\overrightarrow{s_i,s_{n-1}}\dots\overrightarrow{s_{n-l},s_{2n-l-i-1}}\alpha_{n-l-1+k})(w_i^{\vee})\\[5pt]
=&(\overleftarrow{s_i,s_{n-1}}\dots\overleftarrow{s_{n-l},s_{2n-l-i-1}}(\epsilon_{n-l-1+k}-\epsilon_{n-l+k}))(w_i^{\vee})\\[5pt]
=&(\epsilon_{i-l-1+k}-\epsilon_{i-l+k})(w_i^{\vee})=1-1=0,
 \end{array}
 $$

as $i-l+k\leq i$.\\\par
So all coefficients for variables corresponding to $\mathbf{i}^{A_{n-1}}$ are zero.
We now consider the variables corresponding to $\mathbf{i}^{\tilde{D_n}}$. Here we have $(\mathbf{i}^{\tilde{D_n}})_{k,l}=n-l+k-1$. For $l$ odd, we have $(\mathbf{i}^{\tilde{D_n}})_{l,l}=n$ instead of $n-1$.
The entries of the first $i$ blocks are all in the subword, so we just need to compute the coefficients for $l>i$. For $i<k\leq l<m$ we get:
$$
    \begin{array}{lll}
    &(s_{z^{A_{n-1},(i)}}s_{z^{\tilde{D_n},(i)}_1}\overrightarrow{s_{n-i-1},s_{n-2}}\dots\overrightarrow{s_{n-l},s_{n-l+i-1}}\alpha_{n-l+k-1})(w_i^{\vee})\\[5pt]
    =&(s_{z^{A_{n-1},(i)}}s_{z^{\tilde{D_n},(i)}_1}\overrightarrow{s_{n-i-1},s_{n-2}}\dots\overrightarrow{s_{n-l},s_{n-l+i-1}}(\epsilon_{n-l+k-1}-\epsilon_{n-l+k}))(w_i^{\vee})\\[5pt]
    =&(s_{z^{A_{n-1},(i)}}s_{z^{\tilde{D_n},(i)}_1}(\epsilon_{n-l+k-1-i}-\epsilon_{n-l+k-i}))(w_i^{\vee})\\[5pt]
    =&(s_{z^{A_{n-1},(i)}}(\epsilon_{n-l+k-1-i}-\epsilon_{n-l+k-i}))(w_i^{\vee}) =(\epsilon_{n-l+k-1}-\epsilon_{n-l+k})(w_i^{\vee})=0-0=0
    \end{array}
$$
as $n-l+k-1>i$. \\\par
For $k=l$ if $(\mathbf{i}^{\tilde{D_n}})_{k,l}=n$ we need to change the sign between the two $\epsilon$ but as both summands are $0$ this does not change the result.\par
For $l=m$ and $k=i$ we get 
$$
    \begin{array}{lll}
    &(s_{z^{A_{n-1},(i)}}s_{z^{\tilde{D_n},(i)}_1}\overrightarrow{s_{n-i-1},s_{n-2}}\dots\overrightarrow{s_{n-m+1},s_{n-m+i}}\overrightarrow{s_{n-m},s_{n-m+i-2}}\alpha_{n-m+i-1})(w_i^{\vee})\\[5pt]
    =&(s_{z^{A_{n-1},(i)}}s_{z^{\tilde{D_n},(i)}_1}\overrightarrow{s_{n-i-1},s_{n-2}}\dots\overrightarrow{s_{n-m+1},s_{n-m+i}}\overrightarrow{s_{n-m},s_{n-m+i-2}}(\epsilon_{n-m+i-1}-\epsilon_{n-m+i}))(w_i^{\vee})\\[5pt]
    =&(s_{z^{A_{n-1},(i)}}s_{z^{\tilde{D_n},(i)}_1}(\epsilon_{n-m}-\epsilon_{n-1}))(w_i^{\vee}) =(s_{z^{A_{n-1},(i)}}(\epsilon_{n-m}+\epsilon_{n-i+1}))(w_i^{\vee})\\[5pt]
    =&(\epsilon_{n-m+i}+\epsilon_{1})(w_i^{\vee})=0+1=1,
    \end{array}
$$
as $n-m+i>i$.\par
As $i<m$, the case $k=l$ cannot occur here.\par
For $l=m$ and $k=i+1$ we get 
$$
    \begin{array}{lll}
    &(s_{z^{A_{n-1},(i)}}s_{z^{\tilde{D_n},(i)}_1}\overrightarrow{s_{n-i-1},s_{n-2}}\dots\overrightarrow{s_{n-m+1},s_{n-m+i}}\overrightarrow{s_{n-m},s_{n-m+i-2}}\alpha_{n-m+i})(w_i^{\vee})\\[5pt]
    =&(s_{z^{A_{n-1},(i)}}s_{z^{\tilde{D_n},(i)}_1}\overrightarrow{s_{n-i-1},s_{n-2}}\dots\overrightarrow{s_{n-m+1},s_{n-m+i}}\overrightarrow{s_{n-m},s_{n-m+i-2}}(\epsilon_{n-m+i}-\epsilon_{n-m+i+1}))(w_i^{\vee})\\[5pt]
    =&(s_{z^{A_{n-1},(i)}}s_{z^{\tilde{D_n},(i)}_1}(\epsilon_{n-1}-\epsilon_{n-m+1}))(w_i^{\vee})=(s_{z^{A_{n-1},(i)}}(-\epsilon_{n-i+1}-\epsilon_{n-m+1}))(w_i^{\vee})\\[5pt]
    =&(-\epsilon_{1}-\epsilon_{n-m+1+i})(w_i^{\vee}) =-1+0=-1,
    \end{array}
$$
as $n-m+1+>i$.\par
For $k=l=m$ if $(\mathbf{i}^{\tilde{D_n}})_{k,l}=n$ we need to change the sign between the two $\epsilon$ but as the second summand is $0$ this does not change the result.\par
For $l=m$ and $i+1<k\leq m$ we get 
$$
    \begin{array}{lll}
    &(s_{z^{A_{n-1},(i)}}s_{z^{\tilde{D_n},(i)}_1}\overrightarrow{s_{n-i-1},s_{n-2}}\dots\overrightarrow{s_{n-m+1},s_{n-m+i}}\overrightarrow{s_{n-m},s_{n-m+i-2}}\alpha_{n-m+k-1})(w_i^{\vee})\\[5pt]
    =&(s_{z^{A_{n-1},(i)}}s_{z^{\tilde{D_n},(i)}_1}\overrightarrow{s_{n-i-1},s_{n-2}}\dots\overrightarrow{s_{n-m+1},s_{n-m+i}}\overrightarrow{s_{n-m},s_{n-m+i-2}}(\epsilon_{n-m+k-1}-\epsilon_{n-m+k}))(w_i^{\vee})\\[5pt]
    =&(s_{z^{A_{n-1},(i)}}s_{z^{\tilde{D_n},(i)}_1}(\epsilon_{n-m+k-i-1}-\epsilon_{n-m+k-i}))(w_i^{\vee})=(s_{z^{A_{n-1},(i)}}(\epsilon_{n-m+k-i-1}-\epsilon_{n-m+k-i}))(w_i^{\vee})\\[5pt]
    =&(\epsilon_{n-m+k-1}-\epsilon_{n-m+k})(w_i^{\vee})=0+0=0,
    \end{array}
$$
as $n-m+k-1>i$. \\\par
For $k=l=m$ if $(\mathbf{i}^{\tilde{D_n}})_{k,l}=n$ we need to change the sign between the two $\epsilon$ but as both summands are $0$ this does not change the result.\par
For $k=i$ and $m<l<n$ we get 
$$
    \begin{array}{lll}
    &(s_{z^{A_{n-1},(i)}}s_{z^{\tilde{D_n},(i)}_1}\overrightarrow{s_{n-i-1},s_{n-2}}\dots\overrightarrow{s_{n-m+1},s_{n-m+i}}\overrightarrow{s_{n-m},s_{n-m+i-2}}\\[5pt]
    &\overrightarrow{s_{n-m-1},s_{n-m+i-3}},s_{n-m+i-1}\dots\overrightarrow{s_{n-l},s_{n-l+i-2}}\alpha_{n-l+i-1})(w_i^{\vee})\\[5pt]
    =&(s_{z^{A_{n-1},(i)}}s_{z^{\tilde{D_n},(i)}_1}\overrightarrow{s_{n-i-1},s_{n-2}}\dots\overrightarrow{s_{n-m+1},s_{n-m+i}}\overrightarrow{s_{n-m},s_{n-m+i-2}}\\[5pt]
    &\overrightarrow{s_{n-m-1},s_{n-m+i-3}},s_{n-m+i-1}\dots\overrightarrow{s_{n-l},s_{n-l+i-2}}(\epsilon_{n-l+i-1}-\epsilon_{n-l+i}))(w_i^{\vee})\\[5pt]
    =&(s_{z^{A_{n-1},(i)}}s_{z^{\tilde{D_n},(i)}_1}(\epsilon_{n-l}-\epsilon_{n-l+1}))(w_i^{\vee})\\[5pt]
    =&(s_{z^{A_{n-1},(i)}}(\epsilon_{n-l}-\epsilon_{n-l+1}))(w_i^{\vee})\\[5pt]
    =&(\epsilon_{n-l+i}-\epsilon_{n-l+1+i})(w_i^{\vee})=0-0=0,
    \end{array}
$$
as $n-l+i>i$.\par For $k=l$ if $(\mathbf{i}^{\tilde{D_n}})_{k,l}=n$ we need to change the sign between the two $\epsilon$ but as both summands are $0$ this does not change the result.\par

For  $m< l<n$ and $i+2\leq k\leq l$ we get 
$$
    \begin{array}{lll}
    &(s_{z^{A_{n-1},(i)}}s_{z^{\tilde{D_n},(i)}_1}\overrightarrow{s_{n-i-1},s_{n-2}}\dots\overrightarrow{s_{n-m+1},s_{n-m+i}}\overrightarrow{s_{n-m},s_{n-m+i-2}}\\[5pt]
    &\overrightarrow{s_{n-m-1},s_{n-m+i-3}},s_{n-m+i-1}\dots\overrightarrow{s_{n-l},s_{n-l+i-2}},s_{n-l+i}\alpha_{n-l+k-1})(w_i^{\vee})\\[5pt]
    =&(s_{z^{A_{n-1},(i)}}s_{z^{\tilde{D_n},(i)}_1}\overrightarrow{s_{n-i-1},s_{n-2}}\dots\overrightarrow{s_{n-m+1},s_{n-m+i}}\overrightarrow{s_{n-m},s_{n-m+i-2}}\\[5pt]
    &\overrightarrow{s_{n-m-1},s_{n-m+i-3}},s_{n-m+i-1}\dots\overrightarrow{s_{n-l},s_{n-l+i-2}},s_{n-l+i}(\epsilon_{n-l+k-1}-\epsilon_{n-l+k}))(w_i^{\vee})\\[5pt]
    =&(s_{z^{A_{n-1},(i)}}s_{z^{\tilde{D_n},(i)}_1}(\epsilon_{n-l+k-1-i}-\epsilon_{n-l+k-i}))(w_i^{\vee})\\[5pt]
    =&(s_{z^{A_{n-1},(i)}}(\epsilon_{n-l+k-1-i}-\epsilon_{n-l+k-i}))(w_i^{\vee})\\[5pt]
    =&(\epsilon_{n-l+k-1}-\epsilon_{n-l+k})(w_i^{\vee})=0-0=0,
    \end{array}
$$
as $n-l+k-1>i$. \\\par
For $k=l$ if $(\mathbf{i}^{\tilde{D_n}})_{k,l}=n$ we need to change the sign between the two $\epsilon$ but as both summands are $0$ this does not change the result.\par
So only two coefficients are not zero, which gives us the inequality $t^+_{i,m}\geq t^ +_{i+1,m}$ for $1\leq i< m\leq n-1$.\par
Next, we consider the inequalities which we obtain from $s_{n-1}$. We have
$$s_{n-1}\omega_0=\left(\begin{array}{llllll}1&2&\dots&n-2&n-1&n\\
-1&-2&\dots&-(n-2)&-n&\mp(n-1)
\end{array}\right).$$
Now acting by the parabolic subgroup $W_{s_{\tilde{n-1}}}$ gives us the minimal representative of $W_{s_{\tilde{n-1}}}s_{n-1}\omega_0$:
$$z^{(n-1)}=\left(\begin{array}{lllllll}1&2&3&\dots&n-2&n-1&n\\
n&-(n-1)&-(n-2)&\dots&-3& 1&\mp 2
\end{array}\right).$$
A word for $z^{(n-1)}$ is given by 
$$\mathbf{i}^{(n-1)}=\mathbf{i}^{A_{n-1},(n-1)}\mathbf{i}_{1}^{\tilde{D_n},(n-1)}\mathbf{i}_{2}^{\tilde{D_n},(n-1)}$$
where
$$
\begin{array}{lll}
\mathbf{i}^{A_{n-1},(n-1)}&=(\overleftarrow{n-1,1})\\[5pt]
\mathbf{i}_{1}^{\tilde{D_n},(i)}&=(n,n-2,n-1,n-3,n-2,n,\dots \overrightarrow{3,n-2},n/n-1),\\[5pt]
\mathbf{i}_{2}^{\tilde{D_n},(i)}&=({\overrightarrow{2,n-2}}).    
\end{array}
$$
The corresponding permutations are
$$
\begin{array}{lll}
z^{A_{n-1},(n-1
)}&=\left(\begin{array}{llllll}
1&2&3&\dots&n-1&n\\
n&1&2&\dots&n-2&n-1
\end{array}\right),\\\\
z^{\tilde{D_n},(n-1
)}_1&=\left(\begin{array}{lllllll}
1&2&3&4&\dots&n-1&n\\
1&2&-n&-(n-1)&\dots&-4&\mp 3
\end{array}\right),\\\\
z^{\tilde{D_n},(n-1
)}_2&=\left(\begin{array}{lllllll}
1&2&3&\dots&n-2&n-1&n\\
1&3&4&\dots&n-1&2&n
\end{array}\right),\\\\
z^{\tilde{D_n},(n-1
)}&=z^{\tilde{D_n},(n-1)}_1z^{\tilde{D_n},(n-1)}_2\\\\
&=\left(\begin{array}{lllllll}
1&2&3&\dots&n-2&n-1&n\\
1&-n&-(n-1)&\dots&-4&2&\mp 3
\end{array}\right).
\end{array}
$$
We can write $\mathbf{i}^{A_{n-1},(n-1)}$ as a subword of $\mathbf{i}^{A_{n-1}}$ be choosing the first entry from each block of $\mathbf{i}^{A_{n-1},(n-1)}$.\par
Just as above, for $\mathbf{i}^{A_{n-1},(i)}$, we see that all coefficients for the variables with upper index $-$ are zero.\par
We can write $\mathbf{i}_{1}^{\tilde{D_n},(n-1)}$ as a subword of $\mathbf{i}^{\tilde{D_n}}$ by choosing the first $n-3$ blocks of $\mathbf{i}^{\tilde{D_n}}$.\par
Now we have $n-2$ possibilities to write $\mathbf{i}_{2}^{\tilde{D_n},(n-1)}$ as a subword of the two remaining blocks of $\mathbf{i}^{\tilde{D_n}}$: For $1\leq m\leq n-2$ we choose the first $m-1$ entries of the second last block and the entries $m+1$ to $n-2$ of the last block:
$$(\underline{\overrightarrow{2,m}},\overrightarrow{m+1,n-2},n-1/n,\overrightarrow{1,m+1},\underline{\overrightarrow{m+2,n-2}},n/n-1).$$
As all the previous entries are in the subword, we only need to compute the coefficients for some variables of the last two blocks. We fix $m$ and start with the first variable which is not in the subword: $t^+_{m,n-2}$. For $m<n-2$ we obtain:
$$
    \begin{array}{lll}
    &(s_{z^{A_{n-1},(n-1)}}s_{z^{\tilde{D_n},(n-1)}_1}\overrightarrow{s_{2},s_{m}}\alpha_{m+1})(w^{\vee}_{n-1})\\[5pt]
    =&(s_{z^{A_{n-1},(n-1)}}s_{z^{\tilde{D_n},(n-1)}_1}\overrightarrow{s_{2},s_{m}}(\epsilon_{m+1}-\epsilon_{m+2}))(w^{\vee}_{n-1})=(s_{z^{A_{n-1},(n-1)}}s_{z^{\tilde{D_n},(i)}_1}(\epsilon_{2}-\epsilon_{m+2}))(w^{\vee}_{n-1})\\[5pt]
    =&(s_{z^{A_{n-1},(n-1)}}(\epsilon_{2}+\epsilon_{n-m+1}))(w^{\vee}_{n-1})=(\epsilon_{1}+\epsilon_{n-m})(w^{\vee}_{n-1})=\frac{1}{2}+\frac{1}{2}=1
    \end{array}
$$
as $n-m<n$.\par
For $m=n-2$ we obtain:
$$
    \begin{array}{lll}
    &(s_{z^{A_{n-1},(n-1)}}s_{z^{\tilde{D_n},(n-1)}_1}\overrightarrow{s_{2},s_{n-2}}\alpha_{n-1/n})(w^{\vee}_{n-1})\\[5pt]
    =&(s_{z^{A_{n-1},(n-1)}}s_{z^{\tilde{D_n},(n-1)}_1}\overrightarrow{s_{2},s_{n-2}}(\epsilon_{n-1}\mp\epsilon_{n}))(w^{\vee}_{n-1})=(s_{z^{A_{n-1},(n-1)}}s_{z^{\tilde{D_n},(i)}_1}(\epsilon_{2}\mp\epsilon_{n}))(w^{\vee}_{n-1})\\[5pt]
    =&(s_{z^{A_{n-1},(n-1)}}(\epsilon_{2}+\epsilon_{3}))(w^{\vee}_{n-1}) =(\epsilon_{1}+\epsilon_{2})(w^{\vee}_{n-1})=\frac{1}{2}+\frac{1}{2}=1.
    \end{array}
    $$

For $m<k< n-2$ we get as coefficient for $t_{k,n-2}$:
$$
    \begin{array}{lll}
    &(s_{z^{A_{n-1},(n-1)}}s_{z^{\tilde{D_n},(n-1)}_1}\overrightarrow{s_{2},s_{m}}\alpha_{k+1})(w^{\vee}_{n-1})\\[5pt]
    =&(s_{z^{A_{n-1},(n-1)}}s_{z^{\tilde{D_n},(n-1)}_1}\overrightarrow{s_{2},s_{m}}(\epsilon_{k+1}-\epsilon_{k+2}))(w^{\vee}_{n-1})\\[5pt]
    =&(s_{z^{A_{n-1},(n-1)}}s_{z^{\tilde{D_n},(i)}_1}(\epsilon_{k+1}-\epsilon_{k+2}))(w^{\vee}_{n-1})\\[5pt]
    =&(s_{z^{A_{n-1},(n-1)}}(-\epsilon_{n-k+2}+\epsilon_{n-k+1}))(w^{\vee}_{n-1}) =(-\epsilon_{n-k+1}+\epsilon_{n-k})(w^{\vee}_{n-1})=-\frac{1}{2}+\frac{1}{2}=0
    \end{array}
$$
as $n-k+1<n-m+1\leq n$.\par
For $m<n-2$ the coefficient of $t_{n-2,n-2}$ is:
$$
    \begin{array}{lll}
    &(s_{z^{A_{n-1},(n-1)}}s_{z^{\tilde{D_n},(n-1)}_1}\overrightarrow{s_{2},s_{m}}\alpha_{n-1/n})(w^{\vee}_{n-1})\\[5pt]
    =&(s_{z^{A_{n-1},(n-1)}}s_{z^{\tilde{D_n},(n-1)}_1}\overrightarrow{s_{2},s_{m}}(\epsilon_{n-1}\mp\epsilon_{n}))(w^{\vee}_{n-1})=(s_{z^{A_{n-1},(n-1)}}s_{z^{\tilde{D_n},(i)}_1}(\epsilon_{n-1}\mp\epsilon_{n}))(w^{\vee}_{n-1})\\[5pt]
    =&(s_{z^{A_{n-1},(n-1)}}(-\epsilon_{4}+\epsilon_{3}))(w^{\vee}_{n-1}) =(-\epsilon_{3}+\epsilon_{2})(w^{\vee}_{n-1})=-\frac{1}{2}+\frac{1}{2}=0.
    \end{array}
$$
This holds true for $n>3$. For $n=3$ this case does not occur, as $1\leq m<n-2=1$ is not possible.\par
Now we compute the coefficients for the last block.\par
For $1\leq k<m$ the coefficient of $t^+_{k,n-1}$ is:
$$
    \begin{array}{lll}
    &(s_{z^{A_{n-1},(n-1)}}s_{z^{\tilde{D_n},(n-1)}_1}\overrightarrow{s_{2},s_{m}}\alpha_{k})(w^{\vee}_{n-1})=(s_{z^{A_{n-1},(n-1)}}s_{z^{\tilde{D_n},(n-1)}_1}\overrightarrow{s_{2},s_{m}}(\epsilon_{k}-\epsilon_{k+1}))(w^{\vee}_{n-1})\\[5pt]
    =&(s_{z^{A_{n-1},(n-1)}}s_{z^{\tilde{D_n},(i)}_1}(\epsilon_{k+1}-\epsilon_{k+2}))(w^{\vee}_{n-1})=(s_{z^{A_{n-1},(n-1)}}(-\epsilon_{n-k+2}+\epsilon_{n-k+1}))(w^{\vee}_{n-1})\\[5pt]
    =&(-\epsilon_{n-k+1}+\epsilon_{n-k})(w^{\vee}_{n-1})=-\frac{1}{2}+\frac{1}{2}=0
    \end{array}
$$
as $n-k+1<n$.\par
The coefficient of $t^+_{m,n-1}$ is:
$$
    \begin{array}{lll}
    &(s_{z^{A_{n-1},(n-1)}}s_{z^{\tilde{D_n},(n-1)}_1}\overrightarrow{s_{2},s_{m}}\alpha_{m})(w^{\vee}_{n-1})=(s_{z^{A_{n-1},(n-1)}}s_{z^{\tilde{D_n},(n-1)}_1}\overrightarrow{s_{2},s_{m}}(\epsilon_{m}-\epsilon_{m+1}))(w^{\vee}_{n-1})\\[5pt]
    =&(s_{z^{A_{n-1},(n-1)}}s_{z^{\tilde{D_n},(i)}_1}(\epsilon_{m+1}-\epsilon_{2}))(w^{\vee}_{n-1}) =(s_{z^{A_{n-1},(n-1)}}(-\epsilon_{n-m+2}-\epsilon_{2}))(w^{\vee}_{n-1})\\[5pt]
    =&(-\epsilon_{n-m+1}-\epsilon_{1})(w^{\vee}_{n-1})=-\frac{1}{2}-\frac{1}{2}=-1
    \end{array}
$$
as $n-m+1<n$.\par
The next entries are again in the subword, so it remains to compute the coefficient for $t^+_{n-1,n-1}$:
$$
    \begin{array}{lll}
    &(s_{z^{A_{n-1},(n-1)}}s_{z^{\tilde{D_n},(n-1)}_1}\overrightarrow{s_{2},s_{m}}\overrightarrow{s_{m+1},s_{n-2}}\alpha_{n/n-1})(w^{\vee}_{n-1})\\[5pt]
    =&(s_{z^{A_{n-1},(n-1)}}s_{z^{\tilde{D_n},(n-1)}_1}\overrightarrow{s_{2},s_{m}}\overrightarrow{s_{m+1},s_{n-2}}(\epsilon_{n-1}\pm\epsilon_n))(w^{\vee}_{n-1})\\[5pt]
    =&(s_{z^{A_{n-1},(n-1)}}s_{z^{\tilde{D_n},(n-1)}_1}(\epsilon_{2}\pm\epsilon_n))(w^{\vee}_{n-1})=(s_{z^{A_{n-1},(n-1)}}(\epsilon_{2}-\epsilon_3))(w^{\vee}_{n-1})\\[5pt]
    =&(\epsilon_{1}-\epsilon_2)(w^{\vee}_{n-1})=\frac{1}{2}-\frac{1}{2}=0.
    \end{array}
$$
So only two coefficients are not zero, which gives us the inequality $t^+_{m,n-2}\geq t^+_{m,n-1}$ for $1\leq m\leq n-2$.\par
\end{proof}

\begin{proof} of Theorem \ref{thmStringCone}:
We already know from Theorem \ref{thmBZfulfillsineq} that all the points in the string cone fulfill the inequalities (\ref{ineqCimin}). So it remains to show that each $t\in \R^{N}$ fulfilling the inequalities is in the string cone. As the string cone is rational (\cite{Li98}, Proposition 1.5), we might restrict to the case $t\in\Z^N_{\geq 0}$ and use Theorem \ref{thmrecl}:\par
Let $t\in\Z^N$ such that $t$ fulfills all the inequalities (\ref{ineqCimin}). 
We now claim the following: for $j=1,\dots,N$, $m^j$ also fulfills the inequalities (\ref{ineqCimin}). We use the convention $m^j_k=0$ for $k>j$ here.
We will prove this claim by induction on $j$, starting with $j=N$ and going down from $j$ to $j-1$.\par
As $m^N=t$, our induction hypothesis holds true for $j=N$. So we now can assume, that $m^j$ fulfills the inequalities and prove them for $m^{j-1}.$ We will use our double indication for this proof again, writing $k=(k_1,k_2)^-$ for $k\leq\frac{N}{2}$ and $k=(k_1,k_2)^+$ for $k>\frac{N}{2}$. \par
We always assume $k<j$ in the following and use the convention $m^j_{(k_1,k_2)^{\pm}}=0$ for $k_1>k_2$.\par
For $i_k\neq i_j$, which is equivalent to $\alpha_{i_k}\neq\alpha_{i_j}$, we have$\Delta^j(k)=m^j_k$ and so $$m_k^{j-1}=\min\{m_k^j,\Delta^j(k)\}=m_k^j.$$ 
Therefore, we get by our induction hypothesis:
$$
\begin{array}{lll}
m^{j-1}_{(k_1,k_2)^\pm}&=m^j_{(k_1,k_2)^\pm}\overset{I.H.}{\geq}m^j_{(k_1+1,k_2)^\pm}\\[5pt]
&\geq\min\{m^j_{(k_1+1,k_2)^\pm},\Delta^j((k_1+1,k_2)^\pm)\}=m^{j-1}_{(k_1+1,k_2)^\pm}.
\end{array}    
$$
For $k>\frac{N}{2}$ and $k_2<n-1$, we also get
$$
    \begin{array}{lll}
m^{j-1}_{(k_1,k_2)^+}&=m^j_{(k_1,k_2)^+}\overset{I. H.}{\geq}m^j_{(k_1,k_2+1)^+}\\[5pt]
&\geq\min\{m^j_{(k_1,k_2 +1)^+},\Delta^j((k_1,k_2+1)^+)\}=m^{j-1}_{(k_1,k_2+1)^+}.    
    \end{array}
$$

Now we consider the case $i_k=i_j$. \par
For $k_2<n-1$ and $k<\frac{N}{2}$ or $k_1<k_2$ we get
$$
    \begin{array}{lll}
\Delta^j((k_1,k_2)^\pm)&=\max\{\theta((k_1,k_2)^\pm,l,j)\mid k<l\leq j,\;\alpha_{i_l}=\alpha_{i_j}\}\\[5pt]
&\geq \theta((k_1,k_2)^\pm,(k_1+1,k_2+1)^\pm,j)\\[5pt]
&=m^j_{(k_1+1,k_2+1)^\pm}-\sum_{k<s\leq (k_1+1,k_2+1)^\pm}m^j_s\alpha_{i_s}(\alpha^{\vee}_{i_j})\\[5pt]
&=m^j_{(k_1+1,k_2+1)^\pm}+m^j_{(k_1+1,k_2)^\pm}+m^j_{(k_1,k_2+1)^\pm}-2m^j_{(k_1+1,k_2+1)^\pm}\\[5pt]
&=m^j_{(k_1+1,k_2)^\pm}+\underbrace{m^j_{(k_1,k_2+1)^\pm}-m^j_{(k_1+1,k_2+1)^\pm}}_{\geq 0 \;(I.H.)}\\[5pt]
&\geq m^j_{(k_1+1,k_2)^\pm}.
    \end{array}
$$
For $k>\frac{N}{2}$ we can rewrite our calculation from above to obtain 
$$
    \begin{array}{lll}
\Delta^j((k_1,k_2)^+)&=m^j_{(k_1+1,k_2)^+}+m^j_{(k_1,k_2+1)^+}-m^j_{(k_1+1,k_2+1)^+}\\[5pt]
&=m^j_{(k_1,k_2+1)^+}+\underbrace{m^j_{(k_1+1,k_2)^+}-m^j_{(k_1+1,k_2+1)^+}}_{\geq 0 \;(I.H.)}\geq m^j_{(k_1,k_2+1)^+}.
    \end{array}
$$
For $k_1=k_2<n-1$ and $k>\frac{N}{2}$ we get
$$
    \begin{array}{lll}
\Delta^j((k_2,k_2)^+)&=\max\{\theta((k_2,k_2)^+,l,j)\mid k<l\leq j,\;\alpha_{i_l}=\alpha_{i_j}\}\\[5pt]
&\geq \theta((k_2,k_2)^+,(k_2+2,k_2+2)^+,j)\\[5pt]
&=m^j_{(k_2+2,k_2+2)^+}-\sum_{k<s\leq (k_2+2,k_2+2)^+}m^j_s\alpha_{i_s}(\alpha^{\vee}_{i_j})\\[5pt]
&=m^j_{(k_2+2,k_2+2)^+}+m^j_{(k_2,k_2+1)^+}+m^j_{(k_2+1,k_2+2)^+}-2m^j_{(k_2+2,k_2+2)^+}\\[5pt]
&=m^j_{(k_2,k_2+1)^+}+\underbrace{m^j_{(k_2+1,k_2+2)^+}-m^j_{(k_2+2,k_2+2)^+}}_{\geq 0 \;(I.H.)}\\[5pt]
&\geq m^j_{(k_2,k_2+1)^+}\geq 0 (=m^j_{(k_2+1,k_2)^+}).
    \end{array}
$$

From these calculations we get
$$m^{j-1}_{(k_1,k_2)^\pm}=\min\{m^j_{(k_1,k_2)^\pm},\Delta^j((k_1,k_2)^\pm)\}\overset{I.H.}{\geq} m^j_{(k_1+1,k_2)^\pm}=m^{j-1}_{(k_1+1,k_2)^\pm}$$
as $i_{(k_1+1,k_2)^\pm}\neq i_{(k_1,k_2)^\pm}=i_j$\\\\
and
$$m^{j-1}_{(k_1,k_2)^+}=\min\{m^j_{(k_1,k_2)^+},\Delta^j((k_1,k_2)^+)\}\overset{I.H.}{\geq} m^j_{(k_1,k_2+1)^+}=m^{j-1}_{(k_1,k_2+1)^+}$$
as $i_{(k_1,k_2+1)^+}\neq i_{(k_1,k_2)^+}=i_j$.\\\par

For $k\leq \frac{N}{2}$ and $k_1>1,k_2=n-1$ we get
$$
    \begin{array}{lll}
\Delta^j((k_1,n-1)^-)&=\max\{\theta((k_1,n-1)^-,l,j)\mid k<l\leq j,\;\alpha_{i_l}=\alpha_{i_j}\}\\[5pt]
&\geq \theta((k_1,n-1)^-,(1,k_1)^+,j)=m^j_{(1,k_1)^+}-\sum_{k<s\leq (1,k_1)^+}m^j_s\alpha_{i_s}(\alpha^{\vee}_{i_j})\\[5pt]
&=m^j_{(1,k_1)^+}+m^j_{(k_1+1,n-1)^-}+m^j_{(1,k_1-1)^+}-2m^j_{(1,k_1)^+}\\[5pt]
&=m^j_{(k_1+1,n-1)^-}+\underbrace{m^j_{(1,k_1-1)^+}-m^j_{(1,k_1)^+}}_{\geq 0 \;(I.H.)}\geq m^j_{(k_1+1,n-1)^-}.
    \end{array}
$$
For $k\leq \frac{N}{2}$ and $k_1=1,k_2=n-1$ we get
$$
    \begin{array}{lll}
\Delta^j((1,n-1)^-)&=\max\{\theta((1,n-1)^-,l,j)\mid k<l\leq j,\;\alpha_{i_l}=\alpha_{i_j}\}\\[5pt]
&\geq \theta((1,n-1)^-,(2,2)^+,j)=m^j_{(2,2)^+}-\sum_{k<s\leq (2,2)^+}m^j_s\alpha_{i_s}(\alpha^{\vee}_{i_j})\\[5pt]
&=m^j_{(2,2)^+}+m^j_{(2,n-1)^-}+m^j_{(1,2)^+}-2m^j_{(2,2)^+}\\[5pt]
&=m^j_{(2,n-1)^-}+\underbrace{m^j_{(1,2)^+}-m^j_{(2,2)^+}}_{\geq 0 \;(I.H.)}\geq m^j_{(2,n-1)^-}.
    \end{array}
$$

Both cases together yield
$$m^{j-1}_{(k_1,n-1)^-}=\min\{m^j_{(k_1,n-1)^-},\Delta^j((k_1,n-1)^-)\}\overset{I.H.}{\geq} m^j_{(k_1+1,n-1)^-}=m^{j-1}_{(k_1+1,n-1)^-}$$
as $i_{(k_1+1,n-1)^-}\neq i_{(k_1,n-1)^-}=i_j$.\par

For $k>\frac{N}{2}$ and $k_2=n-1$ the case $i_k=i_j$ is not possible.\par
This finishes our claim.\par
Now, we show $$\Delta^j(k)\geq 0\;\forall 2\leq j\leq N,\;1\leq k\leq j-1.$$ We essentially already proved that in the proof of our claim above and have just to recollect the important statements here.\par
For $2\leq j\leq N$ and $1\leq k\leq j-1$ we find:\par
If $i_k\neq i_j$ 
$$\Delta^j(k)=m^j_k\geq 0,$$
as $m^j$ fulfills (\ref{ineqCimin}).\par
If $i_k=i_j$
$$\Delta^j(k)\geq m^j_{(k_1+1,k_2)^\pm}\geq 0.$$
This finishes the proof of the theorem.

\end{proof}

\subsection{\texorpdfstring{$B_n$}{Bn}}
\begin{proof} of Lemma \ref{LemineqB}:
We use Theorem \ref{BZcone}:
We know that $$\omega_0=\begin{pmatrix}1&2&\dots&n-1&n\\
-1&-2&\dots&-(n-1)&- n
\end{pmatrix}.$$
So for $1\leq i\leq n-1$, we get $$s_i\omega_0=\begin{pmatrix}1&2&\dots&i-1&i&i+1&i+2&\dots&n-1&n\\
-1&-2&\dots&-(i-1)&-(i+1)&-i&-(i+2)&\dots&-(n-1)&- n
\end{pmatrix}.$$
Now acting by the parabolic subgroup $W_{\hat{i}}$ gives us the minimal representative of $W_{\hat{i}}s_i\omega_0$ :$$z^{(i)}=\begin{pmatrix}1&2&\dots&i-1&i&i+1&i+2&\dots&n-1&n\\
-i&-(i-1)&\dots&-2&i+1&-1&i+2&\dots&n-1& n
\end{pmatrix}.$$
A word for $z^{(i)}$ is given by $$\mathbf{i}^{(i)}=\mathbf{i}^{A_{n-1},(i)}\mathbf{i}_{1}^{\tilde{B_n},(i)}\mathbf{i}_{2}^{\tilde{B_n},(i)}$$ where
$$
\begin{array}{lll}
\mathbf{i}^{A_{n-1},(i)}&=(\overrightarrow{i,n-1},\overrightarrow{i-1,n-2},\dots,\overrightarrow{1,n-i}),\\[5pt]
\mathbf{i}_{1}^{\tilde{B_n},(i)}&=(n,n-1,n,\overrightarrow{n-2,n},\dots \overrightarrow{n-i+1,n}),\\[5pt]
\mathbf{i}_{2}^{\tilde{B_n},(i)}&=(\overrightarrow{n-i,n-1},\dots,\overrightarrow{2,i+1},\overrightarrow{1,i-1}).    
\end{array}
$$
The corresponding permutations are
$$
\begin{array}{lll}
z^{A_{n-1},(i
)}&=\left(\begin{array}{lllllllll}
1&2&\dots&n-i&n-i+1&n-i+2&\dots&n-1&n\\
i+1&i+2&\dots&n&1&2&\dots&i-1&i
\end{array}\right),\\\\
z^{\tilde{B_n},(i
)}_1&=\left(\begin{array}{lllllllll}
1&2&\dots&n-i&n-i+1&n-i+2&\dots&n-1&n\\
1&2&\dots&n-i&-n&-(n-1)&\dots&-(n-i)&-(n-i+1)
\end{array}\right),\\\\
z^{\tilde{B_n},(i
)}_2&=\left(\begin{array}{llllllllll}
1&2&\dots&i-1&i&i+1&i+2&\dots&n-1&n\\
n-i+1&n-i+2&\dots&n-1&1&n&2&\dots&n-i-1&n-i
\end{array}\right),\\\\
z^{\tilde{D_n},(i
)}&=z^{\tilde{D_n},(i
)}_1z^{\tilde{D_n},(i
)}_2\\&=\left(\begin{array}{llllllllll}
1&2&\dots&i-1&i&i+1&i+2&\dots&n-1&n\\
-n&-(n-1)&\dots&-(n-i+2)&1&-(n-i+1)&2&\dots&n-i-1&n-i
\end{array}\right).
\end{array}
$$

Again, we can write $\mathbf{i}^{A_{n-1},(i)}$ as a subword of $\mathbf{i}^{A_{n-1}}$ in the following way:
$$(n-1,n-2,n-1,\overrightarrow{n-3,n-1},\dots,\overrightarrow{\underline{i,n-1}},\overrightarrow{\underline{i-1,n-2}},n-1,\dots,\overrightarrow{\underline{1,n-i}},\overrightarrow{n-i+1,n-1}).$$
We make the observation, that due to commutation relations for $i< m\leq n$ also the following words are words for $z^{\tilde{B_n},(i)}$:

$$
\begin{array}{lll}
\mathbf{i}_{2}^{\tilde{B_n},(i),m}=&(\overrightarrow{n-i,n-1},\overrightarrow{n-i-1,n-2},\dots,\overrightarrow{n-m+2,n-m+i+1},\overrightarrow{n-m+1,n-m+i-1},\\[5pt]
&\overrightarrow{n-m,n-m+i-2},n-m+i,\overrightarrow{n-m-1,n-m+i-3},n-m+i-1,\\[5pt]
&\dots,\overrightarrow{2,i},i+2,\overrightarrow{1,i-1},i+1).    
\end{array}
$$
We can write $\mathbf{i}^{\tilde{B_n},(i),m}$ as a subword of $\mathbf{i}^{\tilde{B_n}}$ in the following way:
$$
\begin{array}{lll}
&(\underline{n},\underline{n-1,n},\underline{n-2,n-1,n},\dots \underline{\overrightarrow{n-i+1,n}},\underline{\overrightarrow{n-i,n-1}},n,\\[5pt]
&\underline{\overrightarrow{n-i-1,n-2}},n-1,n,\dots,\underline{\overrightarrow{n-m+2,n-m+i+1}},\overrightarrow{n-m+i+2,n},\\[5pt]
&\underline{\overrightarrow{n-m+1,n-m+i-1}},\overrightarrow{n-m+i,n},\underline{\overrightarrow{n-m,n-m+i-2}},\\[5pt]
&n-m+i-1,\underline{n-m+i},\overrightarrow{n-m+i+1,n},\underline{\overrightarrow{n-m-1,n-m+i-3}},n-m+i-2,\\[5pt]
&\underline{n-m+i-1},\overrightarrow{n-m+i,n},\dots,\underline{\overrightarrow{2,i}},i+1,\underline{i+2},,\overrightarrow{i+3,n},\underline{\overrightarrow{1,i-1}},i,\underline{i+1},\overrightarrow{i+2,n}).
\end{array}
$$
This means, we take the first $i$ elements of the first $m-1$ blocks (or the whole block, if its length is less than $i$), the first $i-1$ elements from the remaining blocks and also the $(i+1)$-th element of all blocks right from the $m$-th block.\par
We now start calculating our inequalities. We just need to consider variables corresponding to entries of $\mathbf{i}$, which are not in the subword.\par
\comment{We use our double indication now. In this notation we have $(\mathbf{i}^{A_{n-1}})_{k,l}=n-l-1+k$
For the variables until the first entry of the subword, we get as coefficient
$$\alpha_{i_k}w_i^{\vee}=(\epsilon_{i_k}-\epsilon_{i_{k}+1})w_i^{\vee}=1-1=0,$$
as $i_k<i$.\par
For $i<k\leq l<n$ we get
$$
    \begin{array}{lll}
&(\overleftarrow{s_i,s_1}\dots\overleftarrow{s_l,s_{l-i+1}}\alpha_{l+1-k})(w_i^{\vee})\\[5pt]
=&(\overleftarrow{s_i,s_1}\dots\overleftarrow{s_l,s_{l-i+1}}(\epsilon_{l+1-k}-\epsilon_{l+2-k}))(w_i^{\vee})\\[5pt]
=&(\epsilon_{l+1+i-k}-\epsilon_{l+2+i-k})(w_i^{\vee})=0-0=0,
    \end{array}
$$
as $l+1+i-k\geq i+1$.\par}
All coefficients for variables corresponding to $\mathbf{i}^{A_{n-1}}$ are zero as in the $D_n$ case.
We now consider the variables corresponding to $\mathbf{i}^{\tilde{B_n}}$. Here we have $(\mathbf{i}^{\tilde{B_n}})_{k,l}=n-l+k$.
The entries of the first $i$ blocks are all in the subword, so we just need to compute the coefficients for $i<l$. For $i<k\leq l<m$ we get:
$$
    \begin{array}{lll}
    &(s_{z^{A_{n-1},(i)}}s_{z^{\tilde{B_n},(i)}_1}\overrightarrow{s_{n-i},s_{n-1}}\dots\overrightarrow{s_{n-l+1},s_{n-l+i}}\alpha_{n-l+k})(w_i^{\vee})\\[5pt]
    =&(s_{z^{A_{n-1},(i)}}s_{z^{\tilde{B_n},(i)}_1}\overrightarrow{s_{n-i},s_{n-1}}\dots\overrightarrow{s_{n-l+1},s_{n-l+i}}(\epsilon_{n-l+k}-\epsilon_{n-l+k+1}))(w_i^{\vee})\\[5pt]
    =&(s_{z^{A_{n-1},(i)}}s_{z^{\tilde{B_n},(i)}_1}(\epsilon_{n-l+k-i}-\epsilon_{n-l+k-i+1}))(w_i^{\vee})\\[5pt]
    =&(s_{z^{A_{n-1},(i)}}(\epsilon_{n-l+k-i}-\epsilon_{n-l+k-i+1}))(w_i^{\vee})\\[5pt]
    =&(\epsilon_{n-l+k}-\epsilon_{n-l+k+1})(w_i^{\vee})=0-0=0
    \end{array}
$$
as $n-l+k>i$. \par
For $k=l=$ if $(\mathbf{i}^{\tilde{B_n}})_{k,l}=n$ we need to delete the second $\epsilon$ but as it produces a $0$ this does not change the result.\par
For $l=m$ and $k=i$ we get 
$$
    \begin{array}{lll}
    &(s_{z^{A_{n-1},(i)}}s_{z^{\tilde{B_n},(i)}_1}\overrightarrow{s_{n-i},s_{n-1}}\dots\overrightarrow{s_{n-m+2},s_{n-m+i+1}}\overrightarrow{s_{n-m+1},s_{n-m+i-1}}\alpha_{n-m+i})(w_i^{\vee})\\[5pt]
    =&(s_{z^{A_{n-1},(i)}}s_{z^{\tilde{B_n},(i)}_1}\overrightarrow{s_{n-i},s_{n-1}}\dots\overrightarrow{s_{n-m+2},s_{n-m+i+1}}\overrightarrow{s_{n-m+1},s_{n-m+i-1}}(\epsilon_{n-m+i}-\epsilon_{n-m+i+1}))(w_i^{\vee})\\[5pt]
    =&(s_{z^{A_{n-1},(i)}}s_{z^{\tilde{B_n},(i)}_1}(\epsilon_{n-m+1}-\epsilon_{n}))(w_i^{\vee})=(s_{z^{A_{n-1},(i)}}(\epsilon_{n-m+1}+\epsilon_{n-i+1}))(w_i^{\vee})\\[5pt]
    =&(\epsilon_{n-m+1+i}+\epsilon_{1})(w_i^{\vee})=0+1=1,
    \end{array}
$$
as $n-m+1+i>i$.\par
For $l=m>i+1$ and $k=i+1$ we get 
$$
  \begin{array}{lll}
    &(s_{z^{A_{n-1},(i)}}s_{z^{\tilde{B_n},(i)}_1}\overrightarrow{s_{n-i},s_{n-1}}\dots\overrightarrow{s_{n-m+2},s_{n-m+i+1}}\overrightarrow{s_{n-m+1},s_{n-m+i-1}}\alpha_{n-m+i+1})(w_i^{\vee})\\[5pt]
    =&(s_{z^{A_{n-1},(i)}}s_{z^{\tilde{B_n},(i)}_1}\overrightarrow{s_{n-i},s_{n-1}}\dots\overrightarrow{s_{n-m+2},s_{n-m+i+1}}\overrightarrow{s_{n-m+1},s_{n-m+i-1}}(\epsilon_{n-m+i+1}-\epsilon_{n-m+i+2}))(w_i^{\vee})\\[5pt]
    =&(s_{z^{A_{n-1},(i)}}s_{z^{\tilde{B_n},(i)}_1}(\epsilon_{n}-\epsilon_{n-m+2}))(w_i^{\vee})=(s_{z^{A_{n-1},(i)}}(-\epsilon_{n-i+1}-\epsilon_{n-m+2}))(w_i^{\vee})\\[5pt]
    =&(-\epsilon_{1}-\epsilon_{n-m+2+i})(w_i^{\vee})=-1+0=-1,
    \end{array}
$$
as $n-m+2+i>i$.\par
For $l=m=i+1=k$ we get 
$$
    \begin{array}{lll}
    &(s_{z^{A_{n-1},(i)}}s_{z^{\tilde{B_n},(i)}_1}\overrightarrow{s_{n-i},s_{n-2}}\alpha_{n})(w_i^{\vee})=(s_{z^{A_{n-1},(i)}}s_{z^{\tilde{B_n},(i)}_1}\overrightarrow{s_{n-i},s_{n-2}}(\epsilon_{n}))(w_i^{\vee})\\[5pt]
    =&(s_{z^{A_{n-1},(i)}}(-\epsilon_{n-i+1}))(w_i^{\vee})=(-\epsilon_{1})(w_i^{\vee})=-1.
    \end{array}
$$

For $l=m$ and $i+1<k< m$ we get 
$$
    \begin{array}{lll}
    &(s_{z^{A_{n-1},(i)}}s_{z^{\tilde{B_n},(i)}_1}\overrightarrow{s_{n-i},s_{n-1}}\dots\overrightarrow{s_{n-m+2},s_{n-m+i+1}}\overrightarrow{s_{n-m+1},s_{n-m+i-1}}\alpha_{n-m+k})(w_i^{\vee})\\[5pt]
    =&(s_{z^{A_{n-1},(i)}}s_{z^{\tilde{B_n},(i)}_1}\overrightarrow{s_{n-i},s_{n-1}}\dots\overrightarrow{s_{n-m+2},s_{n-m+i+1}}\overrightarrow{s_{n-m+1},s_{n-m+i-1}}(\epsilon_{n-m+k}-\\[5pt]
    
    &\epsilon_{n-m+k+1}))(w_i^{\vee})=(s_{z^{A_{n-1},(i)}}s_{z^{\tilde{B_n},(i)}_1}(\epsilon_{n-m+k-i}-\epsilon_{n-m+k-i+1}))(w_i^{\vee})\\[5pt]\\
    =&(s_{z^{A_{n-1},(i)}}(\epsilon_{n-m+k-i}-\epsilon_{n-m+k-i+1}))(w_i^{\vee})=(\epsilon_{n-m+k}-\epsilon_{n-m+k+1})(w_i^{\vee})=0+0=0,
    \end{array}
$$
as $n-m+k+i>i$. \par
For $l=m$ and $i+1<k= m$ we get 
$$
    \begin{array}{lll}
    &(s_{z^{A_{n-1},(i)}}s_{z^{\tilde{B_n},(i)}_1}\overrightarrow{s_{n-i},s_{n-1}}\dots\overrightarrow{s_{n-m+2},s_{n-m+i+1}}\overrightarrow{s_{n-m+1},s_{n-m+i-1}}\alpha_{n})(w_i^{\vee})\\[5pt]
    =&(s_{z^{A_{n-1},(i)}}s_{z^{\tilde{B_n},(i)}_1}\overrightarrow{s_{n-i},s_{n-1}}\dots\overrightarrow{s_{n-m+2},s_{n-m+i+1}}\overrightarrow{s_{n-m+1},s_{n-m+i-1}}(\epsilon_{n}))(w_i^{\vee})\\[5pt]
    =&(s_{z^{A_{n-1},(i)}}s_{z^{\tilde{B_n},(i)}_1}(\epsilon_{n-i}))(w_i^{\vee})=(s_{z^{A_{n-1},(i)}}(\epsilon_{n-m+k-i}))(w_i^{\vee})=\epsilon_{n-m+k}(w_i^{\vee})=0
    \end{array}
$$
as $n-m+k>i$.

For $k=i$ and $m<l\leq n$ we get 
$$
    \begin{array}{lll}
    &(s_{z^{A_{n-1},(i)}}s_{z^{\tilde{B_n},(i)}_1}\overrightarrow{s_{n-i},s_{n-1}}\dots\overrightarrow{s_{n-m+2},s_{n-m+i+1}}\overrightarrow{s_{n-m+1},s_{n-m+i-1}}\overrightarrow{s_{n-m},s_{n-m+i-2}},s_{n-m+i}\\[5pt]
    &\dots\overrightarrow{s_{n-l+1},s_{n-l+i-1}}\alpha_{n-l+i})(w_i^{\vee})\\[5pt]
    =&(s_{z^{A_{n-1},(i)}}s_{z^{\tilde{B_n},(i)}_1}\overrightarrow{s_{n-i},s_{n-1}}\dots\overrightarrow{s_{n-m+2},s_{n-m+i+1}}\overrightarrow{s_{n-m+1},s_{n-m+i-1}}\overrightarrow{s_{n-m},s_{n-m+i-2}},s_{n-m+i}\\[5pt]
    &\dots\overrightarrow{s_{n-l+1},s_{n-l+i-1}}(\epsilon_{n-l+i}-\epsilon_{n-l+i+1}))(w_i^{\vee})=(s_{z^{A_{n-1},(i)}}s_{z^{\tilde{B_n},(i)}_1}(\epsilon_{n-l+1}-\epsilon_{n-l+2}))(w_i^{\vee})\\[5pt]
    =&(s_{z^{A_{n-1},(i)}}(\epsilon_{n-l+1}-\epsilon_{n-l+2}))(w_i^{\vee})=(\epsilon_{n-l+1+i}-\epsilon_{n-l+2+i})(w_i^{\vee})=0-0=0,
    \end{array}
$$
as $n-l+1+i>i$.\par

For  $m< l\leq n$ and $i+2\leq k< l$ we get 
$$
    \begin{array}{lll}
    &(s_{z^{A_{n-1},(i)}}s_{z^{\tilde{B_n},(i)}_1}\overrightarrow{s_{n-i},s_{n-1}}\dots\overrightarrow{s_{n-m+2},s_{n-m+i+1}}\overrightarrow{s_{n-m+1},s_{n-m+i-1}}\overrightarrow{s_{n-m},s_{n-m+i-2}},s_{n-m+i}\\[5pt]
    &\dots\overrightarrow{s_{n-l+1},s_{n-l+i-1}},s_{n-l+i+1}\alpha_{n-l+k})(w_i^{\vee})\\[5pt]
    =&(s_{z^{A_{n-1},(i)}}s_{z^{\tilde{B_n},(i)}_1}\overrightarrow{s_{n-i},s_{n-1}}\dots\overrightarrow{s_{n-m+2},s_{n-m+i+1}}\overrightarrow{s_{n-m+1},s_{n-m+i-1}}\overrightarrow{s_{n-m},s_{n-m+i-2}},s_{n-m+i}\\[5pt]
    &\dots\overrightarrow{s_{n-l+1},s_{n-l+i-1}},s_{n-l+i+1}(\epsilon_{n-l+k}-\epsilon_{n-l+k+1}))(w_i^{\vee})\\[5pt]
    =&(s_{z^{A_{n-1},(i)}}s_{z^{\tilde{B_n},(i)}_1}(\epsilon_{n-l+k-i}-\epsilon_{n-l+k-i+1}))(w_i^{\vee})\\[5pt]
    =&(s_{z^{A_{n-1},(i)}}(\epsilon_{n-l+k-i}-\epsilon_{n-l+k-i+1}))(w_i^{\vee})=(\epsilon_{n-l+k}-\epsilon_{n-l+k+1})(w_i^{\vee})=0-0=0,
    \end{array}
$$
as $n-l+k>i$. \\\par
For  $m< l\leq n$ and $i+2\leq k= l$ we get 
$$
    \begin{array}{lll}
    &(s_{z^{A_{n-1},(i)}}s_{z^{\tilde{B_n},(i)}_1}\overrightarrow{s_{n-i},s_{n-1}}\dots\overrightarrow{s_{n-m+2},s_{n-m+i+1}}\overrightarrow{s_{n-m+1},s_{n-m+i-1}}\overrightarrow{s_{n-m},s_{n-m+i-2}},s_{n-m+i}\\[5pt]
    &\dots\overrightarrow{s_{n-l+1},s_{n-l+i-1}},s_{n-l+i+1}\alpha_{n})(w_i^{\vee})\\[5pt]
    =&(s_{z^{A_{n-1},(i)}}s_{z^{\tilde{B_n},(i)}_1}\overrightarrow{s_{n-i},s_{n-1}}\dots\overrightarrow{s_{n-m+2},s_{n-m+i+1}}\overrightarrow{s_{n-m+1},s_{n-m+i-1}}\overrightarrow{s_{n-m},s_{n-m+i-2}},s_{n-m+i}\\[5pt]
    &\dots\overrightarrow{s_{n-l+1},s_{n-l+i-1}},s_{n-l+i+1}(\epsilon_{n}))(w_i^{\vee}) =(s_{z^{A_{n-1},(i)}}s_{z^{\tilde{B_n},(i)}_1}(\epsilon_{n-i}))(w_i^{\vee})\\[5pt]
    =&(s_{z^{A_{n-1},(i)}}(\epsilon_{n-i}))(w_i^{\vee})=(\epsilon_{n})(w_i^{\vee}) =0.
    \end{array}
$$

So only two coefficients are not zero, which gives us the inequality $t^+_{i,m}\geq t^+_{i+1,m}$ for $1\leq i< m\leq n$.\\\par
Now, we encounter a difference to the $D_n$ proof. Taking a look at 
$$
    \begin{array}{lll}
    \mathbf{i}_{1}^{\tilde{B_n},(i)}\mathbf{i}_{2}^{\tilde{B_n},(i),i+1}=&(n,n-1,n,\overrightarrow{n-2,n},\dots \overrightarrow{n-i+1,n})\\[5pt]
    &(\overrightarrow{n-i,n-2},\overrightarrow{n-i-1,n-3},n-1,\overrightarrow{n-i-2,n-4},n-2,\\[5pt]
&\dots,\overrightarrow{2,i},i+2,\overrightarrow{1,i-1},i+1)
    \end{array}
$$
we see that, due to the fact that in $B_n$  $s_n$ and $s_{n-2}$ do commute, we can move the $n$ from the $i$-th block to the $(i+1)$-th block and rewrite the above expression as $\mathbf{i}_{1'}^{\tilde{B_n},(i)}\mathbf{i}_{2'}^{\tilde{B_n},(i),i+1}$, where
$$
    \begin{array}{lll}
    \mathbf{i}_{1'}^{\tilde{B_n},(i)}&=(n,n-1,n,\overrightarrow{n-2,n},\dots \overrightarrow{n-i+1,n-1}),\\[5pt]
\mathbf{i}_{2'}^{\tilde{B_n},(i),i+1}=&(\overrightarrow{n-i,n-2},n,\overrightarrow{n-i-1,n-3},n-1,\overrightarrow{n-i-2,n-4},n-2,\\[5pt]
&\dots,\overrightarrow{2,i},i+2,\overrightarrow{1,i-1},i+1).
    \end{array}
$$
This does not change any coefficient of variables corresponding to blocks greater than $i+1$. In each of the blocks $i$ and $i+1$ there is only one entry not in the subword, so we only need to compute two coefficients:

For $k=l=i$ and we obtain:

$$
    \begin{array}{lll}
    &(s_{z^{A_{n-1},(i)}}s_{z^{\tilde{B_n},(i)}_{1'}}\alpha_{n})(w_i^{\vee})=(s_{z^{A_{n-1},(i)}}s_{z^{\tilde{B_n},(i)}_{1'}}(\epsilon_{n}))(w_i^{\vee})\\[5pt]
    =&(s_{z^{A_{n-1},(i)}}(\epsilon_{n-i+1}))(w_i^{\vee})=(\epsilon_{1})(w_i^{\vee})=1.
    \end{array}
$$

For $k=i$ and $l=i+1$ we obtain 

$$
    \begin{array}{lll}
    &(s_{z^{A_{n-1},(i)}}s_{z^{\tilde{B_n},(i)}_{1'}}\overrightarrow{n-i,n-2}\alpha_{n-1})(w_i^{\vee})=(s_{z^{A_{n-1},(i)}}s_{z^{\tilde{B_n},(i)}_{1'}}\overrightarrow{n-i,n-2}(\epsilon_{n-1}-\epsilon_{n})(w_i^{\vee})\\[5pt]
    =&(s_{z^{A_{n-1},(i)}}s_{z^{\tilde{B_n},(i)}_{1'}}(\epsilon_{n-i}-\epsilon_{n})(w_i^{\vee})=(s_{z^{A_{n-1},(i)}}(\epsilon_{n-i}-\epsilon_{n-i+1}))(w_i^{\vee})\\[5pt]
    =&(\epsilon_{n}-\epsilon_{1})(w_i^{\vee})=0-1=-1
    \end{array}
$$
This gives us the inequality $t^+_{i,i}\geq t^+_{i,i+1}$ for $1\leq i<n$.
\end{proof}
\begin{proof} of Theorem \ref{StringConeB}
We already know from Theorem \ref{thmBZfulfillsineqB} that all the points in the string cone fulfill the inequalities (\ref{ineqCiminB}). So it remains to show that each $t\in \R^{N}$ fulfilling the inequalities is in the string cone. As the string cone is rational (\cite{Li98}, Proposition 1.5), we might restrict to the case $t\in\Z^N_{\geq 0}$ and use Theorem \ref{thmrecl}:\par
Let $t\in\Z^N$ such that $t$ fulfills all the inequalities (\ref{ineqCiminB}). 
We now claim the following: for $j=1,\dots,N$, $m^j$ also fulfills the inequalities (\ref{ineqCiminB}). We use the convention $m^j_k=0$ for $k>j$ here.
We will prove this claim by induction on $j$, starting with $j=N$ and going down from $j$ to $j-1$.\par
As $m^N=t$, our induction hypothesis holds true for $j=N$. So we now can assume, that $m^j$ fulfills the inequalities and prove them for $m^{j-1}.$ We will also use our double indication for this proof again, writing $k=(k_1,k_2)^-$ for $k\leq\frac{n(n-1)}{2}$ and $k=(k_1,k_2)^+$ for $k>\frac{n(n-1)}{2}$. \par
We always assume $k<j$ in the following and use the convention $m^j_{(k_1,k_2)^{\pm}}=0$ for $k_1>k_2$.\par
For $i_k\neq i_j$ - which is equivalent to $\alpha_{i_k}\neq\alpha_{i_j}$ - we have$\Delta^j(k)=m^j_k$ and so $$m_k^{j-1}=\min\{m_k^j,\Delta^j(k)\}=m_k^j.$$ 
Therefore, we get by our induction hypothesis:
$$
\begin{array}{lll}
m^{j-1}_{(k_1,k_2)^\pm}&=m^j_{(k_1,k_2)^\pm}\overset{I.H.}{\geq}m^j_{(k_1+1,k_2)^\pm}\\[5pt]
&\geq\min\{m^j_{(k_1+1,k_2)^\pm},\Delta^j((k_1+1,k_2)^\pm)\}=m^{j-1}_{(k_1+1,k_2)^\pm}.
\end{array}    
$$
For $k>\frac{n(n-1)}{2}$ and $k_2<n$, we also get
$$
    \begin{array}{lll}
m^{j-1}_{(k_1,k_2)^+}&=m^j_{(k_1,k_2)^+}\overset{I. H.}{\geq}m^j_{(k_1,k_2+1)^+}\\[5pt]
&\geq\min\{m^j_{(k_1,k_2 +1)^+},\Delta^j((k_1,k_2+1)^+)\}=m^{j-1}_{(k_1,k_2+1)^+}.    
    \end{array}
$$
\par

Now we consider the case $i_k=i_j$. \par
For $k_1\leq k_2<n-1$ and $k<\frac{n(n-1)}{2}$ or $k_1\leq k_2< n,\;k_1\neq k_2-1$ and $k>\frac{n(n-1)}{2}$ we get
$$
\begin{array}{rcl}
\Delta^j((k_1,k_2)^\pm)&=&\max\{\theta((k_1,k_2)^\pm,l,j)\mid k<l\leq j,\;\alpha_{i_l}=\alpha_{i_j}\}\\[5pt]
&\geq &\theta((k_1,k_2)^\pm,(k_1+1,k_2+1)^\pm,j)\\[5pt]
&=&m^j_{(k_1+1,k_2+1)^\pm}-\sum\limits_{k<s\leq (k_1+1,k_2+1)^\pm}m^j_s\alpha_{i_s}(\alpha^{\vee}_{i_j})\\[6pt]
&=&m^j_{(k_1+1,k_2+1)^\pm}+m^j_{(k_1+1,k_2)^\pm}+m^j_{(k_1,k_2+1)^\pm}-2m^j_{(k_1+1,k_2+1)^\pm}\\[5pt]
&=&m^j_{(k_1+1,k_2)^\pm}+\underbrace{m^j_{(k_1,k_2+1)^\pm}-m^j_{(k_1+1,k_2+1)^\pm}}_{\geq 0 \;(I.H.)}\\[-5pt]
&\geq& m^j_{(k_1+1,k_2)^\pm}.
\end{array}
$$
% \begin{equation}\notag
%     \begin{aligned}
% \Delta^j((k_1,k_2)^\pm)&=\max\{\theta((k_1,k_2)^\pm,l,j)\mid k<l\leq j,\;\alpha_{i_l}=\alpha_{i_j}\}\\\\
% &\geq \theta((k_1,k_2)^\pm,(k_1+1,k_2+1)^\pm,j)\\\\
% &=m^j_{(k_1+1,k_2+1)^\pm}-\sum_{k<s\leq (k_1+1,k_2+1)^\pm}m^j_s\alpha_{i_s}(\alpha^{\vee}_{i_j})\\\\
% &=m^j_{(k_1+1,k_2+1)^\pm}+m^j_{(k_1+1,k_2)^\pm}+m^j_{(k_1,k_2+1)^\pm}-2m^j_{(k_1+1,k_2+1)^\pm}\\\\
% &=m^j_{(k_1+1,k_2)^\pm}+\underbrace{m^j_{(k_1,k_2+1)^\pm}-m^j_{(k_1+1,k_2+1)^\pm}}_{\geq 0 \;(I.H.)}\\\\
% &\geq m^j_{(k_1+1,k_2)^\pm}.
%     \end{aligned}
% \end{equation}
For $k>\frac{n(n-1)}{2}$ we can rewrite our calculation from above to obtain 
$$
\begin{array}{lll}
\Delta^j((k_1,k_2)^+)&=m^j_{(k_1+1,k_2)^+}+m^j_{(k_1,k_2+1)^+}-m^j_{(k_1+1,k_2+1)^+}\\[5pt]
&=m^j_{(k_1,k_2+1)^+}+\underbrace{m^j_{(k_1+1,k_2)^+}-m^j_{(k_1+1,k_2+1)^+}}_{\geq 0 \;(I.H.)}\\[5pt]
&\geq m^j_{(k_1,k_2+1)^+}.
\end{array}
$$
For $k_1=k_2-1<n$ and $k>\frac{n(n-1)}{2}$ we get
$$
    \begin{array}{lll}
\Delta^j((k_2-1,k_2)^+)&=\max\{\theta((k_2-1,k_2)^+,l,j)\mid k<l\leq j,\;\alpha_{i_l}=\alpha_{i_j}\}\\[5pt]
&\geq \theta((k_2-1,k_2)^+,(k_2,k_2+1)^+,j)\\[5pt]
&=m^j_{(k_2,k_2+1)^+}-\sum\limits_{k<s\leq (k_2,k_2+1)^+}m^j_s\alpha_{i_s}(\alpha^{\vee}_{i_j})\\[6pt]
&=m^j_{(k_2,k_2+1)^+}+2m^j_{(k_2,k_2)^+}+m^j_{(k_2-1,k_2+1)^+}-2m^j_{(k_2,k_2+1)^+}\\[5pt]
&=2m^j_{(k_2,k_2)^+}+\underbrace{m^j_{(k_2-1,k_2+1)^+}-m^j_{(k_2,k_2+1)^+}}_{\geq 0 \;(I.H.)}\\[5pt]
&\geq m^j_{(k_2,k_2)^+}.\\[5pt]
    \end{array}
$$
We can rewrite this calculation to obtain
$$
\begin{array}{lll}
\Delta^j((k_2-1,k_2)^+)&=m^j_{(k_2,k_2+1)^+}+2m^j_{(k_2,k_2)^+}+m^j_{(k_2-1,k_2+1)^+}-2m^j_{(k_2,k_2+1)^+}\\[5pt]
&=m^j_{(k_2-1,k_2+1)^+}+\underbrace{2m^j_{(k_2,k_2)^+}-m^j_{(k_2,k_2+1)^+}}_{\geq 0 \;(I.H.)}\\[5pt]
&\geq m^j_{(k_2-1,k_2+1)^+}.\\
\end{array}
$$

From these calculations we get
$$m^{j-1}_{(k_1,k_2)^\pm}=\min\{m^j_{(k_1,k_2)^\pm},\Delta^j((k_1,k_2)^\pm)\}\overset{I.H.}{\geq} m^j_{(k_1+1,k_2)^\pm}=m^{j-1}_{(k_1+1,k_2)^\pm}$$
as $i_{(k_1+1,k_2)^\pm}\neq i_{(k_1,k_2)^\pm}=i_j$\\\\
and
$$m^{j-1}_{(k_1,k_2)^+}=\min\{m^j_{(k_1,k_2)^+},\Delta^j((k_1,k_2)^+)\}\overset{I.H.}{\geq} m^j_{(k_1,k_2+1)^+}=m^{j-1}_{(k_1,k_2+1)^+}$$
as $i_{(k_1,k_2+1)^+}\neq i_{(k_1,k_2)^+}=i_j$.\\\par

For $k\leq \frac{n(n-1)}{2}$ and $k_1>1,k_2=n-1$ we get
$$
    \begin{array}{lll}
\Delta^j((k_1,n-1)^-)&=\max\{\theta((k_1,n-1)^-,l,j)\mid k<l\leq j,\;\alpha_{i_l}=\alpha_{i_j}\}\\[5pt]
&\geq \theta((k_1,n-1)^-,(1,k_1+1)^+,j)\\[5pt]
&=m^j_{(1,k_1+1)^+}-\sum\limits_{k<s\leq (1,k_1+1)^+}m^j_s\alpha_{i_s}(\alpha^{\vee}_{i_j})\\[6pt]
&=m^j_{(1,k_1+1)^+}+m^j_{(k_1+1,n-1)^-}+m^j_{(1,k_1)^+}-2m^j_{(1,k_1+1)^+}\\[5pt]
&=m^j_{(k_1+1,n-1)^-}+\underbrace{m^j_{(1,k_1)^+}-m^j_{(1,k_1+1)^+}}_{\geq 0 \;(I.H.)}\\[5pt]
&\geq m^j_{(k_1+1,n-1)^-}.
    
\end{array}
$$
For $k\leq \frac{n(n-1)}{2}$ and $k_1=1,k_2=n-1$ we get
$$
    \begin{array}{lll}
\Delta^j((1,n-1)^-)&=\max\{\theta((1,n-1)^-,l,j)\mid k<l\leq j,\;\alpha_{i_l}=\alpha_{i_j}\}\\[5pt]
&\geq \theta((1,n-1)^-,(1,2)^+,j)\\[5pt]
&=m^j_{(1,2)^+}-\sum\limits_{k<s\leq (1,2)^+}m^j_s\alpha_{i_s}(\alpha^{\vee}_{i_j})\\[6pt]
&=m^j_{(1,2)^+}+m^j_{(2,n-1)^-}+2m^j_{(1,1)^+}-2m^j_{(1,2)^+}\\[5pt]
&=m^j_{(2,n-1)^-}+\underbrace{2m^j_{(1,1)^+}-m^j_{(1,2)^+}}_{\geq 0 \;(I.H.)}\\[5pt]
&\geq m^j_{(2,n-1)^-}.
    \end{array}
$$

Both cases together yield
$$m^{j-1}_{(k_1,n-1)^-}=\min\{m^j_{(k_1,n-1)^-},\Delta^j((k_1,n-1)^-)\}\overset{I.H.}{\geq} m^j_{(k_1+1,n-1)^-}=m^{j-1}_{(k_1+1,n-1)^-}$$
as $i_{(k_1+1,n-1)^-}\neq i_{(k_1,n-1)^-}=i_j$.\par

For $k>\frac{n(n-1)}{2}$ and $k_2=n$ the case $i_k=i_j$ is not possible.\par
This finishes our claim.\par
Now, we show $$\Delta^j(k)\geq 0\;\forall 2\leq j\leq N,\;1\leq k\leq j-1.$$ We essentially already proved that in the proof of our claim above and have just to recollect the important statements here.\par
For $2\leq j\leq N$ and $1\leq k\leq j-1$ we find:\par
If $i_k\neq i_j$ 
$$\Delta^j(k)=m^j_k\geq 0,$$
as $m^j$ fulfills (\ref{ineqCiminB}).\par
If $i_k=i_j$
$$\Delta^j(k)\geq m^j_{(k_1+1,k_2)^\pm}\geq 0.$$
This finishes the proof of the theorem.
\end{proof}

\bibliographystyle{plain}
\bibliography{bibfile}

\end{document}